\documentclass{amsart}

\usepackage[english]{babel}

\usepackage{amstext,amsmath,amsthm,amssymb,mathrsfs,bbm}
\usepackage{float}

\usepackage{tikz-cd}

\usepackage{dsfont}
\usepackage{bm}

\usepackage{txfonts}
\usepackage[T1]{fontenc}

\usepackage{amssymb,amsmath,amsfonts,amsthm,mathrsfs}
\usepackage[colorlinks=true, citecolor=blue, anchorcolor=red]{hyperref}

\usepackage{verbatim}
\usepackage{enumerate}
\usepackage{amssymb}
\usepackage{graphicx}
\usepackage{calc}
\usepackage{latexsym}
\usepackage[ansinew]{inputenc}

\usepackage{oldgerm}

\usepackage{float}

\usepackage{bm}

\usepackage{txfonts}
\usepackage[T1]{fontenc}

\usepackage{amstext}

\usepackage[colorinlistoftodos,prependcaption,textsize=tiny]{todonotes}

\usepackage{bm}

\usepackage{txfonts}
\DeclareMathAlphabet{\mathpzc}{OT1}{pzc}{m}{it}

\newcommand{\mc}[1]{\mathcal #1}
\newcommand{\supp}{\operatorname{supp}}
\newcommand{\IB}{\mathbb{IB}}
\newcommand{\D}{\mathbb D}
\newcommand{\I}{\mathbb I}
\newcommand{\B}{\mathbb B}
\newcommand{\M}{\mathbb M}
\newcommand{\R}{\mathbb R}
\newcommand{\N}{\mathbb N}
\newcommand{\C}{\mathbb C}
\newcommand{\Z}{\mathbb Z}
\newcommand{\E}{\mathbb E}
\newcommand{\F}{\mathbb F}

\newcommand{\Prob}{\mathbb P}
\renewcommand{\Re}{\mathop{\text{\upshape{Re}}}}

\newcommand{\ms}[1]{\mathscr{#1}}

\renewcommand{\epsilon}{\varepsilon}
\renewcommand{\bar}[1]{\overline{#1}}

\newcommand{\norm}[1]{\left\lVert#1\right\rVert}

\renewcommand{\tilde}{\widetilde}

\newtheorem{theorem}{Theorem}[section]
\newtheorem{lemma}[theorem]{Lemma}
\newtheorem{corollary}[theorem]{Corollary}
\newtheorem{proposition}[theorem]{Proposition}
\theoremstyle{definition}
\newtheorem{definition}[theorem]{Definition}
\newtheorem{remark}[theorem]{Remark}
\newtheorem{example}[theorem]{Example}

\newcommand{\ra}{\ensuremath{\Rightarrow}}

\newcommand{\Longra}{\ensuremath{\Longrightarrow}}
\newcommand{\Longlra}{\ensuremath{\Longleftrightarrow}}

\newcommand{\longra}{\ensuremath{\longrightarrow}}

\newcommand{\normm}[1]{\left|\left|\left|#1\right|\right|\right|}
\newcommand{\ip}[2]{\langle #1, #2 \rangle}

\newcommand{\tr}{\text{Tr}}

\newcommand{\dom}{\mathscr{O}}

\numberwithin{equation}{section}
\renewcommand{\theequation}{\arabic{section}-\arabic{equation}}

\renewcommand{\vec}[1]{\bm{#1}}

\begin{document}

\title[]{Elliptic and Parabolic Boundary Value Problems in Weighted Function Spaces}
\author{Felix Hummel}
\address{Technical University of Munich\\ Department of Mathematics \\ Boltzmannstra{\ss}e 3\\ 85748 Garching bei M\"unchen \\ Germany}
\email{hummel@ma.tum.de}
\author{Nick Lindemulder}
\address{Institute of Analysis \\
Karlsruhe Institute of Technology \\
Englerstra\ss e 2 \\
76131 Karlsruhe\\
Germany}
\email{nick.lindemulder@kit.edu}
\address{Delft Institute of Applied Mathematics\\
Delft University of Technology \\ P.O. Box 5031\\ 2600 GA Delft\\The
Netherlands} \email{n.lindemulder@tudelft.nl}
\date{\today}

\begin{abstract}
In this paper we study elliptic and parabolic boundary value problems with inhomogeneous boundary conditions in weighted function spaces of Sobolev, Bessel potential, Besov and Triebel-Lizorkin type.
As one of the main results, we solve the problem of weighted $L_{q}$-maximal regularity in weighted Besov and Triebel-Lizorkin spaces for the parabolic case, where the spatial weight is a power weight in the Muckenhoupt $A_{\infty}$-class. In the Besov space case we have the restriction that the microscopic parameter equals to $q$. Going beyond the $A_{p}$-range, where $p$ is the integrability parameter of the Besov or Triebel-Lizorkin space under consideration, yields extra flexibility in the sharp regularity of the boundary inhomogeneities. This extra flexibility allows us to treat rougher boundary data and provides a quantitative smoothing effect on the interior of the domain.
The main ingredient is an analysis of anisotropic Poisson operators.
\end{abstract}

\maketitle

\section*{Declarations}
\subsection*{Funding} The first author thanks the Studienstiftung des deutschen Volkes for the scholarship during his doctorate and the EU for the partial support within the TiPES project funded by the European Union's Horizon 2020 research and innovation programme under grant agreement No 820970. This is TiPES contribution {\#}102. \\
The second author was supported by the Vidi subsidy 639.032.427 of the Netherlands Organisation for Scientific Research (NWO) until January 2019.
\subsection*{Conflicts of interest} Not applicable.
\subsection*{Availability of data and material} Not applicable.
\subsection*{Code availability} Not applicable.

\section{Introduction}

The idea to work in weighted function spaces equipped with temporal and/or spatial power weights of the type
\begin{equation}\label{Boutet:eq:intro:weights}
v_{\mu}(t) = t^{\mu} \:\quad (t \in J)  \quad\quad \mbox{and} \quad\quad
w^{\partial\mathscr{O}}_{\gamma}(x) = \mathrm{dist}(x,\partial\mathscr{O})^{\gamma} \:\quad (x \in \mathscr{O}),
\end{equation}
has already proven to be very useful in several situations.
In an abstract semigroup setting temporal weights were introduced by Cl\'ement $\&$ Simonett \cite{Clement&Simonett} and Pr\"uss $\&$ Simonett \cite{pruss_simonett}, in the context of maximal continous regularity and maximal $L_{p}$-regularity, respectively.
Other works on maximal temporally weighted $L_{p}$-regularity are \cite{Koehne&Pruess&Wilke,LeCrone&Pruess&Wilke} for quasilinear parabolic evolution equations and \cite{MeySchnau2} for parabolic problems with inhomogeneous boundary conditions.
Concerning the use of spatial weights in applications to (S)PDEs, we would like to mention
\cite{Alos&Bonaccorsi_stability_spde_Dirichlet_white-noise_boundary_condition,
Brewster&Mitrea_BVP_weighted_Sobolev_spaces_Lipschitz_manifolds,
Brezniak&Goldys&Peszat&Russo_2nd-order_PDE_Dirichlet_white_noise_boundary_cond,
Cioica-Licht&Kim&Lee2018,Cioica-Licht&Kim&Lee&Lindner2018,DongKimweighted15,
Fabri&Goldys_halfline_Diricjlet_boundary_control_and_noise,
Hummel_2021,KimK-H08,Krylovheat,Krylovheatpq,Lindemulder2017_PIBVP,Lindemulder2018_DSOP,Lindemulder2019_DSOE,
LV2018_Dir_Laplace,Lindemulder&Veraar_boundary_noise,Maz'ya&Shaposhnikova2005,
Mitrea&Taylor_Poisson_problem_weighted_Sobolev_spaces_Lipschitz_domains,
Sowers_mult-dim_reaction-diffusion_eqns_white_noise_boundary_perturbations}.

An important feature of the power weights \eqref{Boutet:eq:intro:weights} is that they allow to treat "rougher" behaviour in the initial time and on the boundary by increasing the parameters $\mu$ and $\gamma$, respectively. In \cite{Lindemulder2018_DSOP,Lindemulder2017_PIBVP,LV2018_Dir_Laplace,MeySchnau2,pruss_simonett} this is for instance reflected in the lower regularity of the initial/initial-boundary data that can be dealt with. In the $L_{p}$-approach to parabolic problems with Dirichlet boundary noise, where the noise is a source of roughness on the boundary, weights are even necessary to obtain function space-valued solution processes~\cite{Alos&Bonaccorsi_stability_spde_Dirichlet_white-noise_boundary_condition, Fabri&Goldys_halfline_Diricjlet_boundary_control_and_noise,Lindemulder&Veraar_boundary_noise}.

As in \cite{Lindemulder2017_PIBVP}, in this paper we exploit this feature of the power weights \eqref{Boutet:eq:intro:weights} in the study of vector-valued parabolic initial-boundary value problems of the form
\begin{equation}\label{Boutet:intro:pibvp}
\begin{array}{rllll}
\partial_{t}u(x,t) + \mathcal{A}(x,D,t)u(x,t) &= f(x,t), & x \in \mathscr{O}, & t \in J,  \\
\mathcal{B}_{j}(x',D,t)u(x',t) &= g_{j}(x',t), & x' \in \partial\mathscr{O}, & t \in J, & j=1,\ldots,m, \\
u(x,0) &= u_{0}(x), & x \in \mathscr{O}.
\end{array}
\end{equation}
Here, $J=(0,T)$ for some $T\in(0,\infty)$, $\mathscr{O} \subset \R^{n}$ is a sufficiently smooth domain with a compact boundary $\partial\mathscr{O}$ and the coefficients of the differential operator $\mathcal{A}$ and the boundary operators $\mathcal{B}_{1},\ldots,\mathcal{B}_{n}$ are $\mathcal{B}(X)$-valued, where $X$ is a UMD Banach space.
One could for instance take $X=\C^{N}$, describing a system of $N$ initial-boundary value problems.
Our structural assumptions on $\mathcal{A},\mathcal{B}_{1},\ldots,\mathcal{B}_{m}$ are an ellipticity condition and a condition of Lopatinskii-Shapiro type.
For homogeneous boundary data (i.e. $g_{j}=0$, $j=1,\ldots,m$) these problems  include linearizations of reaction-diffusion systems and of phase field models with Dirichlet, Neumann and Robin conditions. However, if one wants to use linearization techniques to treat such problems with non-linear boundary conditions, it is crucial to have a sharp theory for the fully inhomogeneous problem.

Maximal regularity provides sharp/optimal estimates for PDEs. Indeed, maximal regularity means that there is an isomorphism between the data and the solution of the problem in suitable function spaces. It is an important tool in the theory of nonlinear PDEs: having established maximal regularity for the linearized problem, the nonlinear problem can be treated with tools as the contraction principle and the implicit function theorem (see \cite{Pruess&Simonett2016_book}).

The main result of this paper is concerned with weighted $L_{q}$-maximal regularity in weighted Triebel-Lizorkin spaces for \eqref{Boutet:intro:pibvp}, where we use the weights \eqref{Boutet:eq:intro:weights}.
In order to elaborate on this, let us for reasons of exposition consider as a specific easy example of \eqref{Boutet:intro:pibvp} the heat equation with the Dirichlet boundary condition
\begin{equation}\label{Boutet:DSOP:eq:heat_eq}
\left\{\begin{array}{rll}
\partial_{t}u - \Delta u &= f &\quad\text{on}\quad J \times \mathscr{O},\\
u_{|\partial\mathscr{O}} &= g &\quad\text{on}\quad J \times \partial\mathscr{O},\\
u(0) &= u_{0} &\quad\text{on}\quad \mathscr{O},
\end{array}\right.
\end{equation}
where $J=(0,T)$ with $T \in (0,\infty)$ and where $\mathscr{O}$ is a smooth domain in $\R^{n}$ with a compact boundary $\partial\mathscr{O}$.

In order to introduce the weighted $L_{q}$-maximal regularity problem for \eqref{Boutet:DSOP:eq:heat_eq} in an abstract setting, let $q \in (1,\infty)$, $\mu \in (-1,q-1)$ and $\E \subset \mathcal{D}'(\mathscr{O})$ a Banach space of distributions on $\mathscr{O}$ such that there exists a notion of trace on the associated second order space $\E^{2} = \{ u \in  \mathcal{D}'(\mathscr{O}) : D^{\alpha}u \in \E, |\alpha| \leq 2 \}$ that is described by a bounded linear operator $\mathrm{Tr}_{\partial\mathscr{O}}:\E^{2} \longra \F$ for some suitable Banach space.

In the $L_{q,\mu}$-$\E$-maximal regularity approach to \eqref{Boutet:DSOP:eq:heat_eq} one is looking for solutions $u$ in the \emph{maximal regularity space}
\begin{equation}\label{Boutet:DSOP:eq:max-reg_space}
W^{1}_{q}(J,v_{\mu};\E) \cap L_{q}(J,v_{\mu};\E^{2}),
\end{equation}
where the boundary condition $u_{|\partial\mathscr{O}} = g$ has to be interpreted as $\mathrm{Tr}_{\partial\mathscr{O}}u=g$.
The problem \eqref{Boutet:DSOP:eq:heat_eq} is said to enjoy the property of \emph{maximal $L_{q,\mu}$-$\E$-regularity} if there exists a (necessarily unique) space of initial-boundary data $\mathscr{D}_{i.b.} \subset L_{q}(J,v_{\mu};\F) \times \E$ such that for every $f \in L_{q}(J,v_{\mu};\E)$ it holds that \eqref{Boutet:DSOP:eq:heat_eq} has a unique solution $u$ in \eqref{Boutet:DSOP:eq:max-reg_space}
if and only if $(g,u_{0}) \in \mathscr{D}_{i.b.}$.
In this situation there exists a Banach norm on $\mathscr{D}_{i.b.}$, unique up to equivalence, with
\[
\mathscr{D}_{i.b.} \hookrightarrow L_{q}(J,v_{\mu};\F) \oplus \E,
\]
which makes the associated solution operator a topological linear isomorphism between the data space $L_{q}(J,v_{\mu};\E) \oplus \mathscr{D}_{i.b.}$ and the solution space $W^{1}_{q}(J,v_{\mu};\E) \cap L_{q}(J,v_{\mu};\E^{2})$.
The \emph{maximal $L_{q,\mu}$-$\E$-regularity problem} for \eqref{Boutet:DSOP:eq:heat_eq} consists of establishing maximal $L_{q,\mu}$-$\E$-regularity
for \eqref{Boutet:DSOP:eq:heat_eq} and explicitly determining the space $\mathscr{D}_{i.b.}$.

In the special case that $\E = L_{p}(\mathscr{O},w_{\gamma}^{\partial\mathscr{O}})$, $\E^{2} = W^{2}_{p}(\mathscr{O},w_{\gamma}^{\partial\mathscr{O}})$ and $\F=L_{p}(\partial\mathscr{O})$ with $p \in (1,\infty)$ and $\gamma \in (-1,2p-1)$, $L_{q,\mu}$-$\E$-maximal regularity is referred to as $L_{q,\mu}$-$L_{p,\gamma}$-maximal regularity.

Establishing $L_{q,\mu}$-$L_{p,\gamma}$-maximal regularity with $p\neq q$ allows one to treat more nonlinearities than in the case $p=q$, as it provides more flexibility for scaling or criticality arguments (see e.g. \cite{Giga}, \cite{Koehne&Pruess&Wilke}, \cite{Pruess&Simonett&Wilke2018}, \cite{Pruess&Wilke2017}, \cite{Pruess&Wilke2018}). Such arguments have turned out to be crucial in applications to the Navier-Stokes equations, convection-diffusion equations, the Nerst-Planck-Poisson equations in electro-chemistry, chemotaxis equations and the MHD equation (see \cite{Pruess&Simonett&Wilke2018}, \cite{Pruess&Wilke2018}).

The $L_{q,\mu}$-$L_{p,\gamma}$-maximal regularity problem for \eqref{Boutet:DSOP:eq:heat_eq} has recently been solved (besides some exceptional parameter values) in \cite{LV2018_Dir_Laplace}.
Here, the boundary datum $g$ has to be in the intersection space
\begin{equation}\label{Boutet:eq:intro:space_g}
F^{\delta}_{q,p}(J,v_{\mu};L_{p}(\partial\mathscr{O})) \cap L_{q}(J,v_{\mu};B^{2\delta}_{p,p}(\partial\mathscr{O}))
\end{equation}
with $\delta = \delta_{p,\gamma} := 1-\frac{1+\gamma}{2p}$,
which in the case $q=p$ coincides with $W^{\delta}_{p}(J,v_{\mu};L_{p}(\partial\mathscr{O})) \cap L_{p}(J,v_{\mu};W^{2\delta}_{p}(\partial\mathscr{O}))$;
here $F^{s}_{q,p}$ denotes a Triebel-Lizorkin space and $W^{s}_{p}=B^{s}_{p,p}$ a non-integer order Sobolev-Slobodeckii space or Besov space.

Note that $\delta \in (0,1)$ can be taken arbitrarily close to $0$ by choosing $\gamma$  sufficiently close to $2p-1$. In \cite{Lindemulder2017_PIBVP} the maximal $L_{q,\mu}$-$L_{p,\gamma}$-regularity problem with $\gamma \in (-1,p-1)$ was solved for the more general \eqref{Boutet:intro:pibvp}, which in the special case \eqref{Boutet:DSOP:eq:heat_eq} gives the restriction $\delta \in (\frac{1}{2},1)$.

The restriction $\gamma \in (-1,p-1)$ for the spatial weight $w^{\partial\mathscr{O}}_{\gamma}$ in \cite{Lindemulder2017_PIBVP} is a restriction of harmonic analytic nature. Indeed, $(-1,p-1)$ is the Muckenhoupt $A_{p}$-range for $w^{\partial\mathscr{O}}_{\gamma}$: given $p \in (1,\infty)$ and $\gamma \in \R$, it holds that
\begin{equation}\label{Boutet:eq:intro:Ap-cond_power-weight}
w_{\gamma}^{\partial\mathscr{O}} = \mathrm{dist}(\,\cdot\,,\partial\mathscr{O})^{\gamma} \in A_{p}(\R^{n}) \quad \Longlra \quad \gamma \in (-1,p-1).
\end{equation}
The Muckenhoupt class $A_{p}(\R^{n})$ ($p \in (1,\infty)$) is a class of weights for which many harmonic analytic tools from the unweighted setting, such as Mikhlin Fourier multiplier theorems and Littlewood-Paley decompositions, remain valid for the corresponding weighted $L_{p}$-spaces.
For example, the Littlewood-Paley decomposition for $L_{p}(\R^{n},w)$ with $w \in A_{p}(\R^{n})$ and its variant for $W^{k}_{p}(\R^{n},w)$, $k \in \N$, can be formulated by means of Triebel-Lizorkin spaces as
\begin{equation}\label{Boutet:eq:intro:LP-decomp}
L_{p}(\R^{n},w)= F^{0}_{p,2}(\R^{n},w), \qquad W^{k}_{p}(\R^{n},w)= F^{k}_{p,2}(\R^{n},w).
\end{equation}
The main difficulty in \cite{LV2018_Dir_Laplace} in the non-$A_p$ setting is that these standard tools are no longer available.

One way to avoid these difficulties is to work in weighted Besov and Triebel-Lizorkin spaces instead of $\E = L_{p}(\mathscr{O},w_{\gamma}^{\partial\mathscr{O}})$. The advantage of the scales of weighted Besov and Triebel-Lizorkin spaces is the strong harmonic analytic nature of these function spaces, leading the availability of many powerful tools (see e.g.\ \cite{Bui1982,Bui1994,Bui&Paluszynski&Taibleson1996,Haroske&Piotrowska2008,Haroske&Skrzypczak2008_EntropyI,
Haroske&Skrzypczak2011_EntropyII,Haroske&Skrzypczak2011_EntropyIII,
Lindemulder2019_DSOE,Meyries&Veraar2012_sharp_embedding,
Meyries&Veraar2014_char_class_embeddings,
Meyries&Veraar2014_traces,
Sickel&Skrzypczak&Vybiral2014}). In particular, there is a Mikhlin-H\"ormander Fourier multiplier theorem.

In the special case $\E = F^{s}_{p,r}(\mathscr{O},w_{\gamma}^{\partial\mathscr{O}})$, $\E^{2} = F^{s+2}_{p,r}(\mathscr{O},w_{\gamma}^{\partial\mathscr{O}})$ and $\F=L_{p}(\partial\mathscr{O})$ with $p,r \in (1,\infty)$, $\gamma \in (-1,\infty)$ and $s \in (\frac{1+\gamma}{p}-2,\frac{1+\gamma}{p})$, $L_{q,\mu}$-$\E$-maximal regularity is referred to as $L_{q,\mu}$-$F^{s}_{p,r,\gamma}$-maximal regularity.

The $L_{q,\mu}$-$F^{s}_{p,r,\gamma}$-maximal regularity problem for \eqref{Boutet:DSOP:eq:heat_eq} has recently been solved (besides some exceptional parameter values) in \cite{Lindemulder2018_DSOP}.
Again, the boundary datum $g$ has to be in the intersection space \eqref{Boutet:eq:intro:space_g}, but now with $\delta=\delta_{p,\gamma,s} :=\frac{s}{2}+1-\frac{1+\gamma}{p}$.

As a consequence of \eqref{Boutet:eq:intro:Ap-cond_power-weight} and \eqref{Boutet:eq:intro:LP-decomp}, $L_{q,\mu}$-$F^{0}_{p,2,\gamma}$-maximal regularity coincides with $L_{q,\mu}$-$L_{p,\gamma}$-maximal regularity when $\gamma \in (-1,p-1)$. For other values of $\gamma$ the two notions are independent.
However, there still is a connection  between the $L_{q,\mu}$-$F^{s}_{p,r,\gamma}$-maximal regularity problem and the $L_{q,\mu}$-$L_{p,\gamma}$-maximal regularity problem provided by the following relaxation of \eqref{Boutet:eq:intro:LP-decomp} to an elementary embedding combined with a Sobolev embedding:
\begin{equation}\label{Boutet:eq:intro:Sob_emb+elem_emb}
F^{k+\frac{\nu-\gamma}{p}}_{p,r}(\mathscr{O},w^{\partial\mathscr{O}}_{\nu}) \hookrightarrow F^{k}_{p,1}(\mathscr{O},w^{\partial\mathscr{O}}_{\gamma}) \hookrightarrow W^{k}_{p}(\mathscr{O},w^{\partial\mathscr{O}}_{\gamma}), \qquad \nu > \gamma, r \in [1,\infty].
\end{equation}
Indeed, in view of \eqref{Boutet:eq:intro:Sob_emb+elem_emb} and the invariance
\[
\delta = \delta_{p,\nu,s} = \delta_{p,\gamma}, \qquad s=\frac{\nu-\gamma}{p},
\]
in connection with \eqref{Boutet:eq:intro:space_g}, in order to obtain a solution operator for \eqref{Boutet:DSOP:eq:heat_eq} with $f=0$, $u_{0}=0$ it suffices to treat the $L_{q,\mu}$-$F^{s}_{p,r,\gamma}$-case.

As one of the main result of this paper, we solve the $L_{q,\mu}$-$F^{s}_{p,r,\gamma}$-maximal regularity problem for \eqref{Boutet:intro:pibvp} with $\gamma \in (-1,\infty)$ and $s \in (\frac{1+\gamma}{p}+m^{*}-2m,\frac{1+\gamma}{p}+m_{*})$, where
$m=\frac{1}{2}\mathrm{ord}(\mathcal{A})$, $m^{*} = \max\{\mathrm{ord}(\mathcal{B}_{1}),\ldots,\mathrm{ord}(\mathcal{B}_{m})\}$ and $m_{*} = \min\{\mathrm{ord}(\mathcal{B}_{1}),\ldots,\mathrm{ord}(\mathcal{B}_{m})\}$.
Besides that the $L_{q,\mu}$-$F^{s}_{p,r,\gamma}$-maximal regularity problem for \eqref{Boutet:intro:pibvp} is already interesting on its own, it also contributes to the corresponding $L_{q,\mu}$-$L_{p,\gamma}$-maximal regularity problem through the above discussion, reducing that problem to the case $g_{1}=\ldots=g_{m}=0$.
The latter can be treated in an abstract operator theoretic setting, leading to the problem of determining $R$-sectoriality or even a stronger bounded $H^{\infty}$-calculus (see \cite{Pruess&Simonett2016_book}).
It would be very interesting to extend the boundedness of the $H^{\infty}$-calculus for the Dirichlet Laplacian on $L_{p}(\mathscr{O},w^{\partial\mathscr{O}}_{\gamma})$ obtained in \cite{LV2018_Dir_Laplace} to realizations of elliptic boundary value problems corresponding to \eqref{Boutet:intro:pibvp} and thereby solve the $L_{q,\mu}$-$L_{p,\gamma}$-maximal regularity problem (at least for the case of trivial initial datum $u_{0}=0$).

Whereas, given $\gamma \in (-1,p-1)$, $L_{q,\mu}$-$F^{0}_{p,2,\gamma}$-maximal regularity coincides with $L_{q,\mu}$-$L_{p,\gamma}$-maximal regularity in the scalar-valued setting (or even the Hilbert space-valued setting), they are incomparable in the general Banach space-valued setting.
However, we also provide a solution to the $L_{q,\mu}$-$H^{s}_{p,\gamma}$-maximal regularity problem for \eqref{Boutet:intro:pibvp} with $\gamma \in (-1,p-1)$ and $s \in (\frac{1+\gamma}{p}+m^{*}-2m,\frac{1+\gamma}{p}+m_{*})$, yielding $L_{q,\mu}$-$L_{p,\gamma}$-maximal regularity when $s=0$.
In the $L_{q,\mu}$-$L_{p,\gamma}$-case the proof even simplifies a bit on the function space theoretic side of the problem (see Remark~\ref{Boutet:rmk:theorem:DPIBVP;Lq-Lp}). In particular, this simplifies the previous approaches (\cite{DHP2} ($\mu=0$, $\gamma=0$), \cite{MeySchnau2} ($q=p$, $\mu \in [0,p-1)$, $\gamma=0$) and \cite{Lindemulder2017_PIBVP}).\\
We also solve the $L_{q,\mu}-B^s_{p,q,\gamma}$-maximal regularity problem with inhomogeneous boundary data for \eqref{Boutet:intro:pibvp} and parameters $\gamma\in(-1,\infty)$ and $s \in (\frac{1+\gamma}{p}+m^{*}-2m,\frac{1+\gamma}{p}+m_{*})$. This result is new even for the heat equation \eqref{Boutet:DSOP:eq:heat_eq}, where the optimal space of boundary data is given by
\begin{equation}\label{Boutet:eq:intro:space_g:Besov}
B^{\delta}_{q,q}(J,v_{\mu};L_{p}(\partial\mathscr{O})) \cap L_{q}(J,v_{\mu};B^{2\delta}_{p,q}(\partial\mathscr{O})).
\end{equation}
Here $\delta$ is again given by $\delta=\delta_{p,\gamma,s}=\frac{s}{2}+1-\frac{1+\gamma}{p}$. Note however, that we assume that the time integrability parameter and the microscopic parameter in space coincide.

The main technical ingredient in this paper is an analysis of anisotropic Poisson operators and their mapping properties on weighted mixed-norm anisotropic function spaces. The Poisson operators under consideration naturally occur as (or in) solution operators to the model problems
\begin{equation}\label{intro:pibvp;model_problem}
\begin{array}{rllll}
\partial_{t}u(x,t) + (1+\mathcal{A}(D))u(x,t) &= 0, & x \in \R^{n}_{+}, & t \in \R,  \\
\mathcal{B}_{j}(D)u(x',t) &= g_{j}(x',t), & x' \in \R^{n-1}, & t \in \R, & j=1,\ldots,m, \\
\end{array}
\end{equation}
where $\mathcal{A}(D)$ and $\mathcal{B}_{j}(D)$ are homogeneous with constant coefficients.
Moreover, they are operators $K$ of the form
\begin{equation}\label{intro:pibvp;anisotropic_Poisson}
Kg(x_{1},x',t) = (2\pi)^{-n}\int_{\R^{n-1} \times \R} e^{\imath (x',t) \cdot (\xi',\tau)}\widetilde{k}(x_{1},\xi',\tau)\hat{g}(\xi',\tau)\, d(\xi,\tau), \qquad g \in \mathcal{S}(\R^{n-1} \times \R),
\end{equation}
for some anisotropic Poisson symbol-kernel $\widetilde{k}$.

The anisotropic Poisson operator \eqref{intro:pibvp;anisotropic_Poisson} is an anisotropic $(x',t)$-independent version of the classical Poisson operator from the Boutet the Monvel calculus.
The Boutet the Monvel calculus is a pseudo-differential calculus that in some sense can be considered as a relatively small "algebra", containing the elliptic boundary value problems as well as their solution operators (or parametrices).
The calculus was introduced by, as the name already suggests, Boutet de Monvel \cite{Boutet_de_Monvel1966,Boutet_de_Monvel1971}, having its origin in the works of Vishik and Eskin \cite{Vishkin&Eskin1967}, and was further developed in e.g.\ \cite{Grubb1984_Singular_Green,Grubb1990,Grubb1996_Functional_calculus,Rempel&Schulze1982,Johnsen1996}; for an introduction to or an overview of the subject we refer the reader to \cite{Grubb1996_Functional_calculus,Grubb,Schrohe2001_Short_introduction}.

A parameter-dependent version of the Boutet de Monvel calculus has been introduced and worked out by Grubb and collaborators (see \cite{Grubb1996_Functional_calculus} and the references given therein).
This calculus contains the parameter-elliptic boundary value problems as well as their solution operators (or parametrices).
In particular, resolvent analysis can be carried out in this calculus.

In the present paper we also consider a variant of the parameter-dependent Poisson operators from \cite{Grubb1996_Functional_calculus} in the $x'$-independent setting.
Besides that this is one of the key ingredients in our treatment of the parabolic problems \eqref{Boutet:intro:pibvp} through the anisotropic Poisson operators~\eqref{intro:pibvp;anisotropic_Poisson}, it also forms the basis for our parameter-dependent estimates in weighted Besov, Triebel-Lizorkin and Bessel potential spaces for the elliptic boundary value problems
\begin{equation}\label{Boutet:intro:ebvp}
\begin{array}{rlll}
(\lambda+\mathcal{A}(x,D))u(x) &= f(x), & x \in \mathscr{O}  \\
\mathcal{B}_{j}(x',D)u(x') &= g_{j}(x'), & x' \in \partial\mathscr{O}, & j=1,\ldots,m.
\end{array}
\end{equation}
These parameter-dependent estimates are an extension of \cite{Lindemulder2019_DSOE} on second order elliptic boundary value problems subject to the Dirichlet boundary condition, which was in turn in the spirit of \cite{Denk&Seger2016,Grubb&Kokholm1993} (see Remark~\ref{Boutet:rmk:par-dep_est}).

In the latter the scales of weighted $\mathcal{B}$- and $\mathcal{F}$-spaces, the dual scales to the scales of weighted $B$- and $F$-spaces, are also included. These scales naturally appear in duality theory and can for instance be used in the study of parabolic boundary value problems with multiplicative noise at the boundary in a setting of weighted $L_{p}$-spaces, see Remark~\ref{Boutet:rmk:cor:thm:EBVP_par-dep;dual_scales}.

\subsection*{Outline.}
The outline of the paper is as follows.
\begin{itemize}
\item[$\bullet$] \emph{Section~\ref{Boutet:sec:prelim}:} Preliminaries from weighted (mixed-norm anisotropic) function spaces, distribution theory, UMD Banach spaces and $L_{q}$-maximal regularity, differential boundary value systems.
\item[$\bullet$] \emph{Section~\ref{Boutet:sec:emb}:} Sobolev embedding and trace results for mixed-norm anisotropic function spaces.
\item[$\bullet$] \emph{Section~\ref{Boutet:sec:Poisson}:} Introduction and basic properties of Poisson operators, solution operators to model problems and mapping properties.
\item[$\bullet$] \emph{Section~\ref{Boutet:sec:parabolic}:} $L_{q,\mu}$-maximal regularity for the parabolic boundary value problem \eqref{Boutet:intro:pibvp}.
\item[$\bullet$] \emph{Section~\ref{Boutet:sec:elliptic}:} Parameter-dependent estimates for the elliptic boundary value problem \eqref{Boutet:intro:ebvp}.
\end{itemize}

\subsection*{Notation and convention.}
By $\N$ we denote the natural numbers including $0$, i.e. $\N:=\{0,1,2,\ldots \}$. For the natural numbers starting from $1$ we write $\N_{1}$, i.e. $\N_1=\{1,2,3,\ldots\}$. We write $\Sigma_{\phi} = \{z\in \C \setminus \{0\}: |\arg(z)|< \phi\}$ for the open sector in the complex plane with opening angle $2\phi$.
For $a \in \R$, we put $a_{+} = 0 \vee a = \max\{0,a\}$ and $a_{-}=-(-a)_{+} = a \wedge 0 = \min\{a,0\}$.

In the whole paper let
\[
\sigma_{s_{0},s_{1},p,\gamma} :=
\max\left\{\left(\frac{1+\gamma}{p}-1\right)_{+}-s_{0},\left(\frac{1+\gamma}{p}\right)_{-}+s_{1},s_{1}-s_{0}\right\}
\]
and
\[
\sigma_{s,p,\gamma} := \sigma_{s,s,p,\gamma} =
\max\left\{\left(\frac{1+\gamma}{p}-1\right)_{+}-s,\left(\frac{1+\gamma}{p}\right)_{-}+s,0\right\}.
\]
Note that $\sigma_{s,p,\gamma}=|s|$ if $\gamma\in[-1,p-1]$.\\
Let $X$ be a Banach space and $(S,\mathscr{A},\mu)$ a measure space. Throughout the paper we write $L_0(S;X)$ for the space of all equivalence classes of strongly measurable functions $f\colon S\to X$, where as usual the equivalence relation is the one of functions that coincide almost everywhere.


\section{Preliminaries}\label{Boutet:sec:prelim}

\subsection{Weighted Lebesgue Spaces}\label{PIBVP:subsec:sec:prelim;mixed-norm}

A reference for the general theory of Muckenhoupt weights is \cite[Chapter~9]{Grafakos_modern}.

A \emph{weight} on a measure space $(S,\mathscr{A},\mu)$ is a measurable function $w:S \longra [0,\infty]$ that takes it values almost everywhere in $(0,\infty)$.
We denote by $\mathcal{W}(S)$ the sets of all weights on $(S,\mathscr{A},\mu)$.
For $w \in \mathcal{W}(S)$ and $p \in [1,\infty)$ we denote by $L_{p}(S,w)$ the space of all $f\in L_0(S;\C)$ with
\[
\norm{f}_{L^{p}(S,w)} := \left( \int_{S}|f(x)|^{p}w(x)\,d\mu(x) \right)^{1/p} < \infty.
\]
If $p \in (1,\infty)$, then $w'=w'_{p} := w^{-\frac{1}{p-1}}$ is also a weight on $S$, called the $p$-dual weight of $w$.
Furthermore, for $p \in (1,\infty)$ we have $[L_{p}(S,w)]^{*} = L_{p'}(S,w')$
isometrically with respect to the pairing
\begin{equation}\label{DSOP:eq:subsec:prelim:weights;pairing}
L_{p}(S,w) \times L_{p'}(S,w') \longra \C,\, (f,g) \mapsto \int_{S}fg\,d\mu.
\end{equation}

Supppose $(S,\mathscr{A},\mu) = \bigotimes_{j=1}^{l}(S_{j},\mathscr{A}_{j},\mu_{j})$ is a product measure space. For $\vec{p} \in [1,\infty)^{l}$ and $\vec{w} \in \prod_{j=1}^{l}\mathcal{W}(S_{j})$ we denote by $L_{\vec{p}}(S_{1}\times\ldots\times S_{l},\vec{w})$ the mixed-norm space
\[
L_{\vec{p}}(S,\vec{w}) := L_{p_{l}}(S_{l},w_{l})[\ldots[L_{p_{1}}(S_{1},w_{1})]\ldots],
\]
that is, $L_{\vec{p}}(S,\vec{w})$ is the space of all $f \in L_{0}(S)$ with
\[
\norm{f}_{L_{\vec{p}}(S,\vec{w})} :=
 \left( \int_{S_{l}} \ldots \left(\int_{S_{1}}|f(x)|^{p_{1}}w_{1}(x_{1})d\mu_{1}(x_{1}) \right)^{p_{2}/p_{1}} \ldots w_{l}(x_{l})d\mu_{l}(x_{l}) \right)^{1/p_{l}} < \infty.
\]
We equip $L_{\vec{p}}(S,\vec{w})$ with the norm $\norm{\,\cdot\,}_{L_{\vec{p}}(S,\vec{w})}$, which turns it into a Banach space.
As an extension (and in fact consequence) of \eqref{DSOP:eq:subsec:prelim:weights;pairing}, for $\vec{p} \in (1,\infty)$ we have $[L_{\vec{p}}(S,\vec{w})]^{*} = L_{\vec{p}'}(S,\vec{w}'_{\vec{p}})$
isometrically with respect to the pairing
\begin{equation}\label{DSOP:eq:subsec:prelim:weights;pairing;mixed-norm}
L_{\vec{p}}(S,\vec{w}) \times L_{\vec{p}'}(S,\vec{w}'_{\vec{p}}) \longra \C,\, (f,g) \mapsto \int_{S}fg\,d\mu,
\end{equation}
where $\vec{p}'=(p_{1}',\ldots,p_{l}')$ and $\vec{w}'_{\vec{p}}=(w'_{p_{1}},\ldots,w'_{p_{l}})$.

Given a Banach space $X$, we denote by $L_{\vec{p}}(S,\vec{w};X)$ the associated Bochner space
\[
L_{\vec{p}}(S,\vec{w};X) := L_{\vec{p}}(S,\vec{w})(X) = \{ f \in L_{0}(\R^{n};X) : \norm{f}_{X} \in L_{\vec{p}}(S,\vec{w}) \}.
\]

For $p \in (1,\infty)$ we denote by $A_{p}=A_{p}(\R^{n})$ the class of all Muckenhoupt $A_{p}$-weights, which are all the locally integrable weights for which the $A_{p}$-characteristic $$[w]_{A_{p}}:=\sup_{Q\text{ cube in }\R^n} \left(\frac{1}{|Q|}\int_{Q} w(x)\,dx\right)\left(\frac{1}{|Q|}\int_{Q} w'_p(x)\,dx\right) \in [1,\infty]$$ is finite.
We furthermore set $A_{\infty} := \bigcup_{p \in (1,\infty)}A_{p}$.

For $p \in (1,\infty)$ we denote by $A^{\mathrm{rec}}_{p}=A^{\mathrm{rec}}_{p}(\R^{n})$ the class of all rectangular Muckenhoupt $A_{p}$-weights, which are all the locally integrable weights for which the $A^{\mathrm{rec}}_{p}$-characteristic $[w]_{A^{\mathrm{rec}}_{p}} \in [1,\infty]$ is finite.
Here $[w]_{A^{\mathrm{rec}}_{p}}$ is defined as $[w]_{A_{p}}$ by replacing cubes with sides parallel to the coordinate axes by rectangles with sides parallel to the coordinate axes in the definition.

The relevant weights for this paper are the power weights of the form $w=\mathrm{dist}(\,\cdot\,,\partial\mathscr{O})^{\gamma}$, where $\mathscr{O}$ is a $C^{\infty}$-domain in $\R^{n}$ and where $\gamma \in (-1,\infty)$.
If $\mathscr{O} \subset \R^{n}$ is a Lipschitz domain and $\gamma \in \R$, $p \in (1,\infty)$, then (see \cite[Lemma~2.3]{Farwig&Sohr_Weighted_Lq-theory_Stokes_resolvent_exterior_domains} or \cite[Lemma~2.3]{Mitrea&Taylor_Poisson_problem_weighted_Sobolev_spaces_Lipschitz_domains})
\begin{equation}\label{DBVP:eq:sec:prelim:power_weight_Ap}
w_{\gamma}^{\mathscr{O}} := \mathrm{dist}(\,\cdot\,,\partial\mathscr{O})^{\gamma} \in A_{p} \quad \Longlra \quad \gamma \in (-1,p-1);
\end{equation}
in particular,
\begin{equation}\label{DBVP:eq:sec:prelim:power_weight_A_infty}
w_{\gamma}^{\mathscr{O}} = \mathrm{dist}(\,\cdot\,,\partial\mathscr{O})^{\gamma} \in A_{\infty} \quad \Longlra \quad \gamma \in (-1,\infty).
\end{equation}
For the important model problem case $\mathscr{O} = \R^{n}_{+}$ we simply write $w_{\gamma}:= w_{\gamma}^{\R^{n}_{+}} = \mathrm{dist}(\,\cdot\,,\partial\R^{n}_{+})^{\gamma}$.

Furthermore, in connection with the pairing \eqref{DSOP:eq:subsec:prelim:weights;pairing}, for $p \in (1,\infty)$ we have
\[
w \in A_{p} \quad \Longlra \quad w' \in A_{p'} \quad \Longlra \quad w,w' \in A_{\infty}.
\]

Let $p \in (1,\infty)$. We define $[A_{\infty}]'_{p} = [A_{\infty}]'_{p}(\R^{n})$ as the set of all weights $w$ on $\R^{n}$ for which $w'_{p} = w^{-\frac{1}{p-1}} \in A_{\infty}$. If $\mathscr{O} \subset \R^{n}$ is a Lipschitz domain and $\gamma \in \R$, $p \in (1,\infty)$, then
\begin{equation}
w_{\gamma}^{\mathscr{O}}  \in [A_{\infty}]'_{p} \quad \Longlra \quad \gamma'_{p} := -\frac{\gamma}{p-1} \in (-1,\infty) \quad \Longlra \quad \gamma \in (-\infty,p-1)
\end{equation}
in view of \eqref{DBVP:eq:sec:prelim:power_weight_A_infty}.

\subsection{UMD Spaces and $L_{q}$-maximal Regularity}\label{Boutet:subsec:prelim:UMD}

The general references for this subsection are \cite{Hytonen&Neerven&Veraar&Weis2016_Analyis_in_Banach_Spaces_I,
Hytonen&Neerven&Veraar&Weis2016_Analyis_in_Banach_Spaces_II,Kunstmann&Weis2004_lecture_notes}.

The UMD property of Banach spaces is defined through the unconditionality of martingale differences, which is a primarily probabilistic notion.
A deep result due to Bourgain and Burkholder gives a pure analytic characterization in terms of the Hilbert transform: a Banach space $X$ has the UMD property if and only if it is of class HT, i.e.\ the Hilbert transform $H$ has a bounded extension $H_{X}$ to $L_{p}(\R;X)$ for any/some $p \in (1,\infty)$.
A Banach space with the UMD property is called a UMD Banach space. Some facts:
\begin{itemize}
\item Every Hilbert space is a UMD space;
\item If $X$ is a UMD space, $(S, \Sigma, \mu)$ is $\sigma$-finite and $p\in (1, \infty)$, then $L_p(S;X)$ is a UMD space.
\item UMD spaces are reflexive.
\item Closed subspaces and quotients of UMD spaces are again UMD spaces.
\end{itemize}
In particular, weighted Besov and Triebel-Lizorkin spaces (see Section~\ref{Boutet:subsec:prelim:FS}) are UMD spaces in the reflexive range.

Let $A$ be a closed linear operator on a Banach space $X$. For $q \in (1,\infty)$ and $v \in A_{q}(\R)$ we say that $A$ enjoys the property of
\begin{itemize}
\item $L_{q}(v,\R)$-maximal regularity if $\partial_{t}+A$ is invertible as an operator on $L_{q}(v,\R;X)$ with domain $W^{1}_{q}(\R,v;X) \cap L_{q}(\R,v;D(A))$.
\item $L_{q}(v,\R_{+})$-maximal regularity if $\partial_{t}+A$ is invertible as an operator on $L_{q}(v,\R_{+};X)$ with domain ${_{0}}W^{1}_{q}(\R_{+},v;X) \cap L_{q}(\R_{+},v;D(A))$, where
    \[
    {_{0}}W^{1}_{q}(\R_{+},v;X) = \{ u \in W^{1}_{q}(\R_{+},v;X):u(0)=0\}.
    \]
\end{itemize}
In the specific case of the power weight $v=v_{\mu}$ with $q \in (-1,q-1)$, we speak of $L_{q,\mu}(\R)$-maximal regularity and $L_{q,\mu}(\R_{+})$-maximal regularity.

Note that $L_{q}(v,\R)$-maximal regularity and $L_{q}(v,\R_{+})$-maximal regularity can also be formulated in terms of evolution equations. For instance, $A$ enjoys the property of $L_{q}(v,\R_{+})$-maximal regularity if and only if, for each $f \in L_{q}(v,\R_{+};X)$, there exists a unique solution $u \in W^{1}_{q}(\R_{+},v;X) \cap L_{q}(\R_{+},v;D(A))$ of
\[
u'+Au = f, \quad u(0)=0.
\]

References for $L_{q}(\R)$-maximal regularity and $L_{q}(\R_{+})$-maximal regularity include \cite{Arendt&Duelli2006,Mielke1987} and \cite{Dore2000,Kunstmann&Weis2004_lecture_notes}.
Works on $L_{q}(\R_{+},v)$-maximal regularity include \cite{Chill&Fiorenza2014,Chill&Krol2017,Fackler&Hytonen&Lindemulder2018}.

\begin{lemma}\label{DBVP:lemma:prelim:max-reg;imaginary_axis}
Let $X$ be a Banach space, $q \in (1,\infty)$ and $v \in A_{q}(\R)$.
Let $A$ be a linear operator on $X$ and let $\normm{\,\cdot\,}$ be a Banach norm on $D(A)$ with $(D(A),\normm{\,\cdot\,}) \hookrightarrow D(A)$.
If
\[
\partial_{t}+A:W^{1}_{q}(\R,v;X) \cap L_{q}(\R,v;(D(A),\normm{\,\cdot\,})) \longra L_{q}(\R,v;X)
\]
is an isomorphism of Banach spaces, then $\normm{\,\cdot\,} \eqsim \norm{\,\cdot\,}_{D(A)}$ and $\imath\R \subset \rho(-A)$ with
\[
\norm{(\imath\xi+A)^{-1}}_{\mathcal{B}(X)} \leq \frac{2\| (\partial_t+A)^{-1}\|_{\mathcal{B}(L_{q}(\R,v;X),W^{1}_{q}(\R,v;X))} }{1+|\xi|}, \qquad \xi \in \R.
\]
In particular, $A$ is a closed linear operator on $X$ enjoying the property of $L_{q}(v,\R)$-maximal regularity.
\end{lemma}
\begin{proof}
A slight modification of \cite[Satz~2.2]{Mielke1987} gives a mapping $R:\R \to \mathcal{B}(X,(D(A),\normm{\,\cdot\,}))$ with the property that
\[
(\imath\xi+A)R(\xi)=I_{X} \quad\text{and}\quad
\norm{R(\xi)}_{\mathcal{B}(X)} \lesssim \frac{2\| (\partial_t+A)^{-1}\|_{\mathcal{B}(L_{q}(\R,v;X),W^{1}_{q}(\R,v;X))} }{1+|\xi|}, \qquad \xi \in \R.
\]
Similarly to \cite[Theorem~4.1]{Dore2000}, using \cite[Theorem~3.7]{Dore2000} modified to the real line, it follows from the construction of $R(\xi)$ from \cite[Satz~2.2]{Mielke1987} that also $R(\xi)(\imath\xi+A)=I_{D(A)}$ for each $\xi \in \R$.
This shows that $\imath\R \subset \rho(-A)$ with $(\imath\xi+A)^{-1}=R(\xi)$.
But then
\[
\normm{x} = \normm{R(0)Ax} \lesssim \norm{Ax}_{X} \leq \norm{x}_{D(A)}, \qquad x \in D(A). \qedhere
\]
\end{proof}

\begin{lemma}\label{DBVP:lemma:prelim:max-reg}
Let $X$ be a Banach space, $q \in (1,\infty)$ and $v \in A_{q}(\R)$.
Let $A$ be a closed linear operator on $X$ with $\C_{+} \subset \rho(-A)$ enjoying the property of $L_{q}(\R,v)$-maximal regularity, where $\C_{+} = \{ z \in \C : \Re(z) > 0\}$. Suppose that $\lambda\mapsto \|(\lambda+A)^{-1}\|_{\mathcal{B}(X)}$ is polynomially bounded on $\C_+$.
Then $-A$ is the generator of an exponentially stable analytic semigroup on $X$ and $A$ also enjoys the property of $L_{q}(\R_{+},v)$-maximal regularity.
\end{lemma}
\begin{proof}
As $A$ enjoys the property of $L_{q}(\R,v)$-maximal regularity, Lemma~\ref{DBVP:lemma:prelim:max-reg;imaginary_axis} applies with $\normm{\,\cdot\,}=\norm{\,\cdot\,}_{D(A)}$. Therefore, $\overline{\C_{+}} \subset \rho(-A)$ and $\lambda \mapsto (\lambda+A)^{-1}$ and $\lambda \mapsto \lambda(\lambda+A)^{-1}$ are well-defined analytic functions $\overline{\C_{+}} \to \mathcal{B}(X)$. Moreover, both mappings satisfy the assumptions of the Phragmen-Lindel\"of Theorem (see \cite[Corollary VI.4.4]{Conway_1978}) so that both mappings are bounded. Hence, by the solution formula for the Poisson equation in the half-space (see \cite[Chapter 2, Theorem 14]{Evans_2010}) we have that
\[
\sup_{\lambda \in \C_{+}}\norm{(\lambda+A)^{-1}}_{\mathcal{B}(X)} \leq \sup_{\theta \in \R}\norm{(\imath\theta+A)^{-1}}_{\mathcal{B}(X)} < \infty
\]
and
\[
\sup_{\lambda \in \C_{+}}\norm{\lambda(\lambda+A)^{-1}}_{\mathcal{B}(X)} \leq \sup_{\theta \in \R}\norm{\imath\theta(\imath\theta+A)^{-1}}_{\mathcal{B}(X)} < \infty.
\]
It follows that $-A$ is the generator of an exponentially stable analytic semigroup $(e^{-tA})_{t \geq 0}$ on $X$.

Finally, as $-A$ is the generator of an exponentially stable analytic semigroup on $X$, the variation of constants formula yields $L_{q}(\R_{+},v)$-maximal regularity. Indeed, viewing ${_{0}}W^{1}_{q}(\R_{+},v;X) \cap L_{q}(\R_{+},v;D(A))$ and $L_{q}(\R_{+},v;X)$ as closed subspaces of $W^{1}_{q}(\R,v;X) \cap L_{q}(\R,v;D(A))$ and $L_{q}(\R,v;X)$, respectively, through extension by zero, the formula
\[
[(\partial_{t}+A)^{-1}f](t) = \int_{-\infty}^{t}e^{-(t-s)A}f(s)\,ds,\qquad f \in L_{q}(\R,v;X), t \in \R,
\]
shows that $(\partial_{t}+A)^{-1}$ maps $L_{q}(\R_{+},v;X)$ to ${_{0}}W^{1}_{q}(\R_{+},v;X) \cap L_{q}(\R_{+},v;D(A))$.
\end{proof}

As an application of its operator-valued Fourier multiplier theorem, Weis \cite{Weis2001} characterized $L_{q}(\R_{+})$-maximal regularity in terms of $R$-sectoriality in the setting of UMD Banach spaces. The corresponding result for $L_{q}(\R)$-maximal regularity involves $R$-bisectoriality, see \cite{Arendt&Duelli2006}. Using \cite[Theorem~3.5]{Fackler&Hytonen&Lindemulder2018} and Theorem~\ref{DBVP:thm:appendix:Clement&Pruess2001}, these results carry over to the weighted setting.

Let us introduce the notion of $R$-boundedness.
Let $X$ be a Banach space. Let $(\epsilon_{k})_{k \in \N}$ be a Rademacher sequence on some probability space $(\Omega,\mathcal{F},\Prob)$, i.e.\, a sequence of independent random variables with $\Prob(\epsilon_{k}=1)=\Prob(\epsilon_{k}=-1)=\frac{1}{2}$. A collection of operators $\mathcal{T} \subset \mathcal{B}(X)$ is called $R$-bounded if there exists a finite constant $C \geq 0$ such that, for all $K\in\N$, $T_{0},\ldots,T_{K} \in \mathcal{T}$ and $x_{0},\ldots,x_{K} \in X$,
\[
\norm{\sum_{k=0}^{K}\epsilon_{k}T_{k}x_{k}}_{L_{2}(\Omega;X)} \leq C \norm{\sum_{k=0}^{K}\epsilon_{k}x_{k}}_{L_{2}(\Omega;X)}.
\]
The least such constant $C$ is called the $R$-bound of $\mathcal{T}$ and is denoted by $\mathcal{R}(\mathcal{T})$.

The space $\mathrm{Rad}_{p}(\N;X)$, where $p \in [1,\infty)$, is defined as the Banach space of sequence $(x_{k})_{k \in \N}$ for which  there is convergence of $\sum_{k=0}^{\infty}\epsilon_{k}x_{k}$ in $L_{p}(\Omega;X)$, endowed with the norm
\[
\norm{(x_{k})_{k \in \N}}_{\mathrm{Rad}_{p}(\N;X)} := \norm{\sum_{k=0}^{\infty}\epsilon_{k}x_{k}}_{L_{p}(\Omega;X)} = \sup_{K \geq 0}\norm{\sum_{k=0}^{K}\epsilon_{k}x_{k}}_{L_{p}(\Omega;X)}.
\]
As a consequence of the Kahane-Khintchine inequalities, $\mathrm{Rad}_{p}(\N;X)=\mathrm{Rad}_{q}(\N;X)$ with an equivalence of norms. We put $\mathrm{Rad}(\N;X) = \mathrm{Rad}_{2}(\N;X)$.
Note that a collection of operators $\mathcal{T} \subset \mathcal{B}(X)$ is called $R$-bounded if and only if $\{\mathrm{diag}(T_{0},\ldots,T_{K}) : T_{0},\ldots,T_{K} \in \mathcal{T}\} \subset \mathcal{B}(\mathrm{Rad}(\N;X))$ is a uniformly bounded family of operators, in which case the $R$-bound coincides with that uniform bound; here
\[
\mathrm{diag}(T_{0},\ldots,T_{K})(x_{k})_{k \in \N} = (T_{0}x_{0},\ldots,T_{K}x_{K},0,0,0,\ldots).
\]
Furthermore, note that, as a consequence of the Kahane-Khintchine inequalities and Fubini, given $p \in [1,\infty)$ and a $\sigma$-finite measure space $(S,\mathscr{A},\mu)$, there is a natural isomorphism of Banach spaces
\[
\mathrm{Rad}(\N;L_{p}(S;X)) \simeq L_{p}(S;\mathrm{Rad}(\N;X)).
\]

Having introduced the notion of $R$-boundedness, we can now give the definition of $R$-sectoriality, which is an $R$-boundedness version of sectoriality.

Recall that an unbounded operator $A$ on a Banach space $X$ is a {\em sectorial operator} if $A$ is injective, closed, has dense range and there exists a $\phi \in (0,\pi)$ such that $\Sigma_{\pi-\phi} \subset \rho(-A)$ and
\begin{align*}
\sup_{\lambda\in \Sigma_{\pi-\phi}} \|\lambda (\lambda+A)^{-1}\|_{\mathcal{B}(X)}<\infty.
\end{align*}
The infimum over all possible $\phi$ is called the {\em angle of sectoriality} and is denoted by $\omega(A)$. In this case we also say that $A$ is {\em sectorial of angle} $\omega(A)$. The condition that $A$ has dense range is automatically fulfilled if $X$ is reflexive (see \cite[Proposition 10.1.9]{Hytonen&Neerven&Veraar&Weis2016_Analyis_in_Banach_Spaces_II}).

We say that an unbounded operator $A$ on a Banach space $X$ is an {\em $R$-sectorial operator} if $A$ is injective, closed, has dense range and there exists a $\phi \in (0,\pi)$ such that $\Sigma_{\pi-\phi} \subset \rho(-A)$ and
\begin{align*}
\mathcal{R}(\{ \lambda (\lambda+A)^{-1}: \lambda\in \Sigma_{\pi-\phi}\}) <\infty \quad\text{in}\quad\mathcal{B}(X).
\end{align*}
The infimum over all possible $\phi$ is called the {\em angle of $R$-sectoriality} and is denoted by $\omega_{R}(A)$. In this case we also say that $A$ is {\em $R$-sectorial of angle} $\omega_{R}(A)$.

A way to approach $L_{q}$-maximal regularity is through operator sum methods, as initiated by Dore $\&$ Venni \cite{Dore&Venni1987}.
Using the Kalton--Weis operator sum theorem \cite[Theorem~6.3]{Kalton&Weis2001_sums_of_closed_operators} in combination with \cite[Proposition~2.7]{LV2018_Dir_Laplace}, we obtain the following result:
\begin{proposition}\label{Boutet:prelim:MaxReg_R-Sec}
Let $X$ be a UMD space, $q \in (1,\infty)$ and $v \in A_{q}(\R)$. If $A$ is a closed linear operator on a Banach space $X$ with $0 \in \rho(A)$ that is $R$-sectorial of angle $\omega_{R}(A) < \frac{\pi}{2}$, then $A$ enjoys the properties of $L_{q}(v,\R)$-maximal regularity and $L_{q}(v,\R_{+})$-maximal regularity.
\end{proposition}

\subsection{Decomposition and Anisotropy}\label{Boutet:subsec:prelim:anisotropy}

Let $n = |\mathpzc{d}|_{1} = \mathpzc{d}_{1} + \ldots + \mathpzc{d}_{l}$ with $\mathpzc{d} = (\mathpzc{d}_{1},\ldots,\mathpzc{d}_{l}) \in \N_{1}^{l}$. The decomposition
\[
\R^{n} = \R^{\mathpzc{d}_{1}} \times \ldots \times \R^{\mathpzc{d}_{l}}.
\]
is called the $\mathpzc{d}$-\emph{decomposition} of $\R^{n}$.
For $x \in \R^{n}$ we accordingly write $x = (x_{1},\ldots,x_{l})$ and $x_{j}=(x_{j,1},\ldots,x_{j,\mathpzc{d}_{j}})$, where $x_{j} \in \R^{\mathpzc{d}_{j}}$ and $x_{j,i} \in \R$ $(j=1,\ldots,l; i=1,\ldots,\mathpzc{d}_{j})$. We also say that we view $\R^{n}$ as being
$\mathpzc{d}$-\emph{decomposed}. Furthermore, for each $k \in \{1,\ldots,l\}$ we define the inclusion map
\begin{equation*}
\iota_{k} = \iota_{[\mathpzc{d};k]} : \R^{\mathpzc{d}_{k}} \longra \R^{n},\, x_{k} \mapsto (0,\ldots,0,x_{k},0,\ldots,0),
\end{equation*}
and the projection map
\begin{equation*}
\pi_{k} = \pi_{[\mathpzc{d};k]} :  \R^{n} \longra \R^{\mathpzc{d}_{k}},\, x = (x_{1},\ldots,x_{l}) \mapsto x_{k}.
\end{equation*}
For $\vec{x}\in\R^l$ and $y\in\R^n$ we define
\[
 \vec{x}\cdot_{\mathpzc{d}} y:=\sum_{j=1}^l\sum_{i=1}^{\mathpzc{d}_j} x_jy_{j,i}.
\]

Given $\vec{a} \in (0,\infty)^{l}$, we define the $(\mathpzc{d},\vec{a})$-anisotropic dilation $\delta^{(\mathpzc{d},\vec{a})}_{\lambda}$ on $\R^{n}$ by $\lambda > 0$ to be the mapping $\delta^{(\mathpzc{d},\vec{a})}_{\lambda}$ on $\R^{n}$ given by the formula
\[
\delta^{(\mathpzc{d},\vec{a})}_{\lambda}x := (\lambda^{a_{1}}x_{1},\ldots,\lambda^{a_{l}}x_{l}), \quad\quad x \in \R^{n}.
\]

A $(\mathpzc{d},\vec{a})$-anisotropic distance function on $\R^{n}$ is a function $u:\R^{n} \longra [0,\infty)$ satisfying
\begin{itemize}
\item[(i)] $u(x)=0$ if and only if $x=0$.
\item[(ii)] $u(\delta^{(\mathpzc{d},\vec{a})}_{\lambda}x) = \lambda u(x)$ for all $x \in \R^{n}$ and $\lambda > 0$.
\item[(iii)] There exists a $c>0$ such that $u(x+y) \leq c(u(x)+u(y))$ for all $x,y \in \R^{n}$.
\end{itemize}
All $(\mathpzc{d},\vec{a})$-anisotropic distance functions on $\R^{n}$ are equivalent: Given two $(\mathpzc{d},\vec{a})$-anisotropic distance functions $u$ and $v$ on $\R^{n}$, there exist constants $m,M>0$ such that $m u(x) \leq v(x) \leq M u(x)$ for all $x \in \R^{n}$

In this paper we will use the $(\mathpzc{d},\vec{a})$-anisotropic distance function $|\,\cdot\,|_{\mathpzc{d},\vec{a}}:\R^{n} \longra [0,\infty)$ given by the formula
\begin{equation*}\label{functieruimten:eq:anisotropic_distance}
|x|_{\mathpzc{d},\vec{a}} := \left(\sum_{j=1}^{l}|x_{j}|^{2/a_{j}}\right)^{1/2} \quad\quad (x \in \R^{n}).
\end{equation*}

\subsection{Distribution Theory and Function Spaces}\label{Boutet:subsec:prelim:FS}

As general references for this subsection we would like to mention
\cite{Amann_distributions, Bui1982, Lindemulder_master-thesis}.

\subsubsection{Distribution Theory and Some Generic Function Space Theory}\label{Boutet:subsubsec:prelim:FS}

Let $X$ be a Banach space. The spaces of $X$-valued distributions and $X$-valued tempered distributions on $\R^{n}$ are defined as $\mathcal{D}'(\R^{n};X) := \mathcal{L}(\mathcal{D}(\R^{n}),X)$ and $\mathcal{S}'(\R^{n};X) := \mathcal{L}(\mathcal{S}(\R^{n}),X)$, respectively; for the theory of vector-valued distributions we refer to \cite{Amann_distributions} (and \cite[Section~III.4]{Amann_Lin_and_quasilinear_parabolic}).

Let $\E \hookrightarrow \mathcal{D}'(U;X)$ be a Banach space of distributions on an open subset $U \subset \R^{n}$.
Given an open subset $V \subset U$,
\[
\E(V) := \{ f \in \mathcal{D}'(V;X) : \exists g \in \E, g_{|V} = f  \}
\]
equipped with the norm
\[
\norm{f}_{\E(V)} := \inf\{ \norm{g}_{\E} : g \in \E, g_{|V} = f \}
\]
is a Banach space with $\E(V) \hookrightarrow \mathcal{D}'(V;X)$. Note that $f \mapsto f_{V}$ defines a contraction $\E \to \E(V)$.
Furthermore, note that, if $\E \hookrightarrow \F \hookrightarrow \mathcal{D}'(U;X)$, then $\E(V) \hookrightarrow \F(V)$. More generally, given Banach spaces $\E \hookrightarrow \mathcal{D}'(U_{1};X_{1})$ and $\F \hookrightarrow \mathcal{D}'(U_{2};X_{2})$, $T \in \mathcal{B}(\E,\F)$ and open subsets $V_{1} \subset U_{1}$, $V_{2} \subset U_{2}$ with the property that
\[
\forall f,g \in \E,\: f_{|V_{1}} = g_{|V_{1}} \Longra (Tf)_{|V_{2}} = (Tg)_{|V_{2}},
\]
$T$ induces an operator $\tilde{T} \in \mathcal{B}(\E(V_{1}),\F(V_{2}))$ satisfying $(Tf)_{V_{2}} = \tilde{T}(f_{|V_{1}})$ for all $f \in \E$.

Given a Banach space $Z$, $\mathscr{O}_{M}(\R^{n};Z)$ denotes the space of slowly increasing $Z$-valued smooth functions on $\R^{n}$.
Pointwise multiplication $(f,g) \mapsto fg$ yields separately continuous bilinear mappings
\begin{equation}\label{Boutet:eq:prelim:OM_bilin}
\begin{array}{ccccc}
\mathscr{O}_{M}(\R^{n};\mathcal{B}(X)) &\times& \mathcal{S}(\R^{n};X) &\longra& \mathcal{S}(\R^{n};X),\\
\mathscr{O}_{M}(\R^{n};\mathcal{B}(X)) &\times& \mathcal{S}'(\R^{n};X) &\longra& \mathcal{S}'(\R^{n};X).
\end{array}
\end{equation}
As a consequence, $(m,f) \mapsto \mathscr{F}^{-1}[m\hat{f}]$ yields separately continuous bilinear mappings \eqref{Boutet:eq:prelim:OM_bilin}. We use the following notation:
\[
T_{m}f = \operatorname{OP}[m]f = m(D)f := \mathscr{F}^{-1}[m\hat{g}].
\]

Let $\E \hookrightarrow \mathcal{D}'(U;X)$ be a Banach space of distributions on an open subset $U \subset \R^{n}$. For a finite set of multi-indices $J \subset \N^{d}$ we define the Sobolev space $\mathcal{W}^{J}[\E]$ as the space of all $f \in \E$ with $D^{\alpha}f \in \E$ for every $\alpha \in J$, equipped with the norm
\[
\norm{f}_{\mathcal{W}^{J}[\E]} := \sum_{\alpha \in J}\norm{D^{\alpha}f}_{\E}.
\]
Then $\mathcal{W}^{J}[\E]$ is a Banach space with $\mathcal{W}^{J}[\E] \hookrightarrow \E \hookrightarrow \mathcal{D}'(U;X)$.
Note that if $\F \hookrightarrow \mathcal{D}'(U;X)$ is another Banach space, then
\[
\E \hookrightarrow \F \quad \mbox{implies} \quad
\mathcal{W}^{J}[\E] \hookrightarrow \mathcal{W}^{J}[\F].
\]
Given $\vec{n} \in \N^{l}$, we define $\mathcal{W}^{\vec{n}}_{\mathpzc{d}}[\E] := \mathcal{W}^{J_{\vec{n},\mathpzc{d}}}[\E]$, where
\[
J_{\vec{n},\mathpzc{d}} := \left\{\, \alpha \in \bigcup_{j=1}^{l}\iota_{[\mathpzc{d};j]}\N^{\mathpzc{d}_{j}} \,:\, |\alpha_{j}| \leq n_{j} \,\right\}.
\]

Suppose $\R^{n}$ is $\mathpzc{d}$-decomposed as in Section~\ref{Boutet:subsec:prelim:anisotropy}.
For a Banach space $Z$, $\vec{a} \in (0,\infty)^{l}$ and $N \in \N$ we define $\mathcal{M}^{(\mathpzc{d},\vec{a})}_{N}(Z)$ as the space of all $m \in C^{N}(\R^{n};Z)$ for which
\[
\norm{m}_{\mathcal{M}^{(\mathpzc{d},\vec{a})}_{N}(Z)} :=  \sup_{|\alpha| \leq N}\sup_{\xi \in \R^{n}} (1+|\xi|_{\mathpzc{d},\vec{a}})^{\vec{a} \cdot_{\mathpzc{d}} \alpha}\norm{D^{\alpha}m(\xi)}_{Z} < \infty.
\]
When $\vec{a}=\vec{1}$ we simply write $\mathcal{M}_{N}(Z) = \mathcal{M}^{(\mathpzc{d},\vec{1})}_{N}(Z)$.

Let $\vec{a} \in (0,\infty)^{l}$.
A normed space $\E \subset \mathcal{S}'(\R^{n};X)$ is called $(\mathpzc{d},\vec{a})$-admissible if there exists an $N \in \N$ such that
\[
m(D)f \in \E \quad \mbox{with} \quad \norm{m(D)f}_{\E} \lesssim \norm{m}_{\mathcal{M}^{(\mathpzc{d},\vec{a})}_{N}(\C)}\norm{f}_{\E},
\quad\quad (m,f) \in \mathscr{O}_{M}(\R^{n};\C) \times \E.
\]
In case $\vec{a}=\vec{1}$ we simply speak of admissible.

To each $\sigma \in \R$ we associate the operators $\mathcal{J}^{[\mathpzc{d};j]}_{\sigma} \in \mathcal{L}(\mathcal{S}'(\R^{n};X))$ and $\mathcal{J}^{\mathpzc{d},\vec{a}}_{\sigma} \in \mathcal{L}(\mathcal{S}'(\R^{n};X))$ given by
\[
\mathcal{J}^{[\mathpzc{d};j]}_{\sigma}f := \mathscr{F}^{-1}[(1+|\pi_{[\mathpzc{d};j]}|^{2})^{\sigma/2}\hat{f}]
\quad\quad \mbox{and} \quad\quad \mathcal{J}^{\mathpzc{d},\vec{a}}_{\sigma}f := \sum_{k=1}^{l}\mathcal{J}^{[\mathpzc{d};k]}_{\sigma/a_{k}}f.
\]
We call $\mathcal{J}^{\mathpzc{d},\vec{a}}_{\sigma}$ the $(\mathpzc{d},\vec{a})$-anisotropic Bessel potential operator of order $\sigma$.

Let $\E \hookrightarrow \mathcal{S}'(\R^{n};X)$ be a Banach space.
Given $\vec{n} \in \left( \N_{1} \right)^{l}$, $\varsigma, \vec{a} \in (0,\infty)^{l}$, and $s \in \R$,
we define the Banach spaces
$\mathcal{H}^{\varsigma}_{\mathpzc{d}}[\E], \mathcal{H}^{s,\vec{a}}_{\mathpzc{d}}[\E] \hookrightarrow \mathcal{S}'(\R^{n};X)$ as follows:
\[
\begin{array}{ll}
\mathcal{H}^{\varsigma}_{\mathpzc{d}}[\E] &:=  \{ f \in \mathcal{S}'(\R^{n}) : \mathcal{J}^{[\mathpzc{d};j]}_{\varsigma_{j}}f \in \E, j=1,\ldots,l \},  \\
\mathcal{H}^{s,\vec{a}}_{\mathpzc{d}}[\E] &:= \{ f \in \mathcal{S}'(\R^{n}) : \mathcal{J}^{\mathpzc{d},\vec{a}}_{s}f \in \E \},
\end{array}
\]
with the norms
\[
\norm{f}_{\mathcal{H}^{\varsigma}_{\mathpzc{d}}[\E]} =  \sum_{j=1}^{l}\norm{\mathcal{J}^{[\mathpzc{d};j]}_{\varsigma_{j}}f}_{\E}, \quad\quad
\norm{f}_{\mathcal{H}^{s,\vec{a}}_{\mathpzc{d}}[\E]} = \norm{ \mathcal{J}^{\mathpzc{d},\vec{a}}_{s}f }_{\E}.
\]
Note that $\mathcal{H}^{\varsigma}_{\mathpzc{d}}[\E] \hookrightarrow \mathcal{H}^{s,\vec{a}}_{\mathpzc{d}}[\E]$ contractively in case that $\varsigma = (s/a_{1},\ldots,s/a_{l})$. Furthermore, note that if $\F \hookrightarrow \mathcal{S}'(\R^{n};X)$ is another Banach space, then
\begin{equation}
\E \hookrightarrow \F \quad \mbox{implies} \quad
\mathcal{H}^{\varsigma}_{\mathpzc{d}}[\E] \hookrightarrow \mathcal{H}^{\varsigma}_{\mathpzc{d}}[\F],
\mathcal{H}^{s,\vec{a}}_{\mathpzc{d}}[\E] \hookrightarrow \mathcal{H}^{s,\vec{a}}_{\mathpzc{d}}[\F].
\end{equation}

We write
\[
\tilde{J}_{\vec{n},\mathpzc{d}} := \{0\} \cup \left\{n_j \iota_{[\mathpzc{d};j]}e^{[\mathpzc{d}_{j}]}_{i} : j = 1,\ldots,l, i=1,\ldots,\mathpzc{d}_{i} \right\}, \qquad \vec{n} \in \N_{1}^{l},
\]
where $e^{[\mathpzc{d}_{j}]}_{i}$ is the standard $i$-th basis vector in $\R^{\mathpzc{d}_{j}}$.
If $\E \hookrightarrow \mathcal{S}'(\R^{n};X)$ is a $(\mathpzc{d},\vec{a})$-admissible Banach space for a given $\vec{a} \in (0,\infty)^{l}$, then
\begin{equation}\label{PIBVP:eq:prelim:identities_Sobolev&Bessel-potential}
\mathcal{W}^{J}[\E] = \mathcal{H}^{\vec{n}}_{\mathpzc{d}}[\E] = \mathcal{H}^{s,\vec{a}}_{\mathpzc{d}}[\E], \quad\quad  s \in \R, \vec{n} = s \vec{a}^{-1} \in  \N_{1} ^{l}, \tilde{J}_{\vec{n},\mathpzc{d}} \subset J \subset J_{\vec{n},\mathpzc{d}},
\end{equation}
and
\begin{equation}\label{PIBVP:eq:prelim:identities_Sobolev&Bessel-potential;2}
\mathcal{H}^{\varsigma}_{\mathpzc{d}}[\E] = \mathcal{H}^{s,\vec{a}}_{\mathpzc{d}}[\E], \quad\quad s >0, \varsigma = s \vec{a}^{-1}.
\end{equation}
Furthermore,
\begin{equation}\label{PIBVP:eq:prelim:differential}
D^{\alpha} \in \mathcal{B}(\mathcal{H}^{s,\vec{a}}_{\mathpzc{d}}[\E],\mathcal{H}^{s- \vec{a} \cdot_{\mathpzc{d}} \alpha,\vec{a}}_{\mathpzc{d}}[\E]), \quad\quad s \in \R, \alpha \in \N^{n}.
\end{equation}

Let $\E,\F \hookrightarrow \mathcal{S}'(\R^{n};X)$ be $(\mathpzc{d},\vec{a})$-admissible Banach spaces for a given $\vec{a} \in (0,\infty)^{l}$. If $(\,\cdot\,,\,\cdot\,)$ is an interpolation functor (e.g.\ $(\,\cdot\,,\,\cdot\,) \in \{[\,\cdot\,,\,\cdot\,]_{\theta},(\,\cdot\,,\,\cdot\,)_{\theta,q} : \theta \in [0,1], q \in [1,\infty]\}$), $(\E,\F) \hookrightarrow \mathcal{S}'(\R^{n};X)$ is $(\mathpzc{d},\vec{a})$-admissible as well. Moreover, it holds that
\begin{equation}\label{PIBVP:eq:prelim:interpolation}
(\mathcal{H}^{\varsigma}[\E],\mathcal{H}^{\varsigma}[\F])=\mathcal{H}^{\varsigma}[(\E,\F)], \qquad \varsigma \in (0,\infty)^{l}.
\end{equation}

\subsubsection{Function Spaces}\label{Boutet:subsec:FS}

\paragraph{Anisotropic mixed-norm spaces}\label{Boutet:par:prelim:anisotropic_FS}

Let $X$ be a Banach space and suppose that $\R^{n}$ is $\mathpzc{d}$-decomposed as in Section~\ref{Boutet:subsec:prelim:anisotropy}.

Let $\dom=\prod_{j=1}^{l}\dom_{j} \subset \R^{n}$ with $\dom_{j}$ an open subset of $\R^{\mathpzc{d}_{j}}$ for each $j$.
For $\vec{p} \in (1,\infty)^{l}$ and a weight vector $\vec{w} \in \prod_{j=1}^{l}\mathcal{W}(\dom_{j})$ with $\vec{p}$-dual weight vector $\vec{w}'_{\vec{p}} \in \prod_{j=1}^{l}L_{1,\mathrm{loc}}(\dom_{j})$, there is the inclusion $L_{\vec{p}}(\dom,\vec{w};X) \hookrightarrow \mathcal{D}'(\dom;X)$ (which can be seen through the pairing \eqref{DSOP:eq:subsec:prelim:weights;pairing;mixed-norm}).
So we can define the associated Sobolev space of order $\vec{k} \in \N^{l}$
\[
W^{\vec{k}}_{\vec{p}}(\dom,\vec{w};X) := \mathcal{W}^{(\vec{k},\mathpzc{d})}[L_{\vec{p}}(\dom,\vec{w};X)].
\]
An example of a weight $w$ on a $C^{\infty}-$domain $\dom \subset \R^{n}$ for which the $p$-dual weight $w'_{p} = w^{-\frac{1}{p-1}} \in L_{1,\mathrm{loc}}(\dom)$ is the power weight $w_{\gamma}^{\partial\mathscr{O}} = \mathrm{dist}(\,\cdot\,,\partial\mathscr{O})^{\gamma}$ with $\gamma \in \R$. Furthermore, note that $w'_{p} \in A_{\infty}(\R^{n}) \subset L_{1,\mathrm{loc}}(\R^{n})$ for $w \in [A_{\infty}]'_{p}(\R^{n}) \supset A_{p}(\R^{n})$.

Let $\vec{p} \in (1,\infty)^{l}$ and $\vec{w} \in \prod_{j=1}^{l}A_{p_{j}}(\R^{\mathpzc{d}_{j}})$. Then $\vec{w}'_{\vec{p}} \in \prod_{j=1}^{l}A_{p_{j}'}(\R^{\mathpzc{d}_{j}})$, so that $\mathcal{S}(\R^{n}) \stackrel{d}{\hookrightarrow} L_{\vec{p}'}(\R^{n},\vec{w}'_{\vec{p}})$. Using the pairing \eqref{DSOP:eq:subsec:prelim:weights;pairing;mixed-norm}, we find that $L_{\vec{p}}(\R^{n},\vec{w};X) \hookrightarrow \mathcal{S}'(\R^{n};X)$ in the natural way.
For $\vec{a} \in (0,\infty)^{l}$, $s \in \R$ and $\varsigma \in (0,\infty)^{l}$ we can thus define the Bessel potential spaces
\[
H^{\varsigma}_{\vec{p}}(\R^{n},\vec{w};X) := \mathcal{H}^{(\varsigma,\mathpzc{d})}_{\mathpzc{d}}[L_{\vec{p}}(\R^{n},\vec{w};X)],  \qquad H^{s,\vec{a}}_{\vec{p}}(\R^{n},\vec{w};X) := \mathcal{H}^{s,(\vec{a},\mathpzc{d})}_{\mathpzc{d}}[L_{\vec{p}}(\R^{n},\vec{w};X)].
\]
If $X$ is a UMD space and $\vec{w} \in \prod_{j=1}^{l}A_{p_{j}}^{\operatorname{rec}}(\R^{\mathpzc{d}_{j}})$\footnote{$\vec{w} \in \prod_{j=1}^{l}A_{p_{j}}(\R^{\mathpzc{d}_{j}})$ should already work, but this is not available in the literature and not needed in this paper anyway.}, then $L_{\vec{p}}(\R^{n},\vec{w};X)$ is $(\vec{a},\mathpzc{d})$-admissible (see~\cite{Fackler&Hytonen&Lindemulder2018}).
In particular, if $X$ is a UMD space, then \eqref{PIBVP:eq:prelim:identities_Sobolev&Bessel-potential}, \eqref{PIBVP:eq:prelim:identities_Sobolev&Bessel-potential;2} and \eqref{PIBVP:eq:prelim:differential} hold true with $\E=L_{\vec{p}}(\R^{n},\vec{w};X)$.

Let $\vec{a} \in (0,\infty)^{l}$.
For $0 < A < B < \infty$ we define $\Phi^{\mathpzc{d},\vec{a}}_{A,B}(\R^{n})$ as the set of all sequences $\varphi = (\varphi_{n})_{n \in \N} \subset \mathcal{S}(\R^{n})$ which are constructed in the following way: given a $\varphi_{0} \in \mathcal{S}(\R^{n})$ satisfying
\begin{equation*}
0 \leq \hat{\varphi}_{0} \leq 1,\: \hat{\varphi}_{0}(\xi) = 1 \:\mbox{if}\: |\xi|_{\mathpzc{d},\vec{a}} \leq A,\:
\hat{\varphi}_{0}(\xi) = 0 \:\mbox{if}\: |\xi|_{\mathpzc{d},\vec{a}} \geq B,
\end{equation*}
$(\varphi_{n})_{n \geq 1} \subset \mathcal{S}(\R^{n})$ is defined via the relations
\begin{equation*}
\hat{\varphi}_{n}(\xi) = \hat{\varphi}_{1}(\delta^{(\mathpzc{d},\vec{a})}_{2^{-n+1}}\xi) = \hat{\varphi}_{0}(\delta^{(\mathpzc{d},\vec{a})}_{2^{-n}}\xi) - \hat{\varphi}_{0}(\delta^{(\mathpzc{d},\vec{a})}_{2^{-n+1}}\xi),
\quad \xi \in \R^{n}, n \geq 1.
\end{equation*}
Observe that
\begin{equation*}\label{functieruimten:eq:support_LP-sequence}
\supp \hat{\varphi}_{0} \subset \{ \xi \mid |\xi|_{\mathpzc{d},\vec{a}} \leq B \} \quad \mbox{and} \quad
\supp \hat{\varphi}_{n} \subset \{ \xi \mid 2^{n-1}A \leq |\xi|_{\mathpzc{d},\vec{a}} \leq 2^{n}B \}, \quad  n \geq 1.
\end{equation*}

We put $\Phi^{\mathpzc{d},\vec{a}}(\R^{n}) := \bigcup_{0<A<B<\infty}\Phi^{\mathpzc{d},\vec{a}}_{A,B}(\R^{n})$.
In case $l=1$ we write $\Phi^{a}(\R^{n}) = \Phi^{\mathpzc{d},a}(\R^{n})$, $\Phi(\R^{n}) = \Phi^{1}(\R^{n})$, $\Phi^{a}_{A,B}(\R^{n}) = \Phi^{\mathpzc{d},a}_{A,B}(\R^{n})$, and $\Phi_{A,B}(\R^{n}) = \Phi^{1}_{A,B}(\R^{n})$.

To $\varphi \in \Phi^{\mathpzc{d},a}(\R^{n})$ we associate the family of convolution operators
$(S_{n})_{n \in \N} = (S_{n}^{\varphi})_{n \in \N}   \subset \mathcal{L}(\mathcal{S}'(\R^{n};X),\mathscr{O}_{M}(\R^{n};X)) \subset \mathcal{L}(\mathcal{S}'(\R^{n};X))$ given by
\begin{equation}\label{functieruimten:eq:convolutie_operatoren}
S_{n}f = S_{n}^{\varphi}f := \varphi_{n} * f = \mathscr{F}^{-1}[\hat{\varphi}_{n}\hat{f}] \quad\quad (f \in \mathcal{S}'(\R^{n};X)).
\end{equation}
It holds that $f = \sum_{n=0}^{\infty}S_{n}f$ in $\mathcal{S}'(\R^{n};X)$ respectively in $\mathcal{S}(\R^{n};X)$ whenever $f \in \mathcal{S}'(\R^{n};X)$ respectively $f \in \mathcal{S}(\R^{n};X)$.

Given $\vec{a} \in (0,\infty)^{l}$, $\vec{p} \in [1,\infty)^{l}$, $q \in [1,\infty]$, $s \in \R$, and $\vec{w} \in \prod_{j=1}^{l}A_{\infty}(\R^{\mathpzc{d}_{j}})$,
the Besov space $B^{s,\vec{a}}_{\vec{p},q,\mathpzc{d}}(\R^{n},\vec{w};X)$ is defined as the Banach space of all $f \in \mathcal{S}'(\R^{n};X)$ for which
\[
\norm{f}_{B^{s,\vec{a}}_{\vec{p},q,\mathpzc{d}}(\R^{n},\vec{w};X)}
:= \norm{(2^{ns}S_{n}^{\varphi}f)_{n \in \N}}_{\ell_{q}(\N)[L_{\vec{p}}(\R^{n},\vec{w})](X)} < \infty
\]	
and the Triebel-Lizorkin space $F^{s,\vec{a}}_{\vec{p},q,\mathpzc{d}}(\R^{n},\vec{w};X)$ is defined as the Banach space of all $f \in \mathcal{S}'(\R^{n};X)$ for which
\[
\norm{f}_{F^{s,\vec{a}}_{\vec{p},q,\mathpzc{d}}(\R^{n},\vec{w};X)}
:= \norm{(2^{ns}S_{n}^{\varphi}f)_{n \in \N}}_{L_{\vec{p}}(\R^{n},\vec{w})[\ell_{q}(\N)](X)} < \infty.
\]	
Up to an equivalence of extended norms on $\mathcal{S}'(\R^{n};X)$, $\norm{\,\cdot\,}_{B^{s,\vec{a}}_{\vec{p},q,\mathpzc{d}}(\R^{n},\vec{w};X)}$ and $\norm{\,\cdot\,}_{F^{s,\vec{a}}_{\vec{p},q,\mathpzc{d}}(\R^{n},\vec{w};X)}$ do not depend on the particular choice of $\varphi \in \Phi^{\mathpzc{d},\vec{a}}(\R^{n})$.

Let us note some basic relations between these spaces.
Monotonicity of $\ell^{q}$-spaces yields that, for $1 \leq q_{0} \leq q_{1} \leq \infty$,
\begin{equation}
B^{s,\vec{a}}_{\vec{p},q_{0},\mathpzc{d}}(\R^{n},\vec{w};X) \hookrightarrow B^{s,\vec{a}}_{\vec{p},q_{1},\mathpzc{d}}(\R^{n},\vec{w};X), \quad\quad
F^{s,\vec{a}}_{\vec{p},q_{0},\mathpzc{d}}(\R^{n},\vec{w};X) \hookrightarrow F^{s,\vec{a}}_{\vec{p},q_{1},\mathpzc{d}}(\R^{n},\vec{w};X).
\end{equation}
For $\epsilon > 0$ it holds that
\begin{equation}\label{PIBVP:prelim:eq:elem_embd_epsilon}
B^{s,\vec{a}}_{\vec{p},\infty,\mathpzc{d}}(\R^{n},\vec{w};X) \hookrightarrow B^{s-\epsilon,\vec{a}}_{\vec{p},1,\mathpzc{d}}(\R^{n},\vec{w};X).
\end{equation}
Furthermore, Minkowksi's inequality gives
\begin{equation}\label{PIBVP:prelim:eq:elem_embd_BF_rel}
B^{s,\vec{a}}_{\vec{p},\min\{p_{1},\ldots,p_{l},q\},\mathpzc{d}}(\R^{n},\vec{w};X) \hookrightarrow F^{s,\vec{a}}_{\vec{p},q,\mathpzc{d}}(\R^{n},\vec{w};X) \hookrightarrow
B^{s,\vec{a}}_{\vec{p},\max\{p_{1},\ldots,p_{l},q\},\mathpzc{d}}(\R^{n},\vec{w};X).
\end{equation}

The real interpolation of weighted isotropic scalar-valued Triebel-Lizorkin spaces from \cite[Theorem~3.5]{Bui1982} (see \cite[Proposition~6.1]{Meyries&Veraar2012_sharp_embedding} in the Banach space-valued $A_{p}$-setting), can be extended to the mixed-norm anisotropic Banach space-valued setting.
This yields the following interpolation identity: if $\vec{p} \in [1,\infty)^{l}$, $q_{0},q_{1},q \in [1,\infty]$, $s_{0},s_{1} \in \R$, $s_{0} \neq s_{1}$, $\vec{w} \in \prod_{j=1}^{l}A_{\infty}(\R^{\mathpzc{d}_{j}})$, $\theta \in (0,1)$ and $s=(1-\theta)s_{0}+\theta s_{1}$, then
\begin{equation}\label{Boutet:intro:eq:real_int_F-spaces}
(F^{s_{0},\vec{a}}_{\vec{p},q_0,\mathpzc{d}}(\R^{n},\vec{w};X),F^{s_1,\vec{a}}_{\vec{p},q_1,\mathpzc{d}}(\R^{n},\vec{w};X))_{\theta,q} = B^{s,\vec{a}}_{\vec{p},q,\mathpzc{d}}(\R^{n},\vec{w};X).
\end{equation}

The Besov space $B^{s,\vec{a}}_{\vec{p},q,\mathpzc{d}}(\R^{n},\vec{w};X)$ and the Triebel-Lizorkin space $F^{s,\vec{a}}_{\vec{p},q,\mathpzc{d}}(\R^{n},\vec{w};X)$ are examples of  $(\mathpzc{d},\vec{a})$-admissible Banach spaces.
In fact (see \cite[Proposition~5.2.26]{Lindemulder_master-thesis}), if $\E=B^{s,\vec{a}}_{\vec{p},q,\mathpzc{d}}(\R^{n},\vec{w};X)$ or $\E=F^{s,\vec{a}}_{\vec{p},q,\mathpzc{d}}(\R^{n},\vec{w};X)$, then there  exists an $N \in \N$, independent of $X$, such that
\begin{equation}\label{Boutet:eq:prelim:FM_B&F}
\norm{m(D)f}_{\E} \lesssim_{\vec{p},q,\vec{a},\vec{w}} \norm{m}_{\mathcal{M}^{(\mathpzc{d},\vec{a})}_{N}(\mathcal{B}(X))}\norm{f}_{\E},
\qquad (m,f) \in \mathscr{O}_{M}(\R^{n};\mathcal{B}(X)) \times \E.
\end{equation}

\begin{lemma}\label{Boutet:lemma:R-bdd_FM_F-scale}
Let $X$ be a Banach space, $\vec{a} \in (0,\infty)^{l}$, $\vec{p} \in [1,\infty)^{l}$, $q \in [1,\infty)$, $\vec{w} \in \prod_{j=1}^{l}A_{\infty}(\R^{\mathpzc{d}_{j}})$, $\mathscr{A} \in \{B,F\}$ and $s \in \R$.
There exists $N \in \N$, only depending on $\vec{a},\vec{p},q,n,\vec{w}$, such that if $\mathscr{M} \subset \mathscr{O}_{M}(\R^{n};\mathcal{B}(X))$ satisfies
\[
\norm{\mathscr{M}}_{\mathcal{R}\mathfrak{M}^{(\mathpzc{d},\vec{a})}_{N}} := \sup_{|\alpha| \leq N}
\mathcal{R}\left\{ (1+|\xi|_{\mathpzc{d},\vec{a}})^{\vec{a} \cdot_{\mathpzc{d}} \alpha}D^{\alpha}m(\xi) : \xi \in \R^{n}, m \in \mathscr{M}\right\},
\]
then
\[
\mathcal{R}\{T_{m}:m \in \mathscr{M}\} \lesssim_{\vec{a},\vec{p},q,n,\vec{w}} \norm{\mathscr{M}}_{\mathcal{R}\mathfrak{M}^{(\mathpzc{d},\vec{a})}_{N}} \qquad \text{in} \qquad \mathcal{B}(\mathscr{A}^{s}_{\vec{p},q}(\R^{n},\vec{w};X))).
\]
\end{lemma}
\begin{proof}
For simplicity of notation we only treat the case $\mathscr{A}=F$.
Let $N$ be as in \eqref{Boutet:eq:prelim:FM_B&F} for $\E = F^{s}_{\vec{p},q}(\R^{n},\vec{w};\mathrm{Rad}(\N;X)))$.
Now consider $\mathscr{M} \subset \mathscr{O}_{M}(\R^{n};\mathcal{B}(X))$ satisfying $\norm{\mathscr{M}}_{\mathcal{R}\mathfrak{M}^{(\mathpzc{d},\vec{a})}_{N}} < \infty$.
Let $m_{0},\ldots,m_{M} \in \mathscr{M}$. Then
\[
\vec{m}(\xi) := \mathrm{diag}(m_{1}(\xi),\ldots,m_{M}(\xi))
\]
defines a symbol $\vec{m} \in \mathscr{O}_{M}(\R^{n};\mathcal{B}(\mathrm{Rad}(\N;X)))$ with $\norm{\vec{m}}_{\mathcal{M}_{N}(\mathcal{B}(\mathrm{Rad}(\N;X)))} \leq \norm{\mathscr{M}}_{\mathcal{R}\mathfrak{M}_{N}}$.
So, by \eqref{Boutet:eq:prelim:FM_B&F}, $T_{\vec{m}} \in \mathcal{B}(F^{s}_{p,r}(\R^{n},w;\mathrm{Rad}(\N;X))))$ with
\[
\norm{T_{\vec{m}}}_{\mathcal{B}(F^{s}_{\vec{p},q}(\R^{n},\vec{w};\mathrm{Rad}(\N;X))))} \lesssim_{\vec{a},\vec{p},q,n,\vec{w}} \norm{\mathscr{M}}_{\mathcal{R}\mathfrak{M}_{N}}.
\]
Now note that
\[
F^{s}_{\vec{p},q}(\R^{n},\vec{w};\mathrm{Rad}(\N;X))) \simeq \mathrm{Rad}(\N;F^{s}_{\vec{p},q}(\R^{n},\vec{w};X))
\]
as a consequence of the Kahane-Khintchine inequalities and Fubini.
Finally, the observation that $T_{\vec{m}} = \mathrm{diag}(T_{m_{0}},\ldots,T_{m_{M}})$ completes the proof.
\end{proof}

Let $\vec{a} \in (0,\infty)^{l}$, $\vec{p} \in [1,\infty)^{l}$, $q \in [1,\infty]$, and $\vec{w} \in \prod_{j=1}^{l}A_{\infty}(\R^{\mathpzc{d}_{j}})$. For $s,s_{0} \in \R$ it holds that
\begin{equation}\label{Boutet:eq:prelim:Bessel-pot_B&F}
B^{s+s_{0},\vec{a}}_{\vec{p},q,\mathpzc{d}}(\R^{n},\vec{w};X) = \mathcal{H}^{s,\vec{a}}_{\mathpzc{d}}[B^{s_{0},\vec{a}}_{\vec{p},q,\mathpzc{d}}(\R^{n},\vec{w};X)],
\quad
F^{s+s_{0},\vec{a}}_{\vec{p},q,\mathpzc{d}}(\R^{n},\vec{w};X) = \mathcal{H}^{s,\vec{a}}_{\mathpzc{d}}[F^{s_{0},\vec{a}}_{\vec{p},q,\mathpzc{d}}(\R^{n},\vec{w};X)].
\end{equation}

Let $\vec{p} \in (1,\infty)^{l}$ and $\vec{w} \in \prod_{j=1}^{l}A_{p_{j}}(\R^{\mathpzc{d}})$
If
\begin{itemize}
\item $\E = W^{\vec{n}}_{\vec{p},\mathpzc{d}}(\R^{n},\vec{w};X)$, $\vec{n} \in \N^{l}$, $\vec{n}=s\vec{a}^{-1}$; or
\item $\E = H^{s,\vec{a}}_{\vec{p},\mathpzc{d}}(\R^{n},\vec{w};X)$; or
\item $\E = H^{\varsigma}_{\vec{p},\mathpzc{d}}(\R^{n},\vec{w};X)$, $\vec{a} \in (0,\infty)^{l}$, $\varsigma=s\vec{a}^{-1}$,
\end{itemize}
then we have the inclusions
\begin{equation}\label{PIBVP:eq:elem_embedding_FEF}
F^{s,\vec{a}}_{\vec{p},1,\mathpzc{d}}(\R^{n},\vec{w};X) \hookrightarrow  \E
\hookrightarrow F^{s,\vec{a}}_{\vec{p},\infty,\mathpzc{d}}(\R^{n},\vec{w};X).
\end{equation}

The following result is a representation for anisotropic mixed-norm Triebel-Lizorkin spaces in terms of classical isotropic Triebel-Lizorkin spaces (see Paragraph~\ref{Boutet:par:prelim:isotropic_FS}).
\begin{theorem}[\cite{Lindemulder_Intersection}, Example 5.9]\label{functieruimten:thm:aTL_rep_intersection}
Let $X$ be a Banach space, $l=2$, $\vec{a} \in (0,\infty)^{2}$, $p,q \in (1,\infty)$, $s > 0$,
and $\vec{w} \in A_{p}(\R^{\mathpzc{d}_{1}}) \times A_{q}(\R^{\mathpzc{d}_{2}})$. Then
\begin{equation}\label{functieruimten:eq:thm;aTL_rep_intersection}
F^{s,\vec{a}}_{(p,q),p}(\R^{\mathpzc{d}_1}\times\R^{\mathpzc{d}_2},\vec{w};X) =
F^{s/a_{2}}_{q,p}(\R^{\mathpzc{d}_{2}},w_{2};L^{p}(\R^{\mathpzc{d}_{1}},w_{1};X)) \cap
L^{q}(\R^{\mathpzc{d}_{2}},w_{2};F^{s/a_{1}}_{p,p}(\R^{\mathpzc{d}_{1}},w_{1};X))
\end{equation}
with equivalence of norms.
\end{theorem}
This intersection representation is actually a corollary of a more general intersection representation in \cite{Lindemulder_Intersection}, see \cite[Example~5.5]{Lindemulder_Intersection}. In the above form it can also be found in \cite[Theorem~5.2.35]{Lindemulder_master-thesis}. For the case $X=\C$, $\mathpzc{d}_{1}=1$, $\vec{w}=\vec{1}$ we refer to \cite[Proposition~3.23]{Denk&Kaip}.

There is the following analogue of Theorem~\ref{functieruimten:thm:aTL_rep_intersection} in the Besov space case:
\begin{theorem}[\cite{Lindemulder_Intersection}, Example 5.9]\label{functieruimten:thm:aB_rep_intersection}
Let $X$ be a Banach space, $l=2$, $\vec{a} \in (0,\infty)^{2}$, $p,q \in (1,\infty)$, $s > 0$,
and $\vec{w} \in A_{p}(\R^{\mathpzc{d}_{1}}) \times A_{q}(\R^{\mathpzc{d}_{2}})$. Then
\begin{equation}\label{functieruimten:eq:thm;aTL_rep_intersection:Besov}
B^{s,\vec{a}}_{(p,q),q}(\R^{n},\vec{w};X) =
B^{s/a_{2}}_{q,q}(\R^{\mathpzc{d}_{2}},w_{2};L^{p}(\R^{\mathpzc{d}_{1}},w_{1};X)) \cap
L^{q}(\R^{\mathpzc{d}_{2}},w_{2};B^{s/a_{1}}_{p,q}(\R^{\mathpzc{d}_{1}},w_{1};X))
\end{equation}
with equivalence of norms.
\end{theorem}

In the parameter range that we have defined the spaces $H^{\varsigma}_{\vec{p},\mathpzc{d}}(\R^{n},\vec{w};X)$, $H^{s,\vec{a}}_{\vec{p},\mathpzc{d}}(\R^{n},\vec{w};X)$,
$B^{s,\vec{a}}_{\vec{p},q,\mathpzc{d}}(\R^{n},\vec{w};X)$ and
$F^{s,\vec{a}}_{\vec{p},q,\mathpzc{d}}(\R^{n},\vec{w};X)$ above, the corresponding versions on open subsets $\dom \subset \R^{n}$ are defined by restriction:
\[
\begin{array}{llll}
H^{\varsigma}_{\vec{p},\mathpzc{d}}(\dom,\vec{w};X) &:= [H^{\varsigma}_{\vec{p},\mathpzc{d}}(\R^{n},\vec{w};X)](\dom), \quad & H^{s,\vec{a}}_{\vec{p},\mathpzc{d}}(\dom,\vec{w};X) &:= [H^{s,\vec{a}}_{\vec{p},\mathpzc{d}}(\R^{n},\vec{w};X)](\dom), \\
B^{s,\vec{a}}_{\vec{p},q,\mathpzc{d}}(\dom,\vec{w};X) &:= B^{s,\vec{a}}_{\vec{p},q,\mathpzc{d}}(\R^{n},\vec{w};X)(\dom), \quad &
F^{s,\vec{a},\mathpzc{d}}_{\vec{p},q}(\dom,\vec{w};X) &:= F^{s,\vec{a}}_{\vec{p},q,\mathpzc{d}}(\R^{n},\vec{w};X)(\dom).
\end{array}
\]

\paragraph{Isotropic spaces}\label{Boutet:par:prelim:isotropic_FS}

\subparagraph{\emph{Parameter-independent spaces}}

In the special case $l=1$ and $\vec{a}=1$, the anisotropic mixed-norm spaces introduced in Paragraph~\ref{Boutet:par:prelim:anisotropic_FS} reduce to classical isotropic Sobolev, Bessel potential, Besov and Triebel-Lizorkin spaces $W^{k}_{p}(\dom,w;X)$, $H^{s}_{p}(\dom,w;X)$, $B^{s}_{p,q}(\dom,w;X)$,
$F^{s}_{p,q}(\dom,w;X)$, respectively.
In the case that $\dom$ is a $C^{\infty}$-domain and $w=w^{\partial\dom}_{\gamma}$, we use the notation:
\[
\begin{array}{llll}
W^{k}_{p,\gamma}(\dom;X) &:= W^{k}_{p}(\dom,w^{\partial\dom}_{\gamma};X), \quad & H^{s}_{p,\gamma}(\dom;X) &:= H^{s}_{p}(\dom,w^{\partial\dom}_{\gamma};X), \\
B^{s}_{p,q,\gamma}(\dom;X) &:= B^{s}_{p,q}(\dom,w^{\partial\dom}_{\gamma};X), \quad &
F^{s}_{p,q,\gamma}(\dom;X) &:= F^{s}_{p,q}(\dom,w^{\partial\dom}_{\gamma};X).
\end{array}
\]

If $X$ is a UMD space, $p \in (1,\infty)$ and $w \in A_{p}(\R^{n})$, then $L_{p}(\R^{n},w;X)$ is an admissible Banach space of tempered distributions. By lifting, $H^{s}_{p}(\R^{n},w;X)$ is admissible as well. In fact, there is an operator-valued Mikhlin theorem for $H^{s}_{p}(\R^{n},w;X)$ (obtained by lifting from $L_{p}(\R^{n},w;X)$):
\begin{proposition}\label{Boutet:prop:prelim:operator-Mihklin}
Let X be UMD space, $p \in (1,\infty)$ and $w \in A_{p}(\R^{n})$. If $m \in C^{n+2}(\R^{n}\setminus \{0\};\mathcal{B}(X))$ satisfies
\[
\norm{m}_{\mathcal{R}\mathfrak{M}_{n+2}} = \sup_{|\alpha| \leq n+2}\mathcal{R}\{|\xi|^{\alpha}D^{\alpha}m(\xi) : \xi \in \R^{n}\setminus\{0\}\} < \infty,
\]
then
\[
T_{m}:\mathcal{S}(\R^{n};X) \longra L_{\infty}(\R^{n};X),\, m \mapsto \mathcal{F}^{-1}[m\hat{f}],
\]
extends to a bounded linear operator on $H^{s}_{p}(\R^{n},w;X)$ with
\[
\norm{T_{m}}_{\mathcal{B}(H^{s}_{p}(\R^{n},w;X))} \lesssim_{X,p,w,n} \norm{m}_{\mathcal{R}\mathfrak{M}_{n+2}}
\]
\end{proposition}
\begin{proof}
The case $s=0$ can be obtained as in \cite[Proposition~3.1]{Meyries&Veraar2015_pointwise_multiplication}, from which the case of general $s \in \R$ subsequently follows by lifting.
\end{proof}
If $X$ is a UMD space, $k \in \N$, $p \in (1,\infty)$ and $w \in A_{p}(\R^{n})$, then, as a consequence of the admissibility,
\begin{equation}\label{DSOP:eq:prelim:H=W_UMD}
H^{k}_{p}(\R^{n},w;X) = W^{k}_{p}(\R^{n},w;X).
\end{equation}
In the reverse direction we for instance have that, given a Banach space $X$, if $H^{1}_{p}(\R;X) = W^{1}_{p}(\R;X)$, then $X$ is a UMD space (see \cite{Hytonen&Neerven&Veraar&Weis2016_Analyis_in_Banach_Spaces_I}).

In the scalar-valued case $X=\C$, we have
\begin{equation}\label{eq:identity_Bessel-Potential_Triebel-Lizorkin;scalar-valued}
H^{s}_{p}(\R^{n},w) = F^{s}_{p,2}(\R^{n},w), \quad\quad p \in (1,\infty), w \in A_{p}.
\end{equation}
In the Banach space-valued case, this identity is valid if and only if $X$ is isomorphic to a Hilbert space.
For general Banach spaces $X$ we still have (see \cite[Proposition~3.12]{Meyries&Veraar2012_sharp_embedding})
\begin{equation}\label{eq:relation_Bessel-Potential_Triebel-Lizorkin;Ap}
F^{s}_{p,1}(\R^{n},w;X) \hookrightarrow H^{s}_{p}(\R^{n},w;X) \hookrightarrow F^{s}_{p,\infty}(\R^{n},w;X),
\quad\quad p \in (1,\infty),w \in A_{p}(\R^{n}),
\end{equation}
\begin{equation}\label{eq:relation_Bessel-Potential_Triebel-Lizorkin;Ap;Classical}
F^{k}_{p,1}(\R^{n},w;X) \hookrightarrow W^{k}_{p}(\R^{n},w;X) \hookrightarrow F^{k}_{p,\infty}(\R^{n},w;X),
\quad\quad p \in (1,\infty),w \in A_{p}(\R^{n}),
\end{equation}
and (see \cite[$(7.1)$]{LV2018_Dir_Laplace})
\begin{equation}\label{DSOP:eq:prelim:emb_F_into_Lp}
F^{k}_{p,1,\gamma}(\dom;X)
\hookrightarrow W^{k}_{p,\gamma}(\dom;X), \qquad k \in \N, p \in (1,\infty), \gamma \in (-1,\infty),
\end{equation}
where $\dom \subset \R^{n}$ is $\R^{n}_{+}$ or a $C^{\infty}$-domain with compact boundary.

For UMD spaces $X$ there is a suitable randomized substitute for \eqref{eq:identity_Bessel-Potential_Triebel-Lizorkin;scalar-valued} (see \cite[Proposition~3.2]{Meyries&Veraar2015_pointwise_multiplication}).

Let $\dom \subset \R^{n}$ be a Lipschitz domain, $p \in [1,\infty)$, $r_{0},r_{1} \in [1,\infty]$, $\gamma_{0},\gamma_{1} \in (-1,\infty)$ and $s_{0},s_{1} \in \R$.
By \cite{Meyries&Veraar2012_sharp_embedding,Meyries&Veraar2014_char_class_embeddings}, if $\gamma_{0} > \gamma_{1}$ and $s_{0} = s_{1}+\frac{\gamma_{0}-\gamma_{1}}{p}$, then
\begin{equation}\label{DSOP:eq:prelim:Sob_embd}
F^{s_{0}}_{p,r_{0}}(\R^{n},w^{\partial\dom}_{\gamma_{0}};X) \hookrightarrow F^{s_{1}}_{p,r_{1}}(\R^{n},w^{\partial\dom}_{\gamma_{1}};X).
\end{equation}

For the next result the reader is referred to \cite[Propositions 5.5 $\&$ 5.6]{Lindemulder&Meyries&Veraar2017}.
\begin{proposition}\label{prop:extensionH}
Let $X$ be a UMD space and $p\in (1, \infty)$.
Let $w\in A_p$ be such that $w(-x_1, \tilde{x}) = w(x_1, \tilde{x})$ for all $x_1\in \R$ and $\tilde{x}\in \R^{n-1}$.
\begin{enumerate}[$(1)$]
\item $H^{k,p}(\R^n_+,w;X) = W^{k,p}(\R^n_+,w;X)$ for all $k \in \N$.
\item\label{it:extensionH1} Let $\theta\in [0,1]$ and $s_0, s_1,s\in \R$ be such that $s = s_0 (1-\theta) + s_1 \theta$. Then for $\dom = \R^n$ or $\dom = \R^n_+$ one has
\[[H^{s_0,p}(\dom,w;X), H^{s_1,p}(\dom,w;X)]_{\theta} = H^{s,p}(\dom,w;X)\]
\item\label{it:extensionH2} For each $m\in \N$ there exists an $\mathcal{E}_+^m\in \mathcal{B}(H^{-m,p}(\R^n_+,w;X), H^{-m,p}(\R^n,w;X))$ such that
\begin{itemize}
\item for all $|s|\leq m$, $\mathcal{E}_+^{m}\in \mathcal{B}(H^{s,p}(\R_+^n,w;X), H^{s,p}(\R^n,w;X))$,
\item for all $|s|\leq m$, $f\mapsto (\mathcal{E}_+^{m} f)|_{\R^n_+}$ equals the identity operator on $H^{s,p}(\R_+^n,w;X)$.
    \end{itemize}
Moreover, if $f\in L^p(\R^n_+,w;X)\cap C^{m}(\overline{\R^n_+};X)$, then $\mathcal{E}_+^m f\in C^{m}(\R^n;X)$.
\end{enumerate}
\end{proposition}

\begin{theorem}[Rychkov's extension operator \cite{Rychkov_Restr&extensions_B&F-spaces_Lipschitz_domains}]\label{DBVP:thm:Rychkov's_extension_operator}
Let $\mathscr{O}$ be a special Lipschitz domain in $\R^{n}$ or a Lipschitz domain in $\R^{n}$ with a compact boundary and let $X$ be a Banach space. Then there exists a linear operator
\[
\mathscr{E}: D(\mathscr{E}) \subset \mathcal{D}'(\mathscr{O};X) \longra \mathcal{D}'(\R^{n};X)
\]
with the properties  that
\begin{itemize}
\item $(\mathscr{E}f)_{|\mathscr{O}} = f$ for all $f \in D(\mathscr{E})$;
\item $\mathscr{A}^{s}_{p,q}(\mathscr{O},w;X) \subset D(\mathscr{E})$ with $\mathscr{E} \in \mathcal{B}(\mathscr{A}^{s}_{p,q}(\mathscr{O},w;X),\mathscr{A}^{s}_{p,q}(\R^{n},w;X))$ whenever $p \in [1,\infty)$, $q \in [1,\infty]$ and $w \in A_{\infty}(\R^{n})$.
In particular, $\mathcal{S}(\mathscr{O};X) \subset D(\mathscr{E})$ with $\mathscr{E} \in \mathcal{B}(\mathcal{S}(\mathscr{O};X),BC^{\infty}(\R^{n};X))$.
\end{itemize}
\end{theorem}
\begin{proof}
The existence of such an operator for the unweighted scalar-valued variant was obtained in
\cite[Theorem~4.1]{Rychkov_Restr&extensions_B&F-spaces_Lipschitz_domains}. However, the proof given there extends to the weighted Banach space-valued setting.
\end{proof}

Let $\mathscr{O}$ be either $\R^{n}_{+}$ or a $C^{N}$-domain in $\R^{n}$ with a compact boundary $\partial\mathscr{O}$, where $N \in \N$. Let $X$ be Banach space, $p \in [1,\infty)$, $q \in [1,\infty]$, $\gamma \in (-1,\infty)$ and $s \in \R$. It will be convenient to define
\[
\partial B^{s}_{p,q,\gamma}(\partial\mathscr{O};X) := B^{s-\frac{1+\gamma}{p}}_{p,q}(\partial \mathscr{O};X)
\quad\quad\mbox{and}\quad\quad  \partial F^{s}_{p,q,\gamma}(\partial\mathscr{O};X) := F^{s-\frac{1+\gamma}{p}}_{p,p}(\partial \mathscr{O};X).
\]
If $\mathscr{A} \in \{B,F\}$, $s > \frac{1+\gamma}{p}$ and $s>\max\big\{s,\left(\frac{1+\gamma}{p}-1\right)_++1-s\big\}$, then we have retractions
\[
\mathrm{tr}_{\partial\mathscr{O}}: \mathscr{A}^{s}_{p,q}(\R^{n},w^{\partial\mathscr{O}}_{\gamma};X) \longra \partial\mathscr{A}^{s}_{p,q,\gamma}(\partial\mathscr{O};X)
\]
and
\[
\mathrm{Tr}_{\partial\mathscr{O}}: \mathscr{A}^{s}_{p,q,\gamma}(\mathscr{O};X) \longra \partial\mathscr{A}^{s}_{p,q,\gamma}(\partial\mathscr{O};X)
\]
that are related by $\mathrm{tr}_{\partial\mathscr{O}} = \mathrm{Tr}_{\partial\mathscr{O}} \circ \mathscr{E}$, where $\mathscr{E}$ is any choice of Rychkov's extension operator (from Theorem~\ref{DBVP:thm:Rychkov's_extension_operator}). There is compatibility for both of the trace operators $\mathrm{tr}_{\partial\mathscr{O}}$ and $\mathrm{Tr}_{\partial\mathscr{O}}$ on the different function spaces that are allowed above.

Let us now introduce reflexive Banach space-valued versions of the $\mathcal{B}$- and $\mathcal{F}$-scales, the scales dual to the $B$- and $F$-scales, respectively, as considered in \cite{Lindemulder2019_DSOE}.
Let $X$ be a reflexive Banach space, $p,q \in (1,\infty)$, $w \in [A_{\infty}]'_{p}(\R^{n})$ and $s \in \R$. Recall that $w'_{p} \in A_{\infty}$ by definition of $[A_{\infty}]'_{p}(\R^{n})$.
For $\mathscr{A} \in \{B,F\}$, $\mathscr{A}^{-s}_{p',q'}(\R^{n},w'_{p};X^{*})$ is a reflexive Banach space with
\[
\mathcal{S}(\R^{d};X) \stackrel{d}{\hookrightarrow} \mathscr{A}^{-s}_{p',q'}(\R^{n},w'_{p};X^{*}) \hookrightarrow \mathcal{S}'(\R^{n};X),
\]
so that
\[
\mathcal{S}(\R^{n};X) \stackrel{d}{\hookrightarrow} [\mathscr{A}^{-s}_{p',q'}(\R^{n},w'_{p};X^{*})]^{*} \hookrightarrow \mathcal{S}'(\R^{n};X)
\]
under the natural identifications.
We define
\[
\mathcal{B}^{s}_{p,q}(\R^{n},w;X) := [B^{-s}_{p',q'}(\R^{n},w'_{p};X^{*})]^{*} \quad\: \mbox{and} \quad\:
\mathcal{F}^{s}_{p,q}(\R^{n},w;X) := [F^{-s}_{p',q'}(\R^{n},w'_{p};X^{*})]^{*}.
\]
For $w \in A_{p}$ we have
\begin{equation}\label{DBVP:prelim:eq:dual_B&F;Ap}
\mathcal{B}^{s}_{p,q}(\R^{n},w;X) = B^{s}_{p,q}(\R^{n},w;X) \quad\: \mbox{and} \quad\:
\mathcal{F}^{s}_{p,q}(\R^{n},w;X) = F^{s}_{p,q}(\R^{n},w;X).
\end{equation}
Notationally it will be convenient to define
\[
\{B,F,\mathcal{B},\mathcal{F}\} \longra \{B,F,\mathcal{B},\mathcal{F}\},\, \mathscr{A} \mapsto \mathscr{A}^{\bullet} = \left\{\begin{array}{ll}
\mathcal{B},& \mathscr{A}=B,\\
\mathcal{F},& \mathscr{A}=F,\\
B,& \mathscr{A}=\mathcal{B},\\
F,& \mathscr{A}=\mathcal{F}.
\end{array}\right.
\]

Let $X$ be a reflexive Banach space, $p,q \in (1,\infty)$, $\gamma \in (-\infty,p-1)$ and $s \in \R$. We put
\[
\partial\mathcal{B}^{s}_{p,q,\gamma}(\partial\mathscr{O};X) := B^{s-\frac{1+\gamma}{p}}_{p,q}(\partial \mathscr{O};X)
\quad\quad\mbox{and}\quad\quad  \partial\mathcal{F}^{s}_{p,q,\gamma}(\partial \mathscr{O};X) :=  F^{s-\frac{1+\gamma}{p}}_{p,p}(\partial \mathscr{O};X).
\]
\\

\subparagraph{\emph{Parameter-dependent spaces}}
We now present an extension to the reflexive Banach space-valued setting of the parameter-dependent function spaces discussed in \cite[Section~6]{Lindemulder2019_DSOE}, which was in turn partly based on \cite{Grubb&Kokholm1993}. As the theory presented in \cite[Section~6]{Lindemulder2019_DSOE} carries over verbatim to this setting, we only state results without proofs. The reflexivity condition comes from duality arguments involving the dual scales that are needed outside the $A_{p}$-range.
Although for the $B$- and $F$-scales duality is only used in Corollary~\ref{DBVP:cor:prop:par-dep_trace;tensoring_delta}, for simplicity we restrict ourselves to the setting of reflexive Banach spaces from the start.

For $\sigma \in \R$ and $\mu \in [0,\infty)$ we define $\Xi^{\sigma}_{\mu} \in \mathcal{L}(\mathcal{S}(\R^{n};X)) \cap \mathcal{L}(\mathcal{S}'(\R^{n};X))$ by
\[
\Xi^{\sigma}_{\mu}f := \mathscr{F}^{-1}[\ip{\,\cdot\,}{\mu}^{\sigma}\hat{f}], \quad\quad f \in \mathcal{S}'(\R^{n};X),
\]
where $\ip{\xi}{\mu} = (1+|\xi|^{2}+\mu^{2})^{1/2}$.

Let $X$ be a reflexive Banach space and let either
\begin{itemize}
\item[(i)] $p \in [1,\infty)$, $q \in [1,\infty]$, $w \in A_{\infty}(\R^{n})$ and $\mathscr{A} \in \{B,F\}$; or
\item[(ii)] $p,q \in (1,\infty)$, $w \in [A_{\infty}]'_{p}(\R^{n})$ and $\mathscr{A} \in \{\mathcal{B},\mathcal{F}\}$.
\end{itemize}
For $s,s_{0} \in \R$ and $\mu \in [0,\infty)$ we define
\[
\norm{f}_{\mathscr{A}^{s,\mu,s_{0}}_{p,q}(\R^{n},w;X)} :=
\norm{\Xi^{s-s_{0}}_{\mu}f}_{\mathscr{A}^{s_{0}}_{p,q}(\R^{n},w;X)}, \quad\quad f \in \mathcal{S}'(\R^{n};X)
\]
and denote by $\mathscr{A}^{s,\mu,s_{0}}_{p,q}(\R^{n},w;X)$ the space $\{ f \in \mathcal{S}'(\R^{n};X) : \norm{f}_{\mathscr{A}^{s,\mu,s_{0}}_{p,q}(\R^{n},w;X)} < \infty \}$ equipped with this norm.
For the Bessel-potential scale we proceed in a similar way. Suppose that $X$ is a UMD Banach space and let $p \in (1,\infty)$ and $w\in A_p(\R^n)$. We define
\[
 \|f\|_{H^{s,\mu,s_0}_p(\R^n,w;X)}:=\| \Xi^{s-s_0}_\mu f\|_{H^{s_0}_p(\R^n,w;X)}
\]
and write $H^{s,\mu,s_0}_p(\R^n,w;X)$ for the space $\{ f \in \mathcal{S}'(\R^{n};X) : \norm{f}_{H^{s,\mu,s_{0}}_{p}(\R^{n},w;X)} < \infty \}$ endowed with this norm.

It trivially holds that
\begin{align}\begin{aligned}\label{DBVP:eq:par-dep_spaces_isometry_Bessel-pot}
\Xi^{t}_{\mu} &: \mathscr{A}^{s,\mu,s_{0}}_{p,q}(\R^{n},w;X) \stackrel{\simeq}{\longra}
\mathscr{A}^{s-t,\mu,s_{0}}_{p,q}(\R^{n},w;X),\quad \mbox{isometrically},\\
\Xi^{t}_{\mu} &: H^{s,\mu,s_{0}}_{p}(\R^{n},w;X) \stackrel{\simeq}{\longra}
H^{s-t,\mu,s_{0}}_{p}(\R^{n},w;X),\quad \mbox{isometrically}.
\end{aligned}
\end{align}
It furthermore holds that $\mathscr{A}^{s,\mu,s_{0}}_{p,q}(\R^{n},w;X) = \mathscr{A}^{s}_{p,q}(\R^{n},w;X)$ as well as $H^{s,\mu,s_{0}}_{p}(\R^{n},w;X) = H^{s}_{p}(\R^{n},w;X)$, but with an equivalence of norms that is $\mu$-dependent.
If $s,s_{0},\tilde{s}_{0} \in \R$ with $s_{0} \leq \tilde{s}_{0}$, then
\begin{align*}
\mathscr{A}^{s,\mu,s_{0}}_{p,q}(\R^{n},w;X) \hookrightarrow \mathscr{A}^{s,\mu,\tilde{s}_{0}}_{p,q}(\R^{n},w;X) \quad \mbox{uniformly in $\mu \in [0,\infty)$},\\
H^{s,\mu,s_{0}}_{p}(\R^{n},w;X) \hookrightarrow H^{s,\mu,\tilde{s}_{0}}_{p}(\R^{n},w;X) \quad \mbox{uniformly in $\mu \in [0,\infty)$}.
\end{align*}

For an open subset $U \subset \R^{n}$ we put
\[
\mathscr{A}^{s,\mu,s_{0}}_{p,q}(U,w;X) := [\mathscr{A}^{s,\mu,s_{0}}_{p,q}(\R^{n},w;X)](U),\quad H^{s,\mu,s_{0}}_{p}(U,w;X) := [H^{s,\mu,s_{0}}_{p}(\R^{n},w;X)](U).
\]
If $s \geq s_{0}$ and $\mathscr{O}$ is either $\R^{n}$, $\R^{n}_{+}$ or a Lipschitz domain in $\R^{n}$ with a compact boundary $\partial \mathscr{O}$, then it holds that
\begin{align}\begin{aligned}\label{DBVP:eq:par-dep;equiv_norm_parameter_expl;domain}
\norm{f}_{\mathscr{A}^{s,\mu,s_{0}}_{p,q}(\mathscr{O},w;X)} &\eqsim
\norm{f}_{\mathscr{A}^{s}_{p,q}(\mathscr{O},w;X)} + \langle \mu \rangle^{s-s_{0}}\norm{f}_{\mathscr{A}^{s_{0}}_{p,q}(\mathscr{O},w;X)},\\
\norm{f}_{H^{s,\mu,s_{0}}_{p}(\mathscr{O},w;X)} &\eqsim
\norm{f}_{H^{s}_{p}(\mathscr{O},w;X)} + \langle \mu \rangle^{s-s_{0}}\norm{f}_{H^{s_{0}}_{p}(\mathscr{O},w;X)},\quad\quad f \in \mathcal{S}'(\R^{n}), \mu \in [0,\infty)
\end{aligned}
\end{align}

Let $X$ be a reflexive Banach space, $p,q \in (1,\infty)$, $(w,\mathscr{A}) \in A_{\infty}(\R^{n}) \times \{B,F\} \cup [A_{\infty}]'_{p}(\R^{n}) \times \{\mathcal{B},\mathcal{F}\}$ and $\mathscr{B}=\mathscr{A}^{\bullet}$.
For $s,s_{0}$ it holds that
\[
[\mathscr{A}^{s,\mu,s_{0}}_{p,q}(\R^{n},w;X)]^{*} = \mathscr{B}^{-s,\mu,-s_{0}}_{p',q'}(\R^{n},w'_{p};X^{*}),
\quad \mbox{uniformly in $\mu \in [0,\infty)$}.
\]

Next we consider a vector-valued version of the parameter-dependent Besov spaces as introduced in \cite{Grubb&Kokholm1993}, but in the notation of \cite[Section~6]{Lindemulder2019_DSOE}.
Let $X$ be a reflexive Banach space, $p \in [1,\infty)$, $q \in [1,\infty]$ and $s \in \R$.
For each $\mu \in [0,\infty)$ the norm $\norm{\,\cdot\,}_{\B^{s,\mu}_{p,q}(\R^{n};X)}$ is defined by:
\[
\norm{f}_{\B^{s,\mu}_{p,q}(\R^{n};X)} :=
\langle \mu \rangle^{s-\frac{d}{p}}\norm{M_{\mu}f}_{B^{s}_{p,q}(\R^{n};X)},
\quad\quad f \in \mathcal{S}'(\R^{n};X),
\]
where $M_{\mu} \in \mathcal{L}(\mathcal{S}(\R^{n};X)) \cap \mathcal{L}(\mathcal{S}'(\R^{n};X))$ denotes the operator of dilation by $\langle \mu \rangle^{-1}$.
We furthermore write $\B^{s,\mu}_{p,q}(\R^{n};X)$ for the space $\{ f \in \mathcal{S}'(\R^{n};X) : \norm{f}_{\B^{s,\mu}_{p,q}(\R^{n};X)} < \infty \}$ equipped with this norm.
Then
\begin{equation}\label{DBVP:eq:par-dep_Besov;lifting}
\Xi^{t}_{\mu} : \B^{s,\mu}_{p,q}(\R^{n};X) \stackrel{\simeq}{\longra} \B^{s-t,\mu}_{p,q}(\R^{n};X), \quad \mbox{uniformly in $\mu$}.
\end{equation}
If $s>0$, then it holds that
\begin{equation}\label{DBVP:eq:par-dep_Besov;sum}
\norm{f}_{\B^{s,\mu}_{p,q}(\R^{n};X)} \eqsim \norm{f}_{B^{s}_{p,q}(\R^{n};X)} + \langle \mu \rangle^{s}\norm{f}_{L_{p}(\R^{n};X)}, \quad\quad f \in \mathcal{S}'(\R^{n};X), \mu \in [0,\infty).
\end{equation}
If $p \in (1,\infty)$ and $q \in [1,\infty)$, then
\begin{equation}\label{DSOE:eq:duality_par-dep_B}
[\B^{s,\mu}_{p,q}(\R^{n};X)]^{*} = \B^{-s,\mu}_{p',q'}(\R^{n};X^{*}) \qquad \text{uniformly in $\mu \in [0,\infty)$.}
\end{equation}

For a compact smooth manifold $M$ we define $\B^{s,\mu}_{p,q}(M)$ in terms of $\B^{s,\mu}_{p,q}(\R^{n})$ in the standard way. Then the analogues of \eqref{DBVP:eq:par-dep_Besov;sum} and \eqref{DSOE:eq:duality_par-dep_B} for $\B^{s,\mu}_{p,q}(M)$ are valid.

It will be convenient to write
\[
\partial\mathscr{A}^{s,\mu}_{p,q,\gamma}(\partial\mathscr{O};X)
:= \left\{\begin{array}{ll}
\B^{s-\frac{1+\gamma}{p},\mu}_{p,q}(\partial\mathscr{O};X), & \mathscr{A} \in \{B,\mathcal{B}\}, \\
\B^{s-\frac{1+\gamma}{p},\mu}_{p,p}(\partial\mathscr{O};X), & \mathscr{A} \in \{F,\mathcal{F}\}
\end{array}\right.
\]
as well as
\[
 \partial H^{s,\mu}_{p,\gamma}(\partial\mathscr{O};X):=\B^{s-\frac{1+\gamma}{p},\mu}_{p,p}(\partial\mathscr{O};X).
\]

\begin{proposition}\label{DBVP:prop:par-dep_trace}
Let $X$ be a reflexive Banach space, let $\mathscr{O}$ be either $\R^{n}_{+}$ or a $C^{N}$-domain in $\R^{n}$ with a compact boundary $\partial \mathscr{O}$, where $N \in \N$, let $U \in \{\R^{n},\mathscr{O}\}$, let either
\begin{enumerate}[(i)]
\item\label{it:DBVP:prop:par-dep_trace;i} $p \in [1,\infty)$, $q \in [1,\infty]$, $\gamma \in (-1,\infty)$ and $\mathscr{A} \in \{B,F\}$; or
\item\label{it:DBVP:prop:par-dep_trace;ii} $p,q \in (1,\infty)$, $\gamma \in (-\infty,p-1)$ and $\mathscr{A} \in \{\mathcal{B},\mathcal{F}\}$,
\end{enumerate}
and let $s \in (\frac{1+\gamma}{p},\infty)$ and $s_{0} \in (-\infty,\frac{1+\gamma}{p})$.
Assume that 
$$
N> \begin{cases}
\max\big\{s,\left(\frac{1+\gamma}{p}-1\right)_++1-s_0\big\}, & \text{in case \eqref{it:DBVP:prop:par-dep_trace;i}},\\
\max\big\{-s_0,-\left(\frac{1+\gamma}{p}\right)_-+1+s\big\}, & \text{in case \eqref{it:DBVP:prop:par-dep_trace;ii}}.
\end{cases}
$$
Then
\[
\mathrm{tr}_{\partial\mathscr{O}} : \mathscr{A}^{s,\mu,s_{0}}_{p,q}(U,w^{\partial\mathscr{O}}_{\gamma};X)
\longra \partial\mathscr{A}^{s,\mu}_{p,q,\gamma}(\partial\mathscr{O};X) \quad\quad \mbox{uniformly in $\mu \in [0,\infty)$},
\]
that is,
\[
\norm{\mathrm{tr}_{\partial\mathscr{O}}f}_{\partial\mathscr{A}^{s,\mu}_{p,q,\gamma}(\partial\mathscr{O};X)}
\lesssim \norm{f}_{\mathscr{A}^{s,\mu,s_{0}}_{p,q}(U,w^{\partial\mathscr{O}}_{\gamma};X)},
\quad\quad f \in \mathscr{A}^{s,\mu,s_{0}}_{p,q}(U,w^{\partial\mathscr{O}}_{\gamma};X), \mu \in [0,\infty).
\]
The respective assertion also holds for the Bessel potential scale if $X$ is a UMD Banach space and if $\gamma\in(-1,p-1)$ and $N>\max\{s,-s_0\}$.
\end{proposition}

\begin{corollary}\label{DBVP:cor:prop:par-dep_trace;tensoring_delta}
Let $X$ be a reflexive Banach space, $p,q \in (1,\infty)$, $(\gamma,\mathscr{A}) \in (-1,\infty) \times \{B,F\} \cup (-\infty,p-1) \times \{\mathcal{B},\mathcal{F}\}$,
$s \in (-\infty,\frac{1+\gamma}{p}-1)$ and $s_{0} \in (\frac{1+\gamma}{p}-1,\infty)$.
Then
\[
\norm{\delta_{0} \otimes f}_{\mathscr{A}^{s,\mu,s_{0}}_{p,q}(\R^{n},w_{\gamma};X)}
\lesssim \norm{f}_{\partial\mathscr{A}^{s+1,\mu}_{p,q,\gamma}(\R^{d-1};X)}, \quad\quad f \in \partial\mathscr{A}^{s+1,\mu}_{p,q,\gamma}(\R^{d-1};X), \mu \in [0,\infty).
\]
\end{corollary}

\subsection{Differential Boundary Value Systems}\label{Boutet:subsec:sec:prelim:E&LS}

\subsubsection{The Equations} 

Here we introduce some of the notation and terminology that will be used in Sections \ref{Boutet:sec:parabolic} and \ref{Boutet:sec:elliptic} on parabolic and elliptic boundary value problems.

Let $X$ be a Banach space, $\dom \subset \R^{n}$ a $C^{\infty}$-domain with a compact boundary $\partial\dom$ and $J \subset \R$ an interval. Let $m \in \N_{1}$ and let $m_{1},\ldots,m_{m} \in \N$ satisfy $m_{i} \leq 2m-1$ for each $i \in \{1,\ldots,m\}$.

\emph{Systems on $\dom$:}
Consider
\[
\begin{split}
\mathcal{A}(D) &= \sum_{|\alpha| \leq 2m}a_{\alpha}D^{\alpha}, \\
\mathcal{B}_{i}(D) &= \sum_{|\beta| \leq m_{i}}b_{i,\beta}\mathrm{tr}_{\partial\mathscr{O}}D^{\beta}, \quad\quad i=1,\ldots,m,\\
\end{split}
\]
with variable $\mathcal{B}(X)$-valued coefficients $a_{\alpha}$ on $\mathscr{O}$ and $b_{i,\beta}$ on $\partial\mathscr{O}$ which we are going to specify later in Section \ref{Boutet:sec:elliptic}. For the moment we just assume the top order coefficients $a_{\alpha}$, $|\alpha|=2m$, and $b_{i,\beta}$, $|\beta|=m_i$, to be bounded and uniformly continuous.  
We call $(\mathcal{A}(D),\mathcal{B}_{1}(D),\ldots,\mathcal{B}_{m}(D))$ a $\mathcal{B}(X)$-valued boundary value system (of order $2m$) on $\dom$.

\emph{Systems on $\dom \times J$:}
Consider
\[
\begin{split}
\mathcal{A}(D) &= \sum_{|\alpha| \leq 2m}a_{\alpha}D^{\alpha}, \\
\mathcal{B}_{i}(D) &= \sum_{|\beta| \leq m_{i}}b_{i,\beta}\mathrm{tr}_{\partial\mathscr{O}}D^{\beta}, \quad\quad i=1,\ldots,n,\\
\end{split}
\]
with variable $\mathcal{B}(X)$-valued coefficients $a_{\alpha}$ on $\mathscr{O}\times J$  and $b_{i,\beta}$ on $\partial\mathscr{O}\times J$ which we are going to specify later in Section \ref{Boutet:sec:parabolic}. We call $(\mathcal{A}(D),\mathcal{B}_{1}(D),\ldots,\mathcal{B}_{m}(D))$ a $\mathcal{B}(X)$-valued boundary value system (of order $2m$) on $\mathscr{O}\times J$.

\subsubsection{Ellipticity and Lopatinskii-Shapiro Conditions}

Let us now turn to the two structural assumptions on $\mathcal{A},\mathcal{B}_{1},\ldots,\mathcal{B}_{m}$.
For each $\phi \in [0,\pi)$ we introduce the conditions $(\mathrm{E})_{\phi}$ and $(\mathrm{LS})_{\phi}$.

The condition $(\mathrm{E})_{\phi}$ is parameter ellipticity.
In order to state it, we denote by the subscript $\#$ the principal part of a differential operator: given a differential operator $P(D)=\sum_{|\gamma| \leq k}p_{\gamma}D^{\gamma}$ of order $k \in \N$, $P_{\#}(D) = \sum_{|\gamma| = k}p_{\gamma}D^{\gamma}$.

\begin{itemize}
\item[\textbf{$(\mathrm{E})_{\phi}$}] For all $t \in \overline{J}$, $x \in \overline{\mathscr{O}}$ and $|\xi|=1$ it holds that $\sigma(\mathcal{A}_{\#}(x,\xi,t)) \subset \Sigma_{\phi}$. If $\mathscr{O}$ is unbounded, then it in addition holds that $\sigma(\mathcal{A}_{\#}(\infty,\xi,t)) \subset \C_{+}$ for all $t \in \overline{J}$ and $|\xi|=1$.
\end{itemize}
By $\mathcal{A}_{\#}(\infty,\xi,t)$ we mean that the limit $\lim_{|x|\to\infty}\mathcal{A}_{\#}(x,\xi,t)$ exists for all $t\in\overline{J}$ and all $|\xi|=1$ and that $\mathcal{A}_{\#}(\infty,\xi,t)$ is defined as this limit.

The condition $(\mathrm{LS})_{\phi}$ is a condition of Lopatinskii-Shapiro type. Before we can state it, we need to introduce some notation. For each $x \in \partial\mathscr{O}$ we fix an orthogonal matrix $O_{\nu(x)}$ that rotates the outer unit normal $\nu(x)$ of $\partial\mathscr{O}$ at $x$ to $(0,\ldots,0,-1) \in \R^{n}$, and define the rotated operators $(\mathcal{A}^{\nu},\mathcal{B}^{\nu})$ by
\[
\mathcal{A}^{\nu}(x,D,t) := \mathcal{A}(x,O_{\nu(x)}^{T}D,t), \quad
\mathcal{B}^{\nu}(x,D,t) := \mathcal{B}(x,O_{\nu(x)}^{T}D,t).
\]

\begin{itemize}
\item[\textbf{$(\mathrm{LS})_{\phi}$}] For each $t \in \overline{J}$, $x \in \partial\mathscr{O}$, $\lambda \in \overline{\Sigma}_{\pi-\phi}$ and $\xi' \in \R^{n-1}$ with $(\lambda,\xi') \neq 0$ and all $h \in X^{m}$, the ordinary initial value problem
    \begin{equation*}\label{pbvp:eqOIVP.(I)}
\begin{array}{rlll}
\lambda w(y) + \mathcal{A}^{\nu}_{\#}(x,\xi',D_{y},t)w(y) &= 0, & y > 0 & \\
\mathcal{B}^{\nu}_{j,\#}(x,\xi',D_{y},t)w(y)|_{y=0} &= h_{j}, & j=1,\ldots,m.
\end{array}
\end{equation*}
has a unique solution $w \in C^{\infty}([0,\infty);X)$ with $\lim_{y \to \infty}w(y)=0$.
\end{itemize}
In the scalar-valued case, there are several equivalent characterizations for the Lopatinskii-Shapiro condition. It is a common approach to consider the polynomial
\[
 \mathcal{A}_{\#}^{\nu,+}(x,\xi,\tau,t):=\prod_{j=1}^m (\tau-\tau_j(x,\xi',t))
\]
where $\tau_1(x,\xi',t),\ldots,\tau_m(x,\xi',t)$ are the roots of the polynomial $\mathcal{A}^{\nu}_{\#}(x,\xi',\,\cdot\,,t)$ with positive imaginary part. If we write $\overline{\mathcal{B}^{\nu}_{j,\#}}(x,\xi',\tau,t)$ for the equivalence classes of $\mathcal{B}^{\nu}_{j,\#}(x,\xi',\tau,t)$ in $\C[\tau]/(\mathcal{A}_{\#}^{\nu,+}(x,\xi,\tau,t))$, then we can formulate the following result:
\begin{proposition}\label{Boutet:prop:LopatinskiiCover1}
 The Lopatinskii-Shapiro condition is satisfied if and only if $\overline{\mathcal{B}^{\nu}_{j,\#}}(x,\xi',\tau,t)$ $(j=1,\ldots,m)$ are linearly independent in $\C[\tau]/(\mathcal{A}_{\#}^{\nu,+}(x,\xi,\tau,t))$.
\end{proposition}
This condition is sometimes called covering condition. A proof for this statement can for example be found in Chapter 3.2 of \cite{Roitberg1996}. A similar condition can be formulated using the so-called Lopatinskii matrix. If $\tilde{\mathcal{B}^{\nu}_{j,\#}}(x,\xi',\tau,t)$ are the representatives of $\overline{\mathcal{B}^{\nu}_{j,\#}}(x,\xi',\tau,t)$ with minimal degree, then their degree is smaller than $m$. Hence, there is a unique matrix $L(x,\xi',t)\in \C^{m\times m}$ such that
\[
 \begin{pmatrix}
  \tilde{\mathcal{B}^{\nu}_{1,\#}}(x,\xi',\tau,t) \\ \vdots \\ \tilde{\mathcal{B}^{\nu}_{m,\#}}(x,\xi',\tau,t)
 \end{pmatrix}= L(x,\xi',t)  \begin{pmatrix}
  \tau^0 \\ \vdots \\\tau^{m-1}.
 \end{pmatrix}
\]
This matrix $L(x,\xi',t)$ is called Lopatinskii matrix. From Proposition \ref{Boutet:prop:LopatinskiiCover1} one can easily derive the following result:
\begin{proposition}\label{Boutet:prop:LopatinskiiCover}
 The Lopatinskii-Shapiro condition is satisfied if and only if the Lopatinskii matrix $L(x,\xi',t)$ is invertible.
\end{proposition}
Using Proposition \ref{Boutet:prop:LopatinskiiCover} one can easily see that if $B_j(x,\xi',\tau,t)=\tau^{j-1}$, then the Lopatinskii-Shapiro condition is satisfied for all elliptic operators. In particular, this includes the usual Dirichlet boundary conditions for second order equations. Also Neumann boundary conditions satisfy the Lopatinskii-Shapiro condition. For further examples we refer to Section~11.2 in \cite{Wloka1987}.


\section{Embedding and Trace Results for Mixed-norm Anisotropic Spaces}\label{Boutet:sec:emb}

\subsection{Embedding Results}

\begin{proposition}\label{PEBVP:prop:Sobolev_embedding_Besov}
Let $X$ be a Banach space, $\vec{p},\tilde{\vec{p}} \in (1,\infty)^{l}$, $q,\tilde{q} \in [1,\infty]$, $s,\tilde{s} \in \R$, $\vec{a} \in (0,\infty)^{l}$, and $\vec{w},\tilde{\vec{w}} \in \prod_{j=1}^{l}A_{\infty}(\R^{\mathpzc{d}_{j}})$.
Suppose that
\begin{itemize}
\item $p_{1} \leq \tilde{p}_{1}$, $p_{j} = \tilde{p}_{j}$ and $w_{j} = \tilde{w}_{j}$ for $j \in \{2,\ldots,l\}$;
\item $w_{1}(x_{1}) = |x_{1}|^{\gamma_{1}}$ and $\tilde{w}_{1}(x_{1}) = |x_{1}|^{\tilde{\gamma}_{j}}$ for some $\gamma_{1},\tilde{\gamma}_{1} \in (-\mathpzc{d}_{1},\infty)$ satisfying
\[
\frac{\tilde{\gamma}_{1}}{\tilde{p}_{1}} \leq \frac{\gamma_{1}}{p_{1}} \quad \mbox{and} \quad
\frac{\mathpzc{d}_{1}+\tilde{\gamma}_{1}}{\tilde{p}_{1}} < \frac{\mathpzc{d}_{1}+\gamma_{1}}{p_{1}}.
 \]
\end{itemize}
If $s- a_{1}\frac{\mathpzc{d}_{1}+\gamma_{1}}{p_{1}} \geq \tilde{s} - a_{1}\frac{\mathpzc{d}_{1}+\tilde{\gamma}_{1}}{\tilde{p}_{1}}$, then
\[
F^{s,\vec{a}}_{\vec{p},q,\mathpzc{d}}(\R^{n},\vec{w};X) \hookrightarrow F^{\tilde{s},\vec{a}}_{\tilde{\vec{p}},\tilde{q},\mathpzc{d}}(\R^{n},\tilde{\vec{w}};X).
\]
\end{proposition}

 \begin{remark}\label{PEBVP:rmk:p=1}
  Proposition \ref{PEBVP:prop:Sobolev_embedding_Besov} is a partial extension of \cite[Theorem~1.2]{Meyries&Veraar2012_sharp_embedding} to the mixed-norm anisotropic setting. In the scalar-valued setting, there also is a generalization of \cite[Theorem~1.2]{Meyries&Veraar2012_sharp_embedding} to the general $A_{\infty}$ case, see \cite[Theorem~1.2]{Meyries&Veraar2014_char_class_embeddings}. It would be nice to have a similar extension also in the Banach space-valued anisotropic mixed-norm setting. This would also remove the restriction that $\vec{p},\tilde{\vec{p}} \in [1,\infty)^l\setminus(1,\infty)^{l}$ is not allowed in Proposition~\ref{PEBVP:prop:Sobolev_embedding_Besov}.
 \end{remark}

\begin{remark}\label{PEBVP:rmk:prop:Sobolev_embedding_Besov}
In this paper we only apply Proposition~\ref{PEBVP:prop:Sobolev_embedding_Besov} in the case that $\vec{p}=\tilde{\vec{p}}$. In this case the embedding result takes the form: if $ \gamma_{1} > \tilde{\gamma}_{1}$ and $s \geq \tilde{s}+a_{1}\frac{\gamma_{1}-\tilde{\gamma}_{1}}{p_{1}}$, then
\[
F^{s,\vec{a}}_{\vec{p},q,\mathpzc{d}}(\R^{n},\vec{w};X) \hookrightarrow F^{\tilde{s},\vec{a}}_{\vec{p},\tilde{q},\mathpzc{d}}(\R^{n},\tilde{\vec{w}};X).
\]
One of the nice things about this embedding, which has already turned out to be a powerful technical tool in the isotropic case (see e.g.\ \cite{Lindemulder2019_DSOE,Lindemulder2018_DSOP,Lindemulder&Veraar_boundary_noise,
Meyries&Veraar2014_traces}), is the \emph{(inner) trace space invariance} in the sharp case $s = \tilde{s}+a_{1}\frac{\gamma_{1}-\tilde{\gamma}_{1}}{p_{1}}$, see Proposition~\ref{Boutet:prop:trace_thm} below. In the two other embedding results in this section, Lemmas \ref{DSOP:lemma:embd_anisotropic_MR} and \ref{DSOP:lemma:embd_MR_anisotropic} below, there also is such an invariance.
 \end{remark}

\begin{proof}[Proof of Proposition~\ref{PEBVP:prop:Sobolev_embedding_Besov}]
The embedding can be proved in the same way as \cite[Theorem~1.2 (2)$\ra$(1)]{Meyries&Veraar2012_sharp_embedding}, as follows.
It suffices to consider the case $\tilde{q}=1$ and $s- a_{1}\frac{\mathpzc{d}_{1}+\gamma_{1}}{p_{1}} = \tilde{s} - a_{1}\frac{\mathpzc{d}_{1}+\tilde{\gamma}_{1}}{\tilde{p}_{1}}$. Furthermore, in order to prove the norm estimate corresponding to the embedding we may restrict ourselves to $f \in \mathcal{S}(\R^{n};X)$.
Let $\theta \in (0,1)$ be such that
\[
\nu := \frac{\tilde{\gamma}_{1}/\tilde{p}_{1}-(1-\theta)\gamma_{1}/p_{1}}{\frac{1}{\tilde{p}_{1}}-\frac{1-\theta}{p_{1}}} > -\mathpzc{d}_{1},
\]
let $r$ be defined by $\frac{1}{\tilde{p}_{1}}=\frac{1-\theta}{p_{1}}+\frac{\theta}{r}$ and let $t$ be defined by $t-a_{1}\frac{\mathpzc{d}_{1}+\nu}{r}=\tilde{s} - a_{1}\frac{\mathpzc{d}_{1}+\tilde{\gamma}_{1}}{\tilde{p}_{1}}$.
Note that $r \in [\tilde{p}_{1},\infty)$, $t \in (-\infty,s)$, $\tilde{s}=\theta t + (1-\theta)s$ and $\frac{\theta \tilde{p}_{1}}{r}\nu+\frac{(1-\theta)\tilde{p}_{1}}{p_{1}}\gamma_{1} = \tilde{\gamma}$.
Therefore, as \cite[Proposition~5.1]{Meyries&Veraar2012_sharp_embedding} directly extends to the setting of mixed-norm anisotropic Triebel-Lizorkin spaces,
\begin{equation}\label{PEBVP:eq:prop:Sobolev_embedding_Besov;int_ineq}
\norm{f}_{F^{\tilde{s},\vec{a}}_{\tilde{\vec{p}},1,\mathpzc{d}}(\R^{n},\tilde{\vec{w}};X)}
\lesssim \norm{f}_{F^{s,\vec{a}}_{\vec{p},q,\mathpzc{d}}(\R^{n},\vec{w};X)}^{1-\theta}
\norm{f}_{F^{t,\vec{a}}_{(r,\vec{p}'),r,\mathpzc{d}}(\R^{n},(|\,\cdot\,|^{\nu},\tilde{\vec{w}}');X)}^{\theta}.
\end{equation}
Furthermore, as a consequence of \cite[Proposition~4.1]{Meyries&Veraar2012_sharp_embedding}, since
\[
\frac{\tilde{\gamma}_{1}}{\tilde{p}_{1}} - \frac{\nu}{r} = \frac{1-\theta}{\theta}\left( \frac{\gamma_{1}}{p_{1}} - \frac{\tilde{\gamma}_{1}}{\tilde{p}_{1}} \right) \geq 0,
\]
we obtain that
\begin{align}
\norm{f}_{F^{t,\vec{a}}_{(r,\vec{p}'),r,\mathpzc{d}}(\R^{n},(|\,\cdot\,|^{\nu},\tilde{\vec{w}}');X)}
&= \norm{(2^{tk}S^{(\mathpzc{d},\vec{a})}_{k}f)_{k \in \N}}_{L_{\vec{p}',\mathpzc{d}'}(\R^{n-\mathpzc{d}_{1}},\tilde{\vec{w}}')[\ell_{r}(\N)[L_{r}(\R^{\mathpzc{d}_{1}},|\,\cdot\,|^{\nu})]](X)} \nonumber \\
&\lesssim \norm{(2^{\tilde{s}k}S^{(\mathpzc{d},\vec{a})}_{k}f)_{k \in \N}}_{L_{\vec{p}',\mathpzc{d}'}(\R^{n-\mathpzc{d}_{1}},\tilde{\vec{w}}')[\ell_{\tilde{p}_{1}}(\N)[L_{\tilde{p}_{1}}(\R^{\mathpzc{d}_{1}},|\,\cdot\,|^{\tilde{\gamma}_{1}})]](X)} \nonumber \\
&= \norm{f}_{F^{\tilde{s},\vec{a}}_{\tilde{\vec{p}},1,\mathpzc{d}}(\R^{n},\tilde{\vec{w}};X)}. \label{PEBVP:eq:prop:Sobolev_embedding_Besov;embd}
\end{align}
Indeed, here we apply \cite[Proposition~4.1]{Meyries&Veraar2012_sharp_embedding} to $S^{(\mathpzc{d},\vec{a})}_{k}f(\,\cdot\,,x')$ for each $x' \in \R^{d-\mathpzc{d}_{1}}$, which is a Schwartz function with Fourier support in $[-c2^{ka_{1}},c2^{ka_{1}}]^{\mathpzc{d}_{1}}$ (with $c$ independent of $f$ and $n$), to obtain
\begin{align*}
\norm{S^{(\mathpzc{d},\vec{a})}_{k}f(\,\cdot\,,x')}_{L_{r}(\R^{\mathpzc{d}_{1}},|\,\cdot\,|^{\nu};X)}
&\lesssim (2^{ka_{1}})^{\frac{\mathpzc{d}_{1}+\nu}{r}-\frac{\mathpzc{d}_{1}+\tilde{\gamma}_{1}}{\tilde{p}_{1}}}\norm{S^{(\mathpzc{d},\vec{a})}_{k}f(\,\cdot\,,x') }_{L_{\tilde{p}_{1}}(\R^{\mathpzc{d}_{1}},|\,\cdot\,|^{\tilde{\gamma}_{1}};X)} \\
&= 2^{k(t-\tilde{s})}\norm{S^{(\mathpzc{d},\vec{a})}_{k}f(\,\cdot\,,x') }_{L_{\tilde{p}_{1}}(\R^{\mathpzc{d}_{1}},|\,\cdot\,|^{\tilde{\gamma}_{1}};X)},
\end{align*}
from which '$\lesssim$' in \eqref{PEBVP:eq:prop:Sobolev_embedding_Besov;embd} follows.
A combination of \eqref{PEBVP:eq:prop:Sobolev_embedding_Besov;int_ineq} and \eqref{PEBVP:eq:prop:Sobolev_embedding_Besov;embd} gives the desired estimate
\[
\norm{f}_{F^{\tilde{s},\vec{a}}_{\tilde{\vec{p}},1,\mathpzc{d}}(\R^{n},\tilde{\vec{w}};X)} \lesssim \norm{f}_{F^{s,\vec{a}}_{\vec{p},q,\mathpzc{d}}(\R^{n},\vec{w};X)}. \qedhere
\]
\end{proof}

\begin{lemma}\label{DSOP:lemma:embd_anisotropic_MR}
Let $X$ be a UMD Banach space, $q,p,r \in (1,\infty)$, $v \in A_{q}(\R)$, $\gamma \in (-1,\infty)$, $s \in \R$ and $\rho \in (0,\infty)$.
Let $\tilde{s}\in(0,\infty)$ with $\tilde{s}\geq s$, $\tilde{\gamma}:=\gamma+(\tilde{s}-s)p$ and $\sigma:= \frac{\tilde{s}}{\rho}+1$.
Let $\delta \in (0,\infty)$ be such that $\tilde{\gamma}-\delta p \in (-1,p-1)$ and put $\eta:= \frac{1}{\sigma-1}\delta$. Then
\begin{align} F^{\sigma+\frac{\eta}{\rho},(\frac{1}{\rho},1)}_{(p,q),1}(\R^{n}_{+} \times \R,(w_{\tilde{\gamma}+\eta p},v);X) \hookrightarrow W^{1}_{q}(\R,v;F^{s}_{p,r}(\R^{n}_{+},w_{\gamma};X)) \cap L_{q}(\R,v;F^{s+\rho}_{p,r}(\R^{n}_{+},w_{\gamma};X)). \label{DSOP:eq:lemma:embd_anisotropic_MR}
\end{align}
\end{lemma}
\begin{proof}
This can be shown as the scalar-valued case in \cite[(27)]{Lindemulder2018_DSOP}. Note that the duality arguments therein remain valid as $X$ is a UMD space and therefore reflexive. 

\end{proof}

\begin{lemma}\label{DSOP:lemma:embd_MR_anisotropic}
Let $X$ be a UMD Banach space, $q,p \in (1,\infty)$, $v \in A_{q}(\R)$, $\gamma \in (-1,\infty)$, $s \in \R$ and $\rho \in (0,\infty)$. If $\theta \in [0,1]$ satsifies $s+\theta\rho \in (0,\infty) \cap (\frac{1+\gamma}{p}-1,\frac{1+\gamma}{p})$, then
\begin{align}
&W^{1}_{q}(\R,v;F^{s}_{p,\infty}(\R^{n}_{+},w_{\gamma});X) \cap L_{q}(\R,v;F^{s+\rho}_{p,\infty}(\R^{n}_{+},w_{\gamma});X) \nonumber \\
&\qquad\qquad \hookrightarrow \quad H^{1-\theta}_{q}(\R,v;L_{p}(\R^{n}_{+},w_{\gamma-(s+\theta\rho)p});X) \cap L_{q}(\R,v;H^{(1-\theta)\rho}_{p}(\R^{n}_{+},w_{\gamma-(s+\theta\rho)p});X)
\label{DSOP:eq:lemma:embd_MR_anisotropic}.
\end{align}
\end{lemma}
Note that $s+\theta\rho \in (\frac{1+\gamma}{p}-1,\frac{1+\gamma}{p})$ is equivalent to $\gamma-(s+\theta\rho)p \in (-1,p-1)$, which is in turn equivalent to $w_{\gamma-(s+\theta\rho)p} \in A_{p}$.
\begin{proof}
The proof given in \cite[Lemma~3.4]{Lindemulder2018_DSOP} on the scalar-valued case carries over verbatim.
\end{proof}

\subsection{Trace Results}

Proposition~\ref{PEBVP:prop:Sobolev_embedding_Besov} with $\vec{p} = \tilde{\vec{p}}$ (see Remark~\ref{PEBVP:rmk:prop:Sobolev_embedding_Besov}) enables us to give an alternative proof of the trace theorem \cite[Theorem~4.6]{Lindemulder2017_PIBVP} for anisotropic weighted mixed-norm Triebel-Lizorkin spaces. The special case $\mathpzc{d}_{1}=1$ in Proposition~\ref{Boutet:prop:trace_thm} actually yields \cite[Theorem~4.6]{Lindemulder2017_PIBVP}, which is the only case that is used in this paper.

For the statement of Proposition~\ref{Boutet:prop:trace_thm} we need some notation and terminology that we first introduce.

\subsubsection{Some notation}

We slightly modify the notation from \cite[Sections 4.3.1 $\&$ 4.3.2]{Lindemulder2017_PIBVP} to our setting.

\paragraph{The working definition of the trace}

Let $\varphi \in \Phi^{\mathpzc{d},a}(\R^{n})$ with associated family of convolution operators $(S_{k})_{k \in \N} \subset \mathcal{L}(\mathcal{S}'(\R^{n};X))$ be fixed.
In order to motivate the definition to be given in a moment, let us first recall that $f = \sum_{k=0}^{\infty}S_{k}f$ in $\mathcal{S}(\R^{n};X)$ (respectively in $\mathcal{S}'(\R^{n};X)$) whenever $f \in \mathcal{S}(\R^{n};X)$ (respectively $f \in \mathcal{S}'(\R^{n};X)$), from which it is easy to see that
\[
f_{|\{0_{\mathpzc{d}_{1}}\}  \times \R^{n-\mathpzc{d}_{1}}} = \sum_{k=0}^{\infty}(S_{k}f)_{|\{0_{\mathpzc{d}_{1}}\}  \times \R^{n-\mathpzc{d}_{1}}} \:\:\mbox{in}\:\:\mathcal{S}(\R^{n-\mathpzc{d}_{1}};X), \quad\quad f \in \mathcal{S}(\R^{n};X).
\]
Furthermore, given a general tempered distribution $f \in \mathcal{S}'(\R^{n};X)$, recall that $S_{k}f \in \mathscr{O}_{M}(\R^{n};X)$. In particular, each $S_{k}f$ has a well-defined classical trace with respect to $\{0_{\mathpzc{d}_{1}}\}  \times \R^{n-\mathpzc{d}_{1}}$.
This suggests to define the trace operator $\tau = \tau^{\varphi}: \mathcal{D}(\gamma^{\varphi}) \subset \mathcal{S}'(\R^{n};X) \longra \mathcal{S}'(\R^{n-\mathpzc{d}_{1}};X)$ by
\begin{equation}\label{PIBVP:eq:working_def_trace}
\tau^{\varphi}f := \sum_{k=0}^{\infty}(S_{k}f)_{|\{0_{\mathpzc{d}_{1}}\}  \times \R^{n-\mathpzc{d}_{1}}}
\end{equation}
on the domain $\mathcal{D}(\tau^{\varphi})$ consisting of all $f \in \mathcal{S}'(\R^{n};X)$ for which this defining series
converges in $\mathcal{S}'(\R^{n-\mathpzc{d}_{1}};X)$. Note that $\mathscr{F}^{-1}\mathcal{E}'(\R^{n};X)$ is a subspace of $\mathcal{D}(\tau^{\varphi})$ on which $\tau^{\varphi}$ coincides with the classical trace of continuous functions with respect to $\{0_{\mathpzc{d}_{1}}\}  \times \R^{n-\mathpzc{d}_{1}}$; of course, for an $f$ belonging to $\mathscr{F}^{-1}\mathcal{E}'(\R^{n};X)$ there are only finitely many $S_{k}f$ non-zero.\\

\paragraph{The distributional trace operator}

Let us now introduce the concept of distributional trace operator.
The motivation for introducing it comes from Lemma~\ref{PIBVP:prop:right_inverse_distr_trace}.

The distributional trace operator $r$ (with respect to the plane $\{0_{\mathpzc{d}_{1}}\}  \times \R^{n-\mathpzc{d}_{1}}$) is defined as follows.
Viewing $C(\R^{\mathpzc{d}_{1}};\mathcal{D}'(\R^{n-\mathpzc{d}_{1}};X))$ as subspace of
$\mathcal{D}'(\R^{n};X) = \mathcal{D}'(\R^{\mathpzc{d}_{1}} \times \R^{n-\mathpzc{d}_{1}};X)$ via the canonical identification $\mathcal{D}'(\R^{\mathpzc{d}_{1}};\mathcal{D}'(\R^{n-\mathpzc{d}_{1}};X)) = \mathcal{D}'(\R^{\mathpzc{d}_{1}} \times \R^{n-\mathpzc{d}_{1}};X)$ (arising from the Schwartz kernel theorem),
\[
C(\R^{\mathpzc{d}_{1}};\mathcal{D}'(\R^{n-\mathpzc{d}_{1}};X)) \hookrightarrow  \mathcal{D}'(\R^{\mathpzc{d}_{1}};\mathcal{D}'(\R^{n-\mathpzc{d}_{1}};X)) = \mathcal{D}'(\R^{\mathpzc{d}_{1}} \times \R^{n-\mathpzc{d}_{1}};X),
\]
we define $r \in \mathcal{L}(C(\R^{\mathpzc{d}_{1}};\mathcal{D}'(\R^{n-\mathpzc{d}_{1}};X)),\mathcal{D}'(\R^{n-\mathpzc{d}_{1}};X))$ as the 'evaluation in $0$ map'
\begin{equation*}
r: C(\R^{\mathpzc{d}_{1}};\mathcal{D}'(\R^{n-\mathpzc{d}_{1}};X)) \longra \mathcal{D}'(\R^{n-\mathpzc{d}_{1}};X),\,f \mapsto \mathrm{ev}_{0}f.
\end{equation*}
Then, in view of
\[
C(\R^{n};X) = C(\R^{\mathpzc{d}_{1}} \times \R^{n-\mathpzc{d}_{1}};X) = C(\R^{\mathpzc{d}_{1}};C(\R^{n-\mathpzc{d}_{1}};X)) \hookrightarrow C(\R^{\mathpzc{d}_{1}};\mathcal{D}'(\R^{n-\mathpzc{d}_{1}};X)),
\]
we have that the distributional trace operator $r$ coincides on $C(\R^{n};X)$ with the classical trace operator with respect to the plane $\{0_{\mathpzc{d}_{1}}\} \times \R^{n-\mathpzc{d}_{1}}$, i.e.,
\begin{equation*}
r: C(\R^{n};X) \longra C(\R^{n-\mathpzc{d}_{1}};X),\,f \mapsto f_{| \{0_{\mathpzc{d}_{1}}\} \times \R^{n-\mathpzc{d}_{1}}}.
\end{equation*}

The following lemma can be established as in \cite[Section~4.2.1]{JS_traces}.
\begin{lemma}\label{PIBVP:prop:right_inverse_distr_trace}
Let $\rho \in \mathcal{S}(\R^{\mathpzc{d}_{1}})$ such that $\rho(0) = 1$ and $\supp \hat{\rho} \subset [1,2]^{\mathpzc{d}_{1}}$, $a_{1} \in \R^{\mathpzc{d}_{1}}$, $\tilde{\mathpzc{d}} \in \N_1^{l-\mathpzc{d}_{1}}$ with $\mathpzc{d}=(\mathpzc{d}_{1},\tilde{\mathpzc{d}})$, $\tilde{\vec{a}} \in (0,\infty)^{l-\mathpzc{d}_{1}}$, and
$(\phi_{n})_{n \in \N} \in \Phi^{\tilde{\mathpzc{d}},\tilde{\vec{a}}}(\R^{d-\mathpzc{d}_{1}})$.
Then, for each $g \in \mathcal{S}'(\R^{d-\mathpzc{d}_{1}};X)$,
\begin{equation}\label{functieruimten:eq:prop;right_inverse_distr_trace_formula}
\mathrm{ext}\,g := \sum_{k=0}^{\infty} \rho(2^{ka_{1}}\,\cdot\,) \otimes [\phi_{k}*g]
\end{equation}
defines a convergent series in $\mathcal{S}'(\R^{n};X)$ with
\begin{equation}\label{functieruimten:eq:prop;right_inverse_distr_trace_supports}
\begin{array}{l}
\supp \mathscr{F}[\rho \otimes [\phi_{0}*g]] \subset \{ \xi \mid |\xi|_{\mathpzc{d},a} \leq c \} \\
\supp \mathscr{F}[\rho(2^{ka_{1}}\,\cdot\,) \otimes [\phi_{k}*g]] \subset \{ \xi \mid c^{-1}2^{k} \leq |\xi|_{\mathpzc{d},a} \leq c2^{k} \}  \:\:, k \geq 1,
\end{array}
\end{equation}
for some constant $c>0$ independent of $g$.
Moreover, the operator $\mathrm{ext}$ defined via this formula is a linear operator
\[
\mathrm{ext}:\mathcal{S}'(\R^{d-\mathpzc{d}_{1}};X) \longra C_{b}(\R^{\mathpzc{d}_{1}};\mathcal{S}'(\R^{d-\mathpzc{d}_{1}};X))
\]
which acts as a right inverse of $r:C(\R^{\mathpzc{d}_{1}};\mathcal{S}'(\R^{d-\mathpzc{d}_{1}};X)) \longra \mathcal{S}'(\R^{d-\mathpzc{d}_{1}};X)$.
\end{lemma}

\subsubsection{The results}\label{Boutet:subsubsec:trace_results}

We will use the following notation. We write $\mathpzc{d}'=(\mathpzc{d}_{2},\ldots,\mathpzc{d}_{l})$.
Similarly, given $\vec{a} \in (0,\infty)^{l}$, $\vec{p} \in [1,\infty)^{l}$ and $\vec{w} \in \prod_{j=1}^{l}A_{\infty}(\R^{\mathpzc{d}_{j}})$, we write $\vec{a}' := (a_{2},\ldots,a_{l})$, $\vec{p}' := (p_{2},\ldots,p_{l})$ and $\vec{w}' := (w_{2},\ldots,w_{l})$.

\begin{proposition}\label{Boutet:prop:trace_thm}
Let $X$ be a Banach space, $\vec{a} \in (0,\infty)^{l}$, $\vec{p} \in (1,\infty)^{l}$, $q \in [1,\infty]$, $\gamma \in (-\mathpzc{d}_{1},\infty)$ and $s > \frac{a_{1}}{p_{1}}(\mathpzc{d}_{1}+\gamma)$. Let $\vec{w} \in \prod_{j=1}^{l}A_{\infty}(\R^{\mathpzc{d}_{j}})$ be such that
$w_{1}(x_{1}) = |x_{1}|^{\gamma}$ and $\vec{w}' \in \prod_{j=2}^{l}A_{p_{j}/r_{j}}(\R^{\mathpzc{d}_{j}})$ for some $\vec{r}'=(r_{2},\ldots,r_{l}) \in (0,1)^{l-1}$ satisfying $s-\frac{a_{1}}{p_{1}}(\mathpzc{d}_{1}+\gamma) > \sum_{j=2}^{l}a_{j}\mathpzc{d}_{j}(\frac{1}{r_{j}}-1)$.\footnote{This technical condition on $\vec{w}'$ is in particular satisfied when $\vec{w}' \in \prod_{j=2}^{l}A_{p_{j}}(\R^{\mathpzc{d}_{j}})$.}
Then the trace operator $\tau = \tau^{\varphi}$ \eqref{PIBVP:eq:working_def_trace} is well-defined on
$F_{\vec{p},q,\mathpzc{d}}^{s,\vec{a}}(\R^{n},\vec{w};X)$, where it is independent of $\varphi$, and
restricts to a retraction
\begin{equation}\label{PIBVP:eq:thm:trace_TL;mapping_prop}
\tau: F_{\vec{p},q,\mathpzc{d}}^{s,\vec{a}}(\R^{n},\vec{w};X) \longra
F_{\vec{p}',p_{1},\mathpzc{d}'}^{s-\frac{a_{1}}{p_{1}}(1+\gamma),\vec{a}'}(\R^{n-\mathpzc{d}_{1}},\vec{w}';X)
\end{equation}
for which the extension operator $\mathrm{ext}$ from Lemma \ref{PIBVP:prop:right_inverse_distr_trace} (with $\tilde{\mathpzc{d}} = \mathpzc{d}'$ and $\tilde{\vec{a}}= \vec{a}'$) restricts to a corresponding coretraction.

\end{proposition}
\begin{proof}
Using the Sobolev embedding from Proposition~\ref{PEBVP:prop:Sobolev_embedding_Besov} with $\vec{p} = \tilde{\vec{p}}$ (see Remark~\ref{PEBVP:rmk:prop:Sobolev_embedding_Besov}) in combination with the invariance of the space on the right-hand side of \eqref{PIBVP:eq:thm:trace_TL;mapping_prop} under this embedding,
we may without loss of generality assume that $p_{1}=q$.
So
\[
L_{\vec{p}}(\R^{n},\vec{w})[\ell_{q}(\N)] = L_{\vec{p}'}(\R^{n-\mathpzc{d}_{1}},\vec{w}'')[\ell_{q}(\N)[L_{p_{1}}(\R^{\mathpzc{d}_{1}},|\,\cdot\,|^{\gamma})]].
\]
Now the proof goes analogously to the proof of \cite[Theorem~5.2.52]{Lindemulder_master-thesis}.
\end{proof}

\begin{corollary}\label{Boutet:cor:prop:trace_thm}
Let $X$ be a Banach space, $\vec{a} \in (0,\infty)^{l}$, $\vec{p} \in (1,\infty)^{l}$, $\gamma \in (-\mathpzc{d}_{1},\mathpzc{d}_{1}(p_{1}-1))$ and $s > \frac{a_{1}}{p_{1}}(\mathpzc{d}_{1}+\gamma)$. Let $\vec{w} \in \prod_{j=1}^{l}A_{p_{j}}(\R^{\mathpzc{d}_{j}})$ be such that
$w_{1}(x_{1}) = |x_{1}|^{\gamma}$. Suppose that either
\begin{itemize}
\item $\E = W^{\vec{n}}_{\vec{p},\mathpzc{d}}(\R^{n},\vec{w};X)$, $\vec{n} \in (\Z_{\geq 1})^{l}$, $\vec{n}=s\vec{a}^{-1}$; or
\item $\E = H^{s,\vec{a}}_{\vec{p},\mathpzc{d}}(\R^{n},\vec{w};X)$; or
\item $\E = H^{\vec{s}}_{\vec{p},\mathpzc{d}}(\R^{n},\vec{w};X)$, $\vec{s} \in (0,\infty)^{l}$, $\vec{a}=s\vec{a}^{-1}$.
\end{itemize}
Then the trace operator $\tau = \tau^{\varphi}$ \eqref{PIBVP:eq:working_def_trace} is well-defined on
$\E$, where it is independent of $\varphi$, and restricts to a retraction
\[
\tau: \E \longra
F_{\vec{p}',p_{1},\mathpzc{d}'}^{s-\frac{a_{1}}{p_{1}}(1+\gamma),\vec{a}'}(\R^{d-1},\vec{w}';X)
\]
for which the extension operator $\mathrm{ext}$ from Lemma \ref{PIBVP:prop:right_inverse_distr_trace} (with $\tilde{\mathpzc{d}} = \mathpzc{d}'$ and $\tilde{\vec{a}}= \vec{a}'$) restricts to a corresponding coretraction.
\end{corollary}

\begin{corollary}\label{Boutet:lemma:DPBVP_boundedness_higher_order_traces}
Let $X$ be a UMD Banach space, $q,p \in (1,\infty)$, $v \in A_{q}(\R)$, $\gamma \in (-1,\infty)$, $s \in (-\infty,\frac{1+\gamma}{p})$, $\rho \in (0,\infty)$ and $\beta \in \N^{n}$. If $s+\rho-|\beta| > \frac{1+\gamma}{p}$, then $\mathrm{tr}_{\partial\R^{n}_{+}} \circ D^{\beta}$ is a bounded linear operator
\[
W^{1}_{q}(\R,v;F^{s}_{p,\infty}(\R^{n}_{+},w_{\gamma};X)) \cap L_{q}(\R,v;F^{s+\rho}_{p,\infty}(\R^{n}_{+},w_{\gamma};X)) \longra F^{\frac{1}{\rho}(s+\rho-|\beta|-\frac{1+\gamma}{p}),(\frac{1}{\rho},1)}_{(p,q),p}(\R^{n-1}\times\R,(1,v);X).
\]
\end{corollary}
\begin{proof}
Let $\theta \in [0,1]$ be such that $s+\theta\rho \in (0,\infty) \cap (\frac{1+\gamma}{p}-1,\frac{1+\gamma}{p})$. Such a $\theta$ exists because $s<\frac{1+\gamma}{p}$ and $s+\rho \geq s+\rho-|\beta|>\frac{1+\gamma}{p}$.
Using Lemma~\ref{DSOP:lemma:embd_MR_anisotropic} together with
\begin{align*}
H^{1-\theta,(\frac{1}{\rho},1)}_{(p,q)}(\R^{n}_{+} \times \R,(w_{\gamma},v);X) &\cap L_{q}(\R,v;H^{(1-\theta)\rho}_{p}(\R^{n}_{+},w_{\gamma-(s+\theta\rho)p});X) \\
&=
H^{1-\theta-\frac{1}{\rho}|\beta|,(\frac{1}{\rho},1)}_{(p,q)}(\R^{n}_{+}\times\R,(w_{\gamma-(s+\theta\rho)p},v);X),
\end{align*}
we find that $D^{\beta}$ is a bounded linear operator from
\[
W^{1}_{q}(\R,v;F^{s}_{p,\infty}(\R^{n}_{+},w_{\gamma};X)) \cap L_{q}(\R,v;F^{s+\rho}_{p,\infty}(\R^{n}_{+},w_{\gamma};X))
\]
to
\[
H^{1-\theta-\frac{1}{\rho}|\beta|,(\frac{1}{\rho},1)}_{(p,q)}(\R^{n}_{+}\times\R,(w_{\gamma-(s+\theta\rho)p},v);X).
\]
The desired result now follows from Corollary~\ref{Boutet:cor:prop:trace_thm}/\cite[Corollary~4.9]{Lindemulder2017_PIBVP} and the observation that
\[
\frac{1}{\rho}(s+\rho-|\beta|-\frac{1+\gamma}{p}) = (1-\theta-\frac{1}{\rho}|\beta|)-\frac{1+[\gamma-(s+\theta\rho)p]}{p}. \qedhere
\]
\end{proof}

Combined with the elementary embedding \eqref{eq:relation_Bessel-Potential_Triebel-Lizorkin;Ap}, the above corollary immediately leads to:

\begin{corollary}\label{Boutet:lemma:DPBVP_boundedness_higher_order_traces;H}
Let $X$ be a UMD Banach space, $q,p \in (1,\infty)$, $v \in A_{q}(\R)$, $\gamma \in (-1,p-1)$, $s \in (-\infty,\frac{1+\gamma}{p})$, $\rho \in (0,\infty)$ and $\beta \in \N^{n}$. If $s+\rho-|\beta| > \frac{1+\gamma}{p}$, then $\mathrm{tr}_{\partial\R^{n}_{+}} \circ D^{\beta}$ is a bounded linear operator
\[
W^{1}_{q}(\R,v;H^{s}_{p}(\R^{n}_{+},w_{\gamma};X)) \cap L_{q}(\R,v;H^{s+\rho}_{p}(\R^{n}_{+},w_{\gamma};X)) \longra F^{\frac{1}{\rho}(s+\rho-|\beta|-\frac{1+\gamma}{p}),(\frac{1}{\rho},1)}_{(p,q),p}(\R^{n-1}\times\R,(1,v);X).
\]
\end{corollary}

The following lemma allows us to derive a Besov space variant of Corollary~\ref{Boutet:lemma:DPBVP_boundedness_higher_order_traces} by real interpolation.
\begin{lemma}\label{Boutet:prelim:lemma_real_interpolation_max-reg_B}
Let $X$ be a UMD Banach space, $q,p \in (1,\infty)$, $r \in (1,\infty)$, $v \in A_{q}(\R)$, $w \in A_{\infty}(\R^{n})$, $s \in \R$ and $\rho \in (0,\infty)$.
Let $\dom$ be either $\R^{n}$ or as in Theorem~\ref{DBVP:thm:Rychkov's_extension_operator}.
Let $s_{0},s_{1} \in \R$ and $\theta \in (0,1)$ be such that $s=s_{0}(1-\theta)+s_{1}\theta$.
Then
\begin{align*}
&\Big(W^{1}_{q}(\R,v;F^{s_{0}}_{p,r}(\dom,w;X)) \cap L_{q}(\R,v;F^{s_{0}+\rho}_{p,r}(\dom,w;X)), \\
& \qquad W^{1}_{q}(\R,v;F^{s_{1}}_{p,r}(\dom,w;X)) \cap L_{q}(\R,v;F^{s_{1}+\rho}_{p,r}(\dom,w;X))\Big)_{\theta,q} \\
& \qquad\qquad\qquad\qquad= W^{1}_{q}(\R,v;B^{s}_{p,q}(\dom,w;X)) \cap L_{q}(\R,v;B^{s+\rho}_{p,q}(\dom,w;X))
\end{align*}
\end{lemma}
\begin{proof}
Thanks to Theorem~\ref{DBVP:thm:Rychkov's_extension_operator} we may restrict ourselves to the case $\dom = \R^{n}$. The result now follows from \eqref{PIBVP:eq:prelim:interpolation} with $\E=L_{q}(\R,v;F^{s_{0}}_{p,r}(\R^{n},w;X))$ and $\F=L_{q}(\R,v;F^{s_{1}}_{p,r}(\R^{n},w;X))$.
Indeed, $\E$ and $\F$ are $(\mathpzc{d},\vec{a})$-admissible with $\mathpzc{d}=(n,1)$ and $\vec{a} = (\frac{1}{\rho},1)$,
\begin{align*}
W^{1}_{q}(\R,v;F^{s_{0}}_{p,r}(\R^{n},w;X)) \cap L_{q}(\R,v;F^{s_{0}+\rho}_{p,r}(\R^{n},w;X) &\stackrel{\eqref{Boutet:eq:prelim:Bessel-pot_B&F}}{=}
\mathcal{H}^{(\rho,1)}_{(n,1)}[\E] \stackrel{\eqref{PIBVP:eq:prelim:identities_Sobolev&Bessel-potential;2}}{=} \mathcal{H}^{1,(\frac{1}{\rho},1)}_{(n,1)}[\E], \\
W^{1}_{q}(\R,v;F^{s_{1}}_{p,r}(\R^{n},w;X)) \cap L_{q}(\R,v;F^{s_{1}+\rho}_{p,r}(\R^{n},w;X) &\stackrel{\eqref{Boutet:eq:prelim:Bessel-pot_B&F}}{=}
\mathcal{H}^{(\rho,1)}_{(n,1)}[\F] \stackrel{\eqref{PIBVP:eq:prelim:identities_Sobolev&Bessel-potential;2}}{=} \mathcal{H}^{1,(\frac{1}{\rho},1)}_{(n,1)}[\F],
\end{align*}
\[
(\E,\F)_{\theta,q} = L_{q}(\R,v;(F^{s_{0}}_{p,r}(\R^{n},w;X),F^{s_{1}}_{p,r}(\R^{n},w;X))_{\theta,q})
\stackrel{\eqref{Boutet:intro:eq:real_int_F-spaces}}{=} L_{q}(\R,v;B^{s}_{p,q}(\R^{n},w;X)),
\]
where we used \cite[Theorem~2.2.10]{Hytonen&Neerven&Veraar&Weis2016_Analyis_in_Banach_Spaces_II} for the real interpolation of $L_{q}$-Bochner spaces, and
\begin{align*}
\mathcal{H}^{1,(\frac{1}{\rho},1)}_{(n,1)}[(\E,\F)_{\theta,q}]
&\stackrel{\eqref{PIBVP:eq:prelim:identities_Sobolev&Bessel-potential;2}}{=}
\mathcal{H}^{(\rho,1)}_{(n,1)}[(\E,\F)_{\theta,q}] \\
&\stackrel{\eqref{Boutet:eq:prelim:Bessel-pot_B&F}}{=}
W^{1}_{q}(\R,v;B^{s}_{p,q}(\R^{n},w;X)) \cap L_{q}(\R,v;B^{s+\rho}_{p,q}(\R^{n},w;X)).
\end{align*}
\end{proof}

\begin{corollary}\label{Boutet:lemma:DPBVP_boundedness_higher_order_traces;B}
Let $X$ be a UMD Banach space, $q,p \in (1,\infty)$, $v \in A_{q}(\R)$, $\gamma \in (-1,\infty)$, $s \in (-\infty,\frac{1+\gamma}{p})$, $\rho \in (0,\infty)$ and $\beta \in \N^{n}$. If $s+\rho-|\beta| > \frac{1+\gamma}{p}$, then $\mathrm{tr}_{\partial\R^{n}_{+}} \circ D^{\beta}$ is a bounded linear operator
\[
W^{1}_{q}(\R,v;B^{s}_{p,q}(\R^{n}_{+},w_{\gamma};X)) \cap L_{q}(\R,v;B^{s+\rho}_{p,q}(\R^{n}_{+},w_{\gamma};X)) \longra B^{\frac{1}{\rho}(s+\rho-|\beta|-\frac{1+\gamma}{p}),(\frac{1}{\rho},1)}_{(p,q),q}(\R^{n-1}\times\R,(1,v);X).
\]
\end{corollary}
\begin{proof}
This follows immediately from Corollary~\ref{Boutet:lemma:DPBVP_boundedness_higher_order_traces} by the real interpolation results from Lemma~\ref{Boutet:prelim:lemma_real_interpolation_max-reg_B} and \eqref{Boutet:intro:eq:real_int_F-spaces}.
\end{proof}

\section{Poisson Operators}\label{Boutet:sec:Poisson}

\subsection{Symbol Classes}\label{subsec:sec:Poisson_operators;symol_classes}

In this subsection we give the definition and derive some properties of the symbol classes we want to work with. We will restrict ourselves to symbols with constant coefficients and infinite regularity in the parameter-dependent case. For the main results of this paper, treating the general symbol classes which are usually considered in the framework of the Boutet de Monvel calculus is not necessary. Nonetheless we will treat them in a forthcoming paper for the discussion of pseudo-differential boundary value problems. Our symbol classes are variants of the classical symbol classes considered in \cite{Grubb1996_Functional_calculus} and the other references we gave in the introduction.

In this section, our parameter-dependent symbols usually depend on a complex variable. If we say that the symbol is differentiable with respect to that variable, we interpret this complex variable as an element of $\R^2$ and mean that the symbol is differentiable in the real sense. Likewise, if there is a complex variable appearing in the Bessel potential, we treat it as a variable in $\R^2$.

In the whole section, we use the notation and conventions of Section \ref{Boutet:subsec:prelim:anisotropy}. So let $l\in\N$ and $\mathpzc{d}\in\N_1^l$ such that $|\mathpzc{d}|=n$.

\begin{definition}\label{tpws:def:ParaIndep_Hormander_class}
 Let $Z$ be a Banach space, $d \in \R$ and $\Sigma\subset\C$ open. Let further $\vec{a}=(a_1,\ldots,a_{l})\in(0,\infty)^l$.
 \begin{enumerate}[(a)]
  \item The parameter-independent \emph{H\"ormander class} of order $d$ with constant coefficients denoted by $S^{d}(\R^{n};Z)$ is the space of all smooth functions $p \in C^{\infty}(\R^{n};Z)$ with
    \[
      \| p \|_{k}^{(d)}:= \sup_{\xi \in \R^{n}\atop\alpha\in\N^n,|\alpha|\leq k}\langle \xi\rangle^{-(d-|\alpha|)}\norm{D^{\alpha}_{\xi} p(\xi)}_{Z} < \infty
    \]
    for all $k\in\N$. Here, as usual the Bessel potential is defined by $\langle\xi\rangle:=(1+|\xi|^2)^{1/2}$.
  \item The anisotropic parameter-independent \emph{H\"ormander class} of order $d$ with constant coefficients denoted by $S^{d}_{\mathpzc{d},\vec{a}}(\R^{n};Z)$ is the space of all smooth functions $p \in C^{\infty}(\R^{n};Z)$ with
 \[
  \| p\|^{(d)}_{\mathpzc{d},\vec{a},k}:=\sup_{\xi\in \R^n\atop \alpha\in\N^n,|\alpha|\leq k} \langle \xi\rangle_{\mathpzc{d},\vec{a}}^{-(d-\vec{a}\cdot_{\mathpzc{d}}\alpha)}\|D_{\xi}^{\alpha} p(\xi)\|_Z <\infty
 \]
  for all $k\in\N$. Here, the anisotropic Bessel potential is defined by $$  \langle \xi\rangle_{\mathpzc{d},\vec{a}}:=(1+|\xi_1|^{2/a_1}+\ldots+|\xi_l|^{2/a_l})^{1/2}.$$
 \end{enumerate}
\end{definition}

\begin{definition}\label{tpws:def:ParaDep_Hormander_class}
 Let $Z$ be a Banach space, $d \in \R$ and $\Sigma\subset\C$ open. Let further $\vec{a}=(\vec{a}'',a_{l+1})=(a_1,\ldots,a_{l+1})\in(0,\infty)^{l}\times(0,\infty)$.
\begin{enumerate}[(a)]
 \item The isotropic parameter-dependent \emph{H\"ormander class} of order $d$ and regularity $\infty$ with constant coefficients denoted by $S^{d,\infty}(\R^{n}\times\Sigma;Z)$ is the space of all smooth functions $p \in C^{\infty}(\R^{n}\times\Sigma;Z)$ with
 \[
    \| p \|_{k}^{(d,\infty)}:= \sup_{(\xi,\mu) \in \R^{n}\times\Sigma \atop \alpha\in\N^n, j\in\N^2,|\alpha|+|j|\leq k}\langle \xi,\mu \rangle^{-(d-|\alpha|-|j|)}\norm{D^{\alpha}_{\xi}D^j_\mu p(\xi,\mu)}_{Z} < \infty
  \]
  for all $k\in\N$. Here, the parameter-dependent Bessel potential is defined by $$  \langle \xi,\mu\rangle:=(1+|\xi|^2+|\mu|^2)^{1/2}.$$
 \item The anisotropic parameter-dependent \emph{H\"ormander class} of order $d$ and regularity $\infty$ with constant coefficients denoted by $S^{d,\infty}_{\mathpzc{d},\vec{a}}(\R^{n}\times\Sigma;Z)$ is the space of all smooth functions $p \in C^{\infty}(\R^{n}\times\Sigma;Z)$ with
 \[
  \| p\|^{(d,\infty)}_{\mathpzc{d},\vec{a},k}:=\sup_{(\xi,\lambda)\in \R^n\times\Sigma\atop\alpha\in\N^n,j\in\N^2, |\alpha|+|j|\leq k} \langle \xi,\lambda\rangle_{\mathpzc{d},\vec{a}}^{-(d-\vec{a}''\cdot_{d}\alpha-a_{l+1}|j|)}\|D_{\xi}^{\alpha}D_{\lambda}^{j} p(\xi,\lambda)\|_Z <\infty
 \]
  for all $k\in\N$. Here, the anisotropic Bessel potential is defined by
  \[
    \langle \xi,\lambda\rangle_{\mathpzc{d},\vec{a}}:=(1+|\xi_1|^{2/a_1}+\ldots+|\xi_l|^{2/a_l}+|\lambda|^{2/a_{l+1}})^{1/2}.
  \]
  In the special case $l=1$ we also omit $\mathpzc{d}$ in the notation and write $S^{d,\infty}_{\vec{a}}(\R^{n}\times\Sigma;Z)$ and $\| p\|^{(d,\infty)}_{\vec{a},k}$ instead.
\end{enumerate}

\end{definition}

\begin{definition}\label{Boutet:def:ParaIndep_Poisson_class}
 Let $Z$ be a Banach space, $d\in\R$ and $1\leq p\leq\infty$. Let further $\mathpzc{d}_1=1$ and $\vec{a}=(a_1,\vec{a}')=(a_1,\ldots,a_{l})\in(0,\infty)\times(0,\infty)^{l-1}$.
  \begin{enumerate}[(a)]
    \item By $S^{d}(\R^{n-1} ;\mathcal{S}_{L_{p}}(\R_{+};Z))$ we denote the space of all smooth functions
	    \[
	      \widetilde{k}\colon \R_+\times \R^{n-1}\ \to Z,\;(x_1,\xi')\mapsto \widetilde{k}(x_1,\xi')
	    \]
	  satisfying
	    \begin{align*}
	      & \| \widetilde{k} \|_{S^{d}(\R^{n-1};\mathcal{S}_{L_{p}}(\R_{+};Z)),\alpha',m,m'}\\
	      := &\sup_{\xi'\in\R^{n-1}}\langle\xi'\rangle^{-(d-m+m' -|\alpha'|)+\tfrac{1}{p}-1}\norm{x_1\mapsto x_1^mD_{x_1}^{m'}D^{\alpha'}_{\xi'} \tilde{k}(x_1,\xi')}_{L^p(\R_+,Z)} < \infty
	    \end{align*}
	  for all $\alpha'\in \N^{n-1}$ and all $m,m'\in\N$. The elements of $S^{d}(\R^{n-1} ;\mathcal{S}_{L_{p}}(\R_{+};Z))$ will be called parameter-independent Poisson symbol-kernels of order $d+1$ or degree $d$.
    \item We denote by $S^{d}_{\mathpzc{d},\vec{a}}(\R^{n-1};\mathcal{S}_{L_{p}}(\R_{+};Z))$ the space of all smooth functions
    \[
    \tilde{k}:  \R_{+}\times\R^{n-1} \longrightarrow Z,\, (x_{1},\xi') \mapsto \tilde{k}(x_{1},\xi')
    \]
    satisfying
    \begin{align*}
    &\norm{\tilde{k}}_{S^{d}_{\mathpzc{d},\vec{a}}(\mathcal{S}_{L_{p}}(\R_{+};Z)),\alpha,m,m'} \\:=&
    \sup_{\xi'\in\R^{n-1}}\langle \xi' \rangle_{\mathpzc{d}',\vec{a}'}^{-(d-(m-m')a_{1}-\vec{a}'\cdot_{\mathpzc{d}'}\alpha')+a_1(\frac{1}{p}-1)}
    \norm{x_{1} \mapsto x_{1}^{m}D^{m'}_{x_{1}}D^{\alpha'}_{\xi'}\tilde{k}(x_{1},\xi')}_{L_{p}(\R_{+};Z)} < \infty
    \end{align*}
    for every $\alpha' \in \N^{n-1}$ an all $m,m' \in \N$. The elements of $S^{d}_{\mathpzc{d},\vec{a}}(\R^{n-1};\mathcal{S}_{L_{p}}(\R_{+};Z))$ will be called anisotropic parameter-independent Poisson symbol-kernels of order $d+a_1$ or degree $d$. In the special case $a_1=\ldots=a_{l}$ we omit $\mathpzc{d}$ in the notation and write $S^{d}_{\vec{a}}(\R^{n-1};\mathcal{S}_{L_{p}}(\R_{+};Z))$ and $\norm{\tilde{k}}_{S^{d}_{\vec{a}}(\mathcal{S}_{L_{p}}(\R_{+};Z)),\alpha,m,m'} $ instead.
 \end{enumerate}
\end{definition}

\begin{definition}\label{Boutet:def:ParaDep_Poisson_class}
 Let $Z$ be a Banach space, $\Sigma\subset\C$ open, $d\in\R$ and $1\leq p\leq\infty$. Let further $\mathpzc{d}_1=1$ and $\vec{a}=(a_1,\hat{\vec{a}},a_{l+1})\in(0,\infty)\times(0,\infty)^{l-1}\times(0,\infty)$.
  \begin{enumerate}[(a)]
    \item By $S^{d,\infty}(\R^{n-1} \times \Sigma;\mathcal{S}_{L_{p}}(\R_{+};Z))$ we denote the space of all smooth functions
    \[
      \widetilde{k}\colon \R_+\times\R^{n-1}\times  \Sigma \to Z,\;(x_1,\xi',\mu)\mapsto \widetilde{k}(x_1,\xi',\mu)
    \]
  satisfying
    \begin{align*}
      & \| \widetilde{k} \|_{S^{d,\infty}(\R^{n-1} \times \Sigma;\mathcal{S}_{L_{p}}(\R_{+};Z)),\alpha',m,m',\gamma}\\
      := &\sup_{(\xi',\mu)\in\R^{n-1}\times\Sigma}\langle\xi',\mu\rangle^{-d+m-m' +|\alpha'|+|\gamma|+\tfrac{1}{p}-1}\norm{x_1\mapsto x_1^mD_{x_1}^{m'}D^{\alpha'}_{\xi'}D^{\gamma}_\mu \tilde{k}(x_1,\xi',\mu)}_{L^p(\R_+,Z)} < \infty
    \end{align*}
  for all $\alpha'\in \N^{n-1}$, $\gamma\in\N^2$ and all $m,m'\in\N$. The elements of $S^{d,\infty}(\R^{n-1} \times \Sigma;\mathcal{S}_{L_{p}}(\R_{+};Z))$ will be called parameter-dependent Poisson symbol-kernels of order $d+1$ or degree $d$ and regularity $\infty$.
   \item We denote by $S^{d,\infty}_{\mathpzc{d},\vec{a}}(\R^{n-1} \times \Sigma;\mathcal{S}_{L_{p}}(\R_{+};Z))$ the space of all smooth functions
    \[
    \tilde{k}:  \R_{+} \times \R^{n-1} \times \Sigma \longrightarrow Z,\, (x_{1},\xi',\lambda) \mapsto \tilde{k}(x_{1},\xi',\lambda)
    \]
    satisfying
    \begin{align*}
    &\norm{\tilde{k}}_{S^{d,\infty}_{\mathpzc{d},\vec{a}}(\mathcal{S}_{L_{p}}(\R_{+};Z)),\alpha',m,m',\gamma}\\
    :=& \sup_{(\xi',\lambda)\in\R^{n-1}\times\Sigma}\langle \xi',\lambda \rangle_{\mathpzc{d}',\vec{a}'}^{-(d-(m-m')a_{1}-\hat{\vec{a}}\cdot_{\mathpzc{d}'}\alpha'-|\gamma|a_{l+1})+a_1(\frac{1}{p}-1)}
    \norm{x_{1} \mapsto x_{1}^{m}D^{m'}_{x_{1}}D^{\alpha'}_{\xi'}D^{\gamma}_{\lambda}\tilde{k}(x_{1},\xi',\lambda)}_{L_{p}(\R_{+};Z)} < \infty
    \end{align*}
    for every $\alpha' \in \N^{n-1}$, $m,m' \in \N$ and $\gamma \in \N^{2}$. The elements of $S^{d,\infty}_{\mathpzc{d},\vec{a}}(\R^{n-1} \times \Sigma;\mathcal{S}_{L_{p}}(\R_{+};Z))$ will be called anisotropic parameter-dependent Poisson symbol-kernels of order $d+a_1$ or degree $d$ and regularity $\infty$. In the special case $a_1=\ldots=a_{l}$ we omit $\mathpzc{d}$ in the notation and write $S^{d,\infty}_{\vec{a}}(\R^{n-1}\times\Sigma;\mathcal{S}_{L_{p}}(\R_{+};Z))$ and $\norm{\tilde{k}}_{S^{d,\infty}_{\vec{a}}(\mathcal{S}_{L_{p}}(\R_{+};Z)),\alpha',m,m',\gamma}$ instead.
 \end{enumerate}
\end{definition}

\begin{lemma}\label{Boutet:lemma:indep_p_Poisson_symb_kernel}
 Let $X$ be a Banach space. For $p\in[1,\infty]$, $m,m'\in\N$ let
 \[
  \|f\|_{\mathcal{S}_{L_p}(\R_+;X),m,m'}:=\|x\mapsto x^mD_x^{m'}f(x)\|_{L_p(\R_+;X)}\quad(f\in\mathcal{S}(\R_+;X)).
 \]
 We write $\mathcal{S}_{L_p}(\R_+;X)$ if we endow $\mathcal{S}(\R_+;X)$ with the topology generated by $\{\|\cdot\|_{\mathcal{S}_{L_p}(\R_+,X),m,m'}:m,m'\in\N\}$.
\begin{enumerate}[(a)]
 \item\label{Boutet:it:lemma:indep_p_Poisson_symb_kernel;(a)} The topology on $\mathcal{S}(\R_+;X)$ generated by the family $\{\|\cdot\|_{\mathcal{S}_{L_p}(\R_+;X),m,m'}:m,m'\in\N\}$ is independent of $p$.
 \item\label{Boutet:it:lemma:indep_p_Poisson_symb_kernel;(b)} The symbol-kernel class $S^{d,\infty}_{\mathpzc{d},\vec{a}}(\R^{n-1} \times \Sigma;\mathcal{S}_{L_{p}}(\R_{+};Z))$ is independent of $p$. The respective assertion also holds in the isotropic or parameter-independent case.
\end{enumerate}
\end{lemma}
\begin{proof}
\eqref{Boutet:it:lemma:indep_p_Poisson_symb_kernel;(a)}: We simply show that $\mathcal{S}_{L_p}(\R_+;X)\hookrightarrow \mathcal{S}_{L_q}(\R_+;X)$ for all choices of $p,q\in[1,\infty]$. If $q<p$ we can use H\"older's inequality. Let $m,m'\in\N$, $r\in[1,\infty]$ such that $1/q=1/r+1/p$. Then, we have
  \begin{align*}
   \|x\mapsto x^mD_x^{m'}f(x)\|_{L_q(\R_+;X)}\leq \|x\mapsto \langle x\rangle^{-2}\|_{L^r(\R_+)}\|x\mapsto \langle x\rangle^2 x^mD_x^{m'}f(x)\|_{L_p(\R_+;X)}\\
   \lesssim \max\{\|x\mapsto x^mD_x^{m'}f(x)\|_{L_p(\R_+;X)},\|x\mapsto x^{m+2}D_x^{m'}f(x)\|_{L_p(\R_+;X)}\}
  \end{align*}
  If $q\geq p$ we use the embedding $W^{1}_p(\R_+;X)\hookrightarrow L_q(\R_+;X)$ (cf. Proposition 3.12 in combination with Proposition 7.2 in \cite{Meyries&Veraar2012_sharp_embedding}). This embedding yields
  \begin{align*}
   \|x\mapsto x^mD_x^{m'}f(x)\|_{L_q(\R_+;X)}\lesssim \|x\mapsto x^mD_x^{m'}f(x)\|_{W^1_p(\R_+;X)} \\
   \lesssim\max\{\|x\mapsto x^mD_x^{m'}f(x)\|_{L_p(\R_+;X)},\|x\mapsto x^mD_x^{m'+1}f(x)\|_{L_p(\R_+;X)},\|x\mapsto x^{m-1}D_x^{m'}f(x)\|_{L_p(\R_+;X)}\}
  \end{align*}
  for all $m,m'\in\N$. Altogether, we obtain the assertion.

  \eqref{Boutet:it:lemma:indep_p_Poisson_symb_kernel;(b)}: We will derive this from \eqref{Boutet:it:lemma:indep_p_Poisson_symb_kernel;(a)} by a scaling argument.
  For simplicity of notation we restrict ourselves to the isotropic parameter-dependent case.
  Consider a smooth function
  \[
  \widetilde{k}\colon \R_+\times\R^{n-1}\times  \Sigma \to Z,\;(x_1,\xi',\mu)\mapsto \widetilde{k}(x_1,\xi',\mu).
  \]
  Let $\alpha' \in \N^{n-1}$, $\gamma \in \N^{2}$ and put $\tilde{k}_{\alpha',\gamma}(x_{1},\xi',\mu) := D^{\alpha'}_{\xi'}D^{\gamma}_{\mu}\tilde{k}(x_{1},\xi',\mu)$ and
$\tilde{h}_{\alpha',\gamma}(t_{1},\xi',\mu) := \tilde{k}_{\alpha',\gamma}(\langle \xi',\mu \rangle^{-1}t_{1},\xi',\mu)$.
Then
\[
\langle \xi',\mu \rangle^{\frac{1}{p}+m-m'}\norm{\tilde{k}_{\alpha',\gamma}(\,\cdot\,,\xi',\mu)}_{\mathcal{S}_{L_{p}},m,m'}
= \norm{\tilde{h}_{\alpha',\gamma}(\,\cdot\,,\xi',\mu)}
_{\mathcal{S}_{L_{p}},m,m'},
\qquad m,m' \in \N, p \in [1,\infty].
\]
Applying the seminorm estimates associated with \eqref{Boutet:it:lemma:indep_p_Poisson_symb_kernel;(a)} to $\tilde{h}_{\alpha',\gamma}(\,\cdot\,,\xi',\mu)$ the desired result follows.
\end{proof}

\begin{remark}\label{Boutet:rem:CommentsOnSymbols}
 \begin{enumerate}[(a)]
  \item Occasionally, we will need the estimates in the definitions of the Poisson symbol-kernel classes with $m$ being a non-negative real number instead of a natural number. But the respective estimates follow by using Young's inequality. Indeed, for example in the anisotropic parameter-dependent case we have for all $\theta\in[0,1]$ that
  \begin{align*}
   \langle\xi',\lambda\rangle_{\mathpzc{d},\vec{a}}^{(m+\theta)a_1}x_1^{p(m+\theta)}&=\langle\xi',\lambda\rangle_{\mathpzc{d},\vec{a}}^{m(1-\theta)a_1}x_1^{pm(1-\theta)}\langle\xi',\lambda\rangle_{\mathpzc{d},\vec{a}}^{(m+1)\theta a_1}x_1^{p(m+1)\theta}\\
   &\leq(1-\theta)\langle\xi',\lambda\rangle_{\mathpzc{d},\vec{a}}^{ma_1}x_1^{pm}+\theta\langle\xi',\lambda\rangle_{\mathpzc{d},\vec{a}}^{(m+1)a_1}x_1^{p(m+1)}.
  \end{align*}
  Using the triangle inequality for the $L_p(\R_+;Z)$-norm yields the desired estimate.
  \item Let $q>0$. Then we have that $S^{d,\infty}_{\mathpzc{d},\vec{a}}=S^{qd,\infty}_{\mathpzc{d},q\vec{a}}$ for all the anisotropic symbol classes, since
  \begin{align*}
    \langle \xi',\lambda\rangle_{\mathpzc{d},a}&=(1+|\xi_2|^{2/a_2}+\ldots+|\xi_l|^{2/a_l}+|\lambda|^{2/a_{l+1}})^{1/2}\\
    &\eqsim(1+|\xi_2|^{2/q a_2}+\ldots+|\xi_l|^{2/q a_l}+|\lambda|^{2/q a_{l+1}})^{q/2}=\langle \xi',\lambda\rangle_{\mathpzc{d},qa}^q.
  \end{align*}
  \item \label{Boutet:rem:SymbolDerivatives} Let $m,m'\in\N$, $\gamma\in\N^2$ as well as $\alpha'\in\N^{n-1}$. Then, it follows from the definition of the symbol-kernels that $\tilde{k}\mapsto x_{1}^{m}D^{m'}_{x_{1}}D^{\alpha'}_{\xi'}D^{\gamma}_{\lambda}\tilde{k}$ is a continuous mapping
  $$
  S^{d,\infty}_{\mathpzc{d},\vec{a}}(\R^{n-1} \times \Sigma;\mathcal{S}_{L_{p}}(\R_{+};Z))\to S^{d-(m-m')a_1-\hat{\vec{a}}\cdot_{\mathpzc{d}'}\alpha'-|\gamma|a_{l+1},\infty}_{\mathpzc{d},\vec{a}}(\R^{n-1} \times \Sigma;\mathcal{S}_{L_{p}}(\R_{+};Z)).$$
  The respective assertion also holds for the other symbol-kernel classes as well as for the H\"ormander symbols.
  \item \label{Boutet:rem:SymbolProducts} Let $d_1,d_2\in\R$ and suppose that we have a continuous bilinear mapping $Z_1\times Z_2\to Z$ for the Banach spaces $Z_1,Z_2$ and $Z$. Then, the  bilinear mapping
  \begin{align*}
   S^{d_1,\infty}_{\mathpzc{d},\vec{a}}(\R^{n} \times \Sigma,Z_1)\times S^{d_2,\infty}_{\mathpzc{d},\vec{a}}(\R^{n} \times \Sigma;Z_2)\to S^{d_1+d_2,\infty}_{\mathpzc{d},\vec{a}}(\R^{n} \times \Sigma;Z), \;(p_1,p_2)\mapsto p\cdot p
  \end{align*}
    is continuous. The respective assertions also hold for the other classes of H\"ormander symbols.
 \end{enumerate}
\end{remark}

\begin{remark}\label{Boutet:rmk:par-dep_Poisson_subst_time}
 Consider the situation of Definition \ref{Boutet:def:ParaDep_Poisson_class}. Suppose that $\Sigma=\Sigma_{\varphi}$ is a sector with opening angle $\varphi>\pi/2$. Let further $\widetilde{k}\in S^{d,\infty}_{\mathpzc{d},\vec{a}}(\R^{n-1} \times \Sigma;\mathcal{S}_{L_{p}}(\R_{+};Z))$ and $\eta\geq0$. Then $\tilde{k}^{(\eta)}\colon: (x_1,\xi',\theta)\mapsto \tilde{k}(x_1,\xi',1+\eta+i\theta)$ is an anisotropic parameter-independent Poisson symbol kernel in $S^{d}_{(\mathpzc{d},1),\vec{a}}(\R^{n};\mathcal{S}_{L_{p}}(\R_{+};Z))$ in the sense of Definition \ref{Boutet:def:ParaIndep_Poisson_class}. Moreover, for all $\alpha'\in\N^{n-1}$, $m,m'\in\N$ and $\gamma\in\N^2$ we have that
\begin{align*}
 &\|\tilde{k}^{(\eta)}\|_{S^d_{\mathpzc{d},\vec{a}}(\mathcal{S}_{L_p}(\R_+,Z)),\alpha',m,m',\gamma}\\
 &\quad\lesssim_{\alpha',m,m',\gamma} (1+\eta^{[d-(m-m')a_1-\hat{\vec{a}}\cdot_{\mathpzc{d}'}\alpha'-|\gamma|a_{l+1}-a_1(\frac{1}{p}-1)]_+})\norm{\tilde{k}}_{S^{d,\infty}_{\mathpzc{d},\vec{a}}(\mathcal{S}_{L_{p}}(\R_{+};Z)),\alpha',m,m',\gamma}.
\end{align*}
\end{remark}

\begin{proposition}\label{Boutet:cor:prop:Poisson_operator_DBVP;classical_par-dep}
 Let $Z$ be a Banach space and $d\in\R$. Let further $\mathpzc{d}_1=1$  and $\vec{a}=(a_1,\hat{\vec{a}},a_{l+1})\in(0,\infty)\times(0,\infty)^{l-1}\times(0,\infty)$. Let $q\in\N$ and $\Sigma_{q}:=\{z^q: z\in\Sigma\}$ for some open $\Sigma\subset\C$.
Let $\tilde{k}\in S^{d,\infty}_{\mathpzc{d},\vec{a}}(\R^{n-1}\times\Sigma_q;\mathcal{S}_{L_1}(\R_+,Z))$ an anisotropic symbol. Then, the transformation $\lambda=\mu^{q}$ leads to a symbol in $S^{qd,\infty}_{\mathpzc{d},\vec{a}_q}(\R^{n-1}\times\Sigma;\mathcal{S}_{L_1}(\R_+,Z))$ where $\vec{a}_q:=(qa_1,\ldots,qa_l,a_{l+1})$, i.e. we have that
\[
 [(x_1,\xi',\mu)\mapsto \tilde{k}(x_1,\xi',\mu^{q})]\in S^{qd,\infty}_{\mathpzc{d},\vec{a}_q}(\R^{n-1}\times\Sigma;\mathcal{S}_{L_1}(\R_+,Z)).
\]
\end{proposition}
\begin{proof} First, we note that
 \[
  \mu^q=(\mu_1+i\mu_2)^q=\sum_{\tilde{q}=0}^q\binom{q}{\tilde{q}}\mu_1^{\tilde{q}}(i\mu_2)^{q-\tilde{q}}.
 \]
We show by induction on $|\gamma|$ with $\gamma\in\N^2$ that $\partial_{\xi'}^{\alpha'}\partial_{\mu}^{\gamma} \tilde{k}(x_1,\xi',\mu^{q})$ is a linear combination of terms of the form $\mu^{j-i}f(x_1,\xi',\mu^q)$ where $j,i\in\N^2$ such that $j-i\in\N^2$ and $\frac{|j|}{q-1}+|i|=|\gamma|$ as well as $f\in S^{d-\hat{\vec{a}}\cdot_{\mathpzc{d}'}\alpha'-(|j|a_{l+1})/(q-1)}_{\mathpzc{d},\vec{a}}(\mathcal{S}_{L_1})$. Obviously, this is true for $\gamma=0$. So let $\partial_{\xi'}^{\alpha}\partial_{\mu}^{\gamma} \tilde{k}(x_1,\xi',\mu^{q})$ be a sum of terms of the form $\mu^{j-i}f(x_1,\xi',q)$ where $j,i\in\N^2$ such that $j-i\in\N^2$ and $\frac{|j|}{q-1}+|i|=|\gamma|$ as well as $f\in S^{d-\hat{\vec{a}}\cdot_{\mathpzc{d}'}\alpha'-(|j|a_{l+1})/(q-1)}_{\mathpzc{d},\vec{a}}(\mathcal{S}_{L_1})$. We consider the summands separately. Then, we have
 \begin{align*}
  &\qquad\partial_{\mu_1}[\mu^{j-i}f(x_1,\xi',\mu^{q})]\\
  &= [j_1-i_1]_+\mu^{j-i-e_1}f(x_1,\xi',\mu^{q})+\mu^{j-i}\bigg(\sum_{\tilde{q}=1}^q\tilde{q}\binom{q}{\tilde{q}}\mu_1^{\tilde{q}-1}(i\mu_2)^{q-\tilde{q}}\bigg)(\partial_{\mu_1}f)(x_1,\xi',\mu^{q}).
 \end{align*}
 A similar computations holds for $\partial_{\mu_2}[\mu^{j-i}f(x_1,\xi',\mu^{q})]$. Hence, by Remark \ref{Boutet:rem:CommentsOnSymbols} (\ref{Boutet:rem:SymbolDerivatives}) the induction is finished. Estimating such terms, we obtain
 \begin{align*}
  \| x_1\mapsto x_1^{m}D_{x_1}^{m'}\mu^{j-i}f(x_1,\xi',\mu^{q}) \|_{L_1} &\leq |\mu^{j-i}| \langle\xi',\mu^{q}\rangle_{\mathpzc{d},\vec{a}}^{d-(m-m')a_1-\hat{\vec{a}}\cdot_{\mathpzc{d}'}\alpha'-(|j|a_{l+1})/(q-1)}\\
  & \eqsim |\mu^{j-i}| \langle\xi',\mu\rangle_{\mathpzc{d},\vec{a}_q}^{q[d-(m-m')a_1-\hat{\vec{a}}\cdot_{\mathpzc{d}'}\alpha']-(q|j|a_{l+1})/(q-1)}\\
  & \leq \langle \xi',\mu\rangle_{\mathpzc{d},\vec{a}_q}^{qd-q(m-m')a_1-\hat{\vec{a}}_q\cdot_{\mathpzc{d}'}\alpha'-(|j|/(q-1)+|i|)a_{l+1}}\\
  &= \langle \xi',\mu\rangle_{\mathpzc{d},\vec{a}_q}^{qd-q(m-m')a_1-\hat{\vec{a}}_q\cdot_{\mathpzc{d}'}\alpha'-|\gamma|a_{l+1}}
 \end{align*}
This proves the assertion.
\end{proof}

\begin{definition}\label{Boutet:def:DefinitionOfTheOperators}
\begin{enumerate}[(a)]
 \item Given a H\"ormander symbol with constant coefficients $p$ or $p_{\mu}:=p(\cdot,\mu)$ in the parameter-dependent case, we define the associated operator
    \begin{align*}
	Pf :=\operatorname{OP}(p)f &=\mathscr{F}^{-1}p\mathscr{F} f\quad(f\in\mathcal{S}'(\R^{n};X)).
    \end{align*} or
    \begin{align*}
      P_{\mu}f :=\operatorname{OP}(p_\mu)f &=\mathscr{F}^{-1}p_\mu\mathscr{F} f\quad(f\in\mathcal{S}'(\R^{n};X)).
    \end{align*}
      respectively.
 \item  Given a Poisson symbol-kernel $k$ or $k_\mu:=k(\cdot,\cdot,\mu)$ in the parameter-dependent case, we define the associated operator
 \begin{align*}
 K:=\operatorname{OPK}(\widetilde{k})g(x)&:=(2\pi)^{1-n}\int_{\R^{n-1}} e^{ix'\xi'}\widetilde{k}(x_1,\xi')\hat{g}(\xi')\,d\xi'\quad(x\in\R^n_+,\;g\in\mathcal{S}(\R^{n-1};X)).
\end{align*} or
\begin{align*}
 K_{\mu}:=\operatorname{OPK}(\widetilde{k}_{\mu})g(x)&:=(2\pi)^{1-n}\int_{\R^{n-1}} e^{ix'\xi'}\widetilde{k}(x_1,\xi',\mu)\hat{g}(\xi')\,d\xi'\quad(x\in\R^n_+,\;g\in\mathcal{S}(\R^{n-1};X)),
\end{align*}
respectively.
\end{enumerate}
\end{definition}

\begin{definition}
 Let $Z$ be a Banach space, $d\in\R$. Let further $\mathpzc{d}_1=1$ and $\vec{a}=(a_1,\ldots,a_{l})\in(0,\infty)^{l}$.
\begin{enumerate}[(a)]
\item We denote by $S^{d}( \R^{n-1} ;\mathscr{S}_{L_{\infty}}(\R;Z))$ the space of all smooth functions
    \[
    p: \R \times \R^{n-1} \longrightarrow Z,\, (\xi_1,\xi') \mapsto p(\xi_1,\xi')
    \]
    satisfying
    \begin{align*}
    \norm{p}_{S^{d}(\mathscr{S}_{L_{\infty}}(\R;Z)),\alpha',m,m'}:=
    \quad\sup_{\xi\in\R^n} \langle \xi'\rangle^{-(d+m-m'-|\alpha'|)}
     \norm{\xi_{1}^{m}D^{m'}_{\xi_{1}}D^{\alpha'}_{\xi'}p(\xi_{1},\xi')}_Z < \infty
    \end{align*}
    for every $\alpha'\in \N^{n-1}$ and $m,m' \in \N$.
\item We denote by $S^{d}_{\mathpzc{d},\vec{a}}( \R^{n-1} ;\mathscr{S}_{L_{\infty}}(\R;Z))$ the space of all smooth functions
    \[
      p: \R \times \R^{n-1}\longrightarrow Z,\, (\xi_1,\xi') \mapsto p(\xi_1,\xi')
    \]
    satisfying
    \begin{align*}
    \norm{p}_{S^{d}_{\mathpzc{d},\vec{a}}(\mathscr{S}_{L_{\infty}}(\R;Z)),\alpha',m,m} :=
    \sup_{\xi\in\R^n} \langle \xi'\rangle_{\mathpzc{d}',\vec{a}'}^{-(d+(m-m')a_1-\vec{a}'\cdot_{\mathpzc{d}'}\alpha')}
     \norm{\xi_{1}^{m}D^{m'}_{\xi_{1}}D^{\alpha'}_{\xi'}p(\xi_{1},\xi')}_Z < \infty
    \end{align*}
    for every $\alpha'\in \N^{n-1}$ and $m,m' \in \N$.
\end{enumerate}
\end{definition}

\begin{definition}\label{Boutet:Def:SomeSymbolClass}
 Let $Z$ be a Banach space, $\Sigma\subset\C$ open, $d\in\R$ and $1\leq p\leq\infty$. Let further $\mathpzc{d}_1=1$ and $\vec{a}=(a_1,\hat{\vec{a}},a_{l+1})\in(0,\infty)\times(0,\infty)^{l-1}\times(0,\infty)$.
\begin{enumerate}[(a)]
\item We denote by $S^{d,\infty}( \R^{n-1} \times \Sigma;\mathscr{S}_{L_{\infty}}(\R;Z))$ the space of all smooth functions
    \[
    p: \R \times \R^{n-1} \times \Sigma \longrightarrow Z,\, (\xi_1,\xi',\mu) \mapsto p(\xi_1,\xi',\mu)
    \]
    satisfying
    \begin{align*}
    \norm{p}_{S^{d,\infty}(\mathscr{S}_{L_{\infty}}(\R;Z)),\alpha',m,m',\gamma}:=
    \quad\sup_{\xi\in\R^n,\mu\in\Sigma} \langle \xi',\mu \rangle^{-(d+m-m'-|\alpha'|-|\gamma|)}
     \norm{\xi_{1}^{m}D^{m'}_{\xi_{1}}D^{\alpha'}_{\xi'}D^{\gamma}_{\mu}p(\xi_{1},\xi',\mu)}_Z < \infty
    \end{align*}
    for every $\alpha'\in \N^{n-1}$, $m,m' \in \N$ and $\gamma \in \N^{2}$.
\item We denote by $S^{d,\infty}_{\mathpzc{d},\vec{a}}( \R^{n-1} \times \Sigma;\mathscr{S}_{L_{\infty}}(\R;Z))$ the space of all smooth functions
    \[
      p: \R \times \R^{n-1} \times \Sigma \longrightarrow Z,\, (\xi_1,\xi',\lambda) \mapsto p(\xi_1,\xi',\lambda)
    \]
    satisfying
    \begin{align*}
    \norm{p}_{S^{d,\infty}_{\mathpzc{d},\vec{a}}(\mathscr{S}_{L_{\infty}}(\R;Z)),\alpha',m,m',\gamma} :=
    \sup_{\xi\in\R^n,\lambda\in\Sigma} \langle \xi',|\lambda| \rangle_{\mathpzc{d}',\vec{a}'}^{-(d+(m-m')a_1-\hat{\vec{a}}\cdot_{\mathpzc{d}'}\alpha'-|\gamma| a_{l+1})}
     \norm{\xi_{1}^{m}D^{m'}_{\xi_{1}}D^{\alpha'}_{\xi'}D^{\gamma}_{\lambda}p(\xi_{1},\xi',\lambda)}_Z < \infty
    \end{align*}
    for every $\alpha'\in \N^{n-1}$, $m,m' \in \N$ and $\gamma \in \N^{2}$.
\end{enumerate}
\end{definition}

\begin{lemma}\label{Boutet:lemma:incl_mathscrSLinfty_into_Hormander}
 Consider the situation of Definition \ref{Boutet:Def:SomeSymbolClass}. We have the continuous embedding
 \[
 S^{d,\infty}_{\mathpzc{d},\vec{a}}( \R^{n-1} \times \Sigma;\mathscr{S}_{L_{\infty}}(\R;Z))\hookrightarrow S^{d,\infty}_{\mathpzc{d},\vec{a}}(\R^{n}\times\Sigma;Z).
 \]
 The respective assertion holds within the isotropic or parameter-independent classes.
\end{lemma}
\begin{proof}
  We only prove the result for the anisotropic and parameter-dependent case, as the other cases can be proven in the exact same way. For given $\alpha\in\N^n$ and $\gamma\in\N^2$ we obtain
  \begin{align*}
   &\quad\sup_{(\xi,\lambda)\in \R^n\times\Sigma} \langle \xi,\lambda\rangle_{\mathpzc{d},\vec{a}}^{-(d-\vec{a}'\cdot_{\mathpzc{d}}\alpha-a_{l+1}\gamma)}\|\partial_{\xi}^{\alpha}\partial_{\lambda}^{\gamma} p(\xi,\lambda)\|_Z\\
   &\lesssim \sup_{(\xi,\lambda)\in \R^n\times\Sigma} \big[\langle \xi',\lambda\rangle_{\mathpzc{d}',\vec{a}'}^{-(d-\vec{a}'\cdot_{\mathpzc{d}}\alpha-a_{l+1}|\gamma|)}\|\partial_{\xi}^{\alpha}\partial_{\lambda}^{\gamma} p(\xi,\lambda)\|_Z\\
   &\qquad+[-(d-\vec{a}'\cdot_{\mathpzc{d}}\alpha-a_{l+1}|\gamma|)]_+\|\xi_1^{\frac{1}{a_1}[-(d-\vec{a}'\cdot_{\mathpzc{d}}\alpha-a_{l+1}\gamma)]_+}\partial_{\xi}^{\alpha}\partial_{\lambda}^{\gamma} p(\xi,\lambda)\|_Z\big]\\&<\infty.
  \end{align*}
  Here we used the first part of Remark \ref{Boutet:rem:CommentsOnSymbols}.
\end{proof}

Note that Lemma~\ref{Boutet:lemma:incl_mathscrSLinfty_into_Hormander} shows us that we can define an operator to a symbol in $S^{d,\infty}_{\mathpzc{d},\vec{a}}( \R^{n-1} \times \Sigma;\mathscr{S}_{L_{\infty}}(\R;Z))$ by the means of Definition \ref{Boutet:def:DefinitionOfTheOperators}.

\begin{lemma}\label{Boutet:lemma:Poisson_tensor_Dirac_delta}
 Let $X,Y$ be a Banach spaces and $d\in\R$. Let further $\mathpzc{d}_1=1$ and $\vec{a}=(a_1,\hat{\vec{a}},a_{l+1})=(a_1,\ldots,a_{l+1})\in(0,\infty)\times(0,\infty)^{l-1}\times(0,\infty)$.
 There is a continuous linear mapping
 \[
 S^{d,\infty}_{\mathpzc{d},\vec{a}}(\R^{n-1} \times \Sigma;\mathcal{S}_{L_{1}}(\R_{+};\mathcal{B}(X,Y))) \longra S^{d,\infty}_{\mathpzc{d},\vec{a}}(\R^{n-1} \times \Sigma;\mathscr{S}_{L_{\infty}}(\R;\mathcal{B}(X,Y))),\,\widetilde{k} \mapsto p,
 \]
 which assigns to each $\widetilde{k}$ a $p$ such that $r_+\operatorname{OP}[p](\delta_0\otimes \,\cdot\,)=\operatorname{OPK}(\tilde{k})$.
 More explicitly, the mapping $\widetilde{k} \mapsto p$ can be defined by means of the diagram
 \[
  \begin{tikzcd}
    S^{d,\infty}_{\mathpzc{d},\vec{a}}(\R^{n-1} \times \Sigma;\mathcal{S}_{L_{1}}(\R_{+};\mathcal{B}(X,Y))) \arrow{r}{E} \arrow[swap]{dr}{\tilde{k}\mapsto p} & S^{d,\infty}_{\mathpzc{d},\vec{a}}(\R^{n-1} \times \Sigma;\mathcal{S}_{L_{1}}(\R;\mathcal{B}(X,Y))) \arrow{d}{\mathscr{F}_{x_1\mapsto \xi_1}} \\
     & S^{d,\infty}_{\mathpzc{d},\vec{a}}(\R^{n-1} \times \Sigma;\mathscr{S}_{L_{\infty}}(\R;\mathcal{B}(X,Y)))
  \end{tikzcd}
\]
where $E$ denotes the Seeley extension as in \cite{Seeley1964} and the space $S^{d,\infty}_{\mathpzc{d},\vec{a}}(\R^{n-1} \times \Sigma;\mathcal{S}_{L_{1}}(\R;\mathcal{B}(X,Y)))$ is defined analogously to $S^{d,\infty}_{\mathpzc{d},\vec{a}}(\R^{n-1} \times \Sigma;\mathcal{S}_{L_{1}}(\R_+;\mathcal{B}(X,Y)))$. The respective assertions also hold within the isotropic or parameter-independent classes.
\end{lemma}
\begin{proof}
 Before we give the proof we should note that our approach here is strongly influenced by \cite[Proposition 4.1]{Johnsen1996} where the Poisson symbol-kernel was also mapped to the corresponding H\"ormander symbol (cf. \cite{Grubb1990,Grubb1996_Functional_calculus}).\\
 We only prove the result for the anisotropic and parameter-dependent case, as the other cases can be proven in the exact same way. The proof consists of three steps:
 \begin{enumerate}[(i)]
  \item We show that the Seeley extension is bounded from $S^{d,\infty}_{\mathpzc{d},\vec{a}}(\R^{n-1} \times \Sigma;\mathcal{S}_{L_{1}}(\R_{+};\mathcal{B}(X,Y))) $ to $S^{d,\infty}_{\mathpzc{d},\vec{a}}(\R^{n-1} \times \Sigma;\mathcal{S}_{L_{1}}(\R;\mathcal{B}(X,Y)))$.
  \item We show that $\mathscr{F}_{x_1\mapsto \xi_1}$ is bounded from $S^{d,\infty}_{\mathpzc{d},\vec{a}}(\R^{n-1} \times \Sigma;\mathcal{S}_{L_{1}}(\R;\mathcal{B}(X,Y)))$ to $S^{d,\infty}_{\mathpzc{d},\vec{a}}(\R^{n-1} \times \Sigma;\mathscr{S}_{L_{\infty}}(\R;\mathcal{B}(X,Y)))$.
  \item We show that $\operatorname{OP}[\mathscr{F}_{x_1\mapsto \xi_1}E \tilde{k}](\delta_0\otimes\,\cdot\,)=\operatorname{OPK}[\tilde{k}]$.
 \end{enumerate}
So let us prove the three steps one by one:
 \begin{enumerate}[(i)]
  \item For the Seeley extension we fix two sequences $(a_k)_{k\in\N},(b_k)_{k\in\N}\subset \R$ such that
  \begin{enumerate}[$(i)$]
   \item $b_k<0$ for all $k\in\N$,
   \item $\sum_{k=1}^\infty |a_k||b_k|^j<\infty$ for all $j\in\N$,
   \item $\sum_{k=1}^\infty a_kb_k^j=1$ for all $j\in\N$,
   \item $b_k\to-\infty$ as $k\to \infty$.
  \end{enumerate}
It was proven in \cite{Seeley1964} that such sequences indeed exist. Moreover, we take a function $\phi\in C^{\infty}(\R_+)$ with $\phi(t)=1$ for $0\leq t\leq 1$ and $\phi(t)=0$ for $t\geq 2$. Then, the Seeley extension for a function $f\in S^{d,\infty}_{\mathpzc{d},\vec{a}}( \R^{n-1} \times \Sigma;\mathcal{S}_{L_{p}}(\R_{+};Z))$ is defined by
\[
 (Ef)(t,\xi',\lambda)=\sum_{k=1}^\infty a_k\phi(b_k t) f(b_k t,\xi',\lambda)\quad(t<0).
\]
  The assertion regarding the smoothness has already been proved by Seeley in \cite{Seeley1964}. Hence, we only have to show that the estimates of the symbol classed are preserved under the Seeley extension. But they indeed hold as
  {\allowdisplaybreaks
  \begin{align*}
  &\quad \| x_1\mapsto x_1^mD_{x_1}^{m'}  D_{\xi'}^{\alpha'}D_\lambda^\gamma E\tilde{k}(x_1,\xi',\lambda)\|_{L_1(\R_-;\mathcal{B}(X,Y))}\\
   &=\big\| x_1\mapsto x_1^mD_{x_1}^{m'}   D_{\xi'}^{\alpha'}D_\lambda^\gamma \sum_{k=1}^\infty a_k\phi(b_k x_1) \tilde{k}(b_k x_1,\xi',\lambda)\big\|_{L_1(\R_-;\mathcal{B}(X,Y))}\\
   &=\big\| x_1\mapsto x_1^m D_{\xi'}^{\alpha'}D_\lambda^\gamma \sum_{k=1}^\infty a_k\sum_{q=0}^{m'}\binom{m'}{q}b_k^{m'}(D_{x_1}^q\phi)(b_k x_1) (D_{x_1}^{m'-q}\tilde{k})(b_k x_1,\xi',\lambda)\big\|_{L_1(\R_-;\mathcal{B}(X,Y))}\\
   &\leq \sum_{k=1}^\infty a_kb_k^{m'}\sum_{q=0}^{m'} \binom{m'}{q}\big\| x_1\mapsto x_1^m  D_{\xi'}^{\alpha'}D_\lambda^\gamma (D_{x_1}^q\phi)(b_k x_1) (D_{x_1}^{m'-q}\tilde{k})(b_k x_1,\xi',\lambda)\big\|_{L_1(\R_-;\mathcal{B}(X,Y))}\\
   &\leq \sum_{k=1}^\infty a_kb_k^{m'}\sum_{q=0}^{m'} \binom{m'}{q}\big\| x_1\mapsto x_1^m  D_{\xi'}^{\alpha'}D_\lambda^\gamma (D_{x_1}^q\phi)(b_k x_1) (D_{x_1}^{m'-q}\tilde{k})(b_k x_1,\xi',\lambda)\big\|_{L_1(\R_-;\mathcal{B}(X,Y))}\\
   &\leq \sum_{k=1}^\infty a_kb_k^{m'-m-1}\sum_{q=0}^{m'} \binom{m'}{q}\big\| y_1\mapsto y_1^m  D_{\xi'}^{\alpha'}D_\lambda^\gamma (D_{y_1}^q\phi)(y_1) (D_{y_1}^{m'-q}\tilde{k})(y_1,\xi',\lambda)\big\|_{L_1(\R_+;\mathcal{B}(X,Y))}\\
  &\leq \sum_{k=1}^\infty a_kb_k^{m'-m-1}\sum_{q=0}^{m'} \binom{m'}{q}\|D_{y_1}^q\phi\|_{L^\infty(\R^n_+)}\big\| y_1\mapsto y_1^m  D_{\xi'}^{\alpha'}D_\lambda^\gamma (D_{y_1}^{m'-q}\tilde{k})(y_1,\xi',\lambda)\big\|_{L_1(\R_+;\mathcal{B}(X,Y))}\\
  &\leq \sum_{k=1}^\infty a_kb_k^{m'-m-1}\sum_{q=0}^{m'} \binom{m'}{q}\|D_{y_1}^q\phi\|_{L^\infty(\R^n_+)}C_{\alpha',m,m'-q,|\gamma|}\langle\xi',\lambda\rangle_{\mathpzc{d}',\vec{a}'}^{d-(m-m'+q)a_1-\hat{\vec{a}}\cdot_{\mathpzc{d}'}\alpha'-\gamma a_{l+1}}\\
  &\leq C \langle\xi',\lambda\rangle_{\mathpzc{d}',\vec{a}'}^{d-(m-m')a_1-\hat{\vec{a}}\cdot_{\mathpzc{d}'}\alpha'-|\gamma| a_{l+1}}.
  \end{align*}}
  \item This follows directly from the above computation together with the definition of the symbol classes and the fact that $\mathscr{F}_{x_1\mapsto\xi_1}$ maps $L_1(\R;\mathcal{B}(X,Y))$ continuously into $L_\infty(\R;\mathcal{B}(X,Y))$.
  \item For all $g\in\mathcal{S}(\R^{n-1})$ and all $x\in\R^n_+$ we have that
   \begin{align*}
  \operatorname{OP}(\mathscr{F}_{x_1\mapsto \xi_1}E \tilde{k})(\delta_0\otimes g)(x)&=(2\pi)^{-n}\int_{\R^{n}} e^{ix\xi} [\mathscr{F}_{x_1\mapsto \xi_1}E \tilde{k}(\xi_1,\xi',\mu)] \mathscr{F}_{x\mapsto\xi}(\delta_0\otimes g)\,d\xi\\
  &=(2\pi)^{-n}\int_{\R^{n}} e^{ix\xi} [\mathscr{F}_{x_1\mapsto \xi_1}E \tilde{k}(\xi_1,\xi',\mu)] 1(\xi_1)\mathscr{F}_{x'\mapsto \xi'}g(\xi')\,d\xi\\
  &=(2\pi)^{1-n}\int_{\R^{n-1}}e^{ix'\xi'} E \tilde{k}(x_1,\xi',\mu)\widehat{g}(\xi')\,d\xi'\\
  &=(2\pi)^{1-n}\int_{\R^{n-1}}e^{ix'\xi'} \tilde{k}(x_1,\xi',\mu)\widehat{g}(\xi')\,d\xi'\\
  &=\operatorname{OPK}(\tilde{k})g(x).
 \end{align*}
 This finishes the proof.
 \end{enumerate}
\end{proof}

\begin{remark}
 Note that in Lemma \ref{Boutet:lemma:Poisson_tensor_Dirac_delta} we can also apply $r_+\operatorname{OP}[p](\delta_0\otimes \,\cdot\,)$ to elements of $\mathcal{S}'(\R^{n-1};X)$, cf. Section \ref{Boutet:subsubsec:prelim:FS}.
\end{remark}

\begin{lemma}\label{Boutet:Lemma:SymbolPointwiseMultiplication}
Let $Z,Z_{1},Z_{2},Z_{3}$ be Banach spaces and $d_{1},d_{2},d_{3} \in \R$. Let further $\mathpzc{d}_1=1$ and $\vec{a}=(a_1,\hat{\vec{a}},a_{l+1})=(a_1,\ldots,a_{l+1})\in(0,\infty)\times(0,\infty)^{l-1}\times(0,\infty)$.
A continuous trilinear mapping $Z_{1} \times Z_{2} \times Z_{3} \to Z$ induces by pointwise multiplication
a continuous trilinear mapping
\[
\begin{array}{c}
S^{d_{1},\infty}_{\mathpzc{d},\vec{a}}(\R^{n}  \times \Sigma;Z_{1}) \\
 \times \\
S^{d_{2},\infty}_{\mathpzc{d},\vec{a}}(\R^{n-1} \times \Sigma;\mathscr{S}_{L_{\infty}}(\R;Z_{2})) \\
 \times \\
S^{d_{3},\infty}_{\mathpzc{d}',\vec{a}'}(\R^{n-1} \times \Sigma;Z_{3})
\end{array}
\longrightarrow S^{d_{1}+d_{2}+d_{3},\infty}_{\mathpzc{d},\vec{a}}(\R^{n-1} \times \Sigma;\mathscr{S}_{L_{\infty}}(\R;Z)),
\]
where $(p_{1},p_{2},p_{3}) \mapsto p$ is given by
\[
p(\xi_{1},\xi',\mu) = p_{1}(\xi,\mu)p_{2}(\xi,\mu)p_{3}(\xi',\mu).
\]
It also holds that
\[
 \operatorname{OP}[p]=\operatorname{OP}[p_1]\circ\operatorname{OP}[p_2]\circ\operatorname{OP}[p_3].
\]
Again, the respective assertions also hold within the isotropic or parameter-independent classes.
\end{lemma}
\begin{proof}
 In order to keep notations shorter, we first show the assertion for constant $p_3$. Hence, we omit it in the notation and estimate the term
 \[
  \|\xi_1^m D_{\xi_1}^{m'}D_{\xi'}^{\alpha'}D_\lambda^j p_1(\xi_1,\xi',\lambda)p_2(\xi_1,\xi',\lambda) \|_{Z}.
 \]
By the product rule and the triangle inequality, it suffices to estimate expressions of the form
\[
 \| D_{\xi_n}^{\tilde{m}'}D_{\xi'}^{\tilde{\alpha}'}D_\lambda^{\tilde{j}}p_1(\xi_1,\xi',\lambda)  \xi_n^m D_{\xi_n}^{\bar{m}'}D_{\xi'}^{\bar{\alpha}'}D_\lambda^{\bar{j}} p_2(\xi_1,\xi',\lambda) \|_Z,
\]
where $|\bar{\alpha}'|+|\tilde{\alpha}'|=|\alpha'|$, $|\tilde{j}|+|\bar{j}|=|j|$ and $\bar{m}'+\tilde{m}'=m'$. But for such an expression, we obtain
\begin{align*}
 &\quad\| D_{\xi_1}^{\tilde{m}'}D_{\xi'}^{\tilde{\alpha}'}D_\lambda^{\tilde{j}}p_1(\xi_1,\xi',\lambda)  \xi_1^m D_{\xi_1}^{\bar{m}'}D_{\xi'}^{\bar{\alpha}'}D_\lambda^{\bar{j}} p_2(\xi_1,\xi',\lambda) \|_Z \\
 &\lesssim \langle \xi,\lambda \rangle_{\mathpzc{d},\vec{a}}^{d_1-\tilde{m}'a_1-\hat{\vec{a}}\cdot_{\mathpzc{d}'}\tilde{\alpha}'-|\tilde{j}|a_{l+1}}\| \xi_1^m D_{\xi_1}^{\bar{m}'}D_{\xi'}^{\bar{\alpha}'}D_\lambda^{\bar{j}} p_2(\xi_1,\xi',\lambda) \|_{Z_2}\\
 &\lesssim \langle \xi',\lambda  \rangle_{\mathpzc{d}',\vec{a}'}^{d_1-\tilde{m}'a_1-\hat{\vec{a}}\cdot_{\mathpzc{d}'}\tilde{\alpha}'-|\tilde{j}|a_{l+1}}\| \xi_1^{m} D_{\xi_1}^{\bar{m}'}D_{\xi'}^{\bar{\alpha}'}D_\lambda^{\bar{j}} p_2(\xi_1,\xi',\lambda) \|_{Z_2}\\
 &\quad+[d_1-\tilde{m}'a_1-\hat{\vec{a}}\cdot_{\mathpzc{d}'}\tilde{\alpha}'-|\tilde{j}|a_{l+1}]_+\| \xi_1^{m+\tfrac{1}{a_1}[d_1-\tilde{m}'a_1-\hat{\vec{a}}\cdot_{\mathpzc{d}'}\tilde{\alpha}'-|\tilde{j}|a_{l+1}]_+} D_{\xi_1}^{\bar{m}'}D_{\xi'}^{\bar{\alpha}'}D_\lambda^{\bar{j}} p_2(\xi_1,\xi',\lambda) \|_{Z_2}\\
 &\lesssim\langle\xi',\lambda \rangle_{\mathpzc{d}',\vec{a}'}^{d_1+d_2-(\tilde{m}'+\bar{m}'-m)a_1-\hat{\vec{a}}\cdot_{\mathpzc{d}'}(\tilde{\alpha}'+\bar{\alpha}')-(|\tilde{j}|+|\bar{j}|)a_{l+1}}\\
 &\quad+[d_1-\tilde{m}'a_1-\hat{\vec{a}}\cdot_{\mathpzc{d}'}\tilde{\alpha}'-|\tilde{j}|a_{l+1}]_+\langle\xi',|\lambda| \rangle_{\mathpzc{d},\vec{a}}^{d_1+d_2-(\tilde{m}'+\bar{m}'-m)a_1-(\tilde{\alpha}'+\bar{\alpha}')-(|\tilde{j}|+|\bar{j}|)a_{l+1}}\\
 &\lesssim\langle\xi',\lambda \rangle_{\mathpzc{d}',\vec{a}'}^{d_1+d_2-(m'-m)a_1-\alpha_2a_2-\ldots-\alpha_la_l-|j|a_{l+1}}.
\end{align*}
A similar computation shows the respective assertion for the case that $p_1$ is constant and $p_3$ is arbitrary. The formula for the operators is trivial.
\end{proof}

\subsection{Solution Operators for Model Problems}\label{Boutet:subsec:sec:poisson:SolutionOperator}

In this subsection we consider the boundary value model problems
\begin{equation}\label{Boutet:eq:DEBVP_bd-data}
\left\{
\begin{array}{rlll}
(1+\lambda) v + \mathcal{A}(D)v &= 0, &\text{on $\R^{n}_{+}$}, \\
\mathcal{B}_{j}(D)v &= g_{j},&\text{on $\R^{n-1}$}, & j=1,\ldots,m.
\end{array}\right.
\end{equation}
 and
\begin{equation}\label{Boutet:eq:DEBVP_bd-data;parabolic}
\left\{\begin{array}{rlll}
\partial_{t}u + (1+\eta+\mathcal{A}(D))u &= 0 &\text{on $\R^{n}_{+} \times \R$},  \\
\mathcal{B}_{j}(D)u &= g_{j},&\text{on $\R^{n-1} \times \R$}, & j=1,\ldots,m,
\end{array}\right.
\end{equation}
for $\eta\geq0$. Here, $\mathcal{A}(D),\mathcal{B}_{1}(D),\ldots,\mathcal{B}_{m}(D)$ is a constant coefficient homogeneous $\mathcal{B}(X)$-valued differential boundary value system on $\R^{n}_{+}$ as considered in Section~\ref{Boutet:subsec:sec:prelim:E&LS}.
In this subsection we restrict ourselves to $g_1,\ldots,g_m\in\mathcal{S}(\R^{n-1};X)$ so that we can later extend the solution by density to the desired spaces.

The following proposition and its corollary are the main results of this subsection. They (together with the mapping properties that will be obtained in Section~\ref{Boutet:subsec:sec:Poisson;mapping_prop}) show that the Poisson operators introduced in Section~\ref{subsec:sec:Poisson_operators;symol_classes} provide the right classes of operators for solving \eqref{Boutet:eq:DEBVP_bd-data} and \eqref{Boutet:eq:DEBVP_bd-data;parabolic}.

\begin{proposition}\label{Boutet:prop:Poisson_operator_DBVP}
Let $X$ be a Banach space and assume that $(\mathcal{A},\mathcal{B}_{1},\ldots,\mathcal{B}_{m})$ satisfies $(\mathrm{E})$ and $(\mathrm{LS})$ for some $\phi \in (0,\pi)$.
Then there exist $\widetilde{k}_{j} \in S^{-\frac{m_{j}+1}{2m},\infty}_{(\frac{1}{2m},\frac{1}{2m},1)}(\R^{n-1} \times \Sigma_{\pi-\phi};\mathcal{S}_{L_{1}}(\R_{+};\mathcal{B}(X)))$, $j=1,\ldots,m$, such that, for each $\lambda \in \Sigma_{\pi-\phi}$,
\[
K_{\lambda}:\mathcal{S}(\R^{n-1};X)^{m} \longra \mathcal{S}(\R^{n}_{+};X),\,
(g_{1},\ldots,g_{m}) \mapsto \sum_{j=1}^{m}\operatorname{OPK}(\widetilde{k}_{j,\lambda})g_{j},
\]
is a solution operator for the elliptic differential boundary value problem \eqref{Boutet:eq:DEBVP_bd-data}.
Moreover, there is uniqueness of solutions in $\mathcal{S}(\R^{n}_{+};X)$: if $u \in \mathcal{S}(\R^{n}_{+};X)$ is a solution of \eqref{Boutet:eq:DEBVP_bd-data}, then $u=K_{\lambda}(g_{1},\ldots,g_{m})$.
\end{proposition}

\begin{remark}
 Proposition \ref{Boutet:prop:Poisson_operator_DBVP} together with Proposition \ref{Boutet:cor:prop:Poisson_operator_DBVP;classical_par-dep} shows that $\tilde{k_j}$ belongs to
 $S^{-m_j-1,\infty}(\R^{n-1}\times\Sigma_{(\pi-\phi)/2m};S_{L_1}(\R_+,\mathcal{B}(X)))$ after the substitution $\lambda=\mu^{2m}$. To be more precise:
 \[
 \tilde{k_j}^{[2m]} := [(x_{1},\xi',\mu) \mapsto \tilde{k_j}(x_{1},\xi',\mu^{2m})] \in S^{-m_j-1,\infty}(\R^{n-1}\times\Sigma_{(\pi-\phi)/2m};S_{L_1}(\R_+,\mathcal{B}(X))).
 \]
\end{remark}

\begin{corollary}\label{Boutet:prop:Poisson_operator_DBVP;parabolic}
Let $X$ be a Banach space and assume that $(\mathcal{A},\mathcal{B}_{1},\ldots,\mathcal{B}_{n})$ satisfies $(\mathrm{E})$ and $(\mathrm{LS})$ for some $\phi \in (0,\frac{\pi}{2})$.
Then there exist $\widetilde{k}_{j}^{(\eta)} \in S^{-\frac{m_{j}+1}{2m}}_{(\frac{1}{2m},\frac{1}{2m},1)}(\R^{n-1}\times \R;\mathcal{S}_{L_{1}}(\R_{+};\mathcal{B}(X)))$, $j=1,\ldots,m$, such that
\[
K^{(\eta)}:\mathcal{S}(\R^{n-1} \times \R;X)^{m} \longra \mathcal{S}(\R^{n}_{+} \times \R;X),\,
(g_{1},\ldots,g_{m}) \mapsto \sum_{j=1}^{m}\operatorname{OPK}(\widetilde{k}_{j}^{(\eta)})g_{j},
\]
is a solution operator for the parabolic differential boundary value problem \eqref{Boutet:eq:DEBVP_bd-data;parabolic}. The seminorms of the symbol-kernels admit polynomial bounds in $\eta$ as described in Remark \ref{Boutet:rmk:par-dep_Poisson_subst_time}.
Moreover, there is uniqueness of solutions in $\mathcal{S}(\R^{n}_{+} \times \R;X)$: if $u \in \mathcal{S}(\R^{n}_{+} \times \R;X)$ is a solution of \eqref{Boutet:eq:DEBVP_bd-data;parabolic}, then $u=K^{(\eta)}(g_{1},\ldots,g_{m})$.
\end{corollary}
\begin{proof}
Under Fourier transformation in time, \eqref{Boutet:eq:DEBVP_bd-data;parabolic} turns into
\begin{equation*}\label{Boutet:eq:DPBVP_bd-data;Fourier_transform_time}
\left\{
\begin{array}{rll}
(1+\eta+i\tau)\mathcal{F}_{t \to \tau}u + \mathcal{A}(D)\mathcal{F}_{t \to \tau}u &= 0,  \\
\mathcal{B}_{j}(D)\mathcal{F}_{t \to \tau}u &= \mathcal{F}_{t \to \tau}g_{j}, & j=1,\ldots,m.
\end{array}\right.
\end{equation*}
The result thus follows from Proposition~\ref{Boutet:prop:Poisson_operator_DBVP} through a substitution as in Remark~\ref{Boutet:rmk:par-dep_Poisson_subst_time}.
\end{proof}

In order to prove Proposition~\ref{Boutet:prop:Poisson_operator_DBVP}, we use a certain solution formula to \eqref{Boutet:eq:DEBVP_bd-data}. Following the considerations in \cite[Proposition 6.2]{Denk&Hieber&Pruess2001_monograph} we can represent the solution in the Fourier image as
\[
 \widehat{u}(x_1,\xi',\lambda)= e^{i\rho A_0(b,\sigma)x_1} M(b,\sigma)\widehat{g}_\rho(\xi')
\]
where
\begin{itemize}
 \item $A_0$ is some smooth function with values in $\mathcal{B}(X^{2m},X^{2m})$ that one obtains from $\lambda-\mathcal{A}(D_{x_1},\xi')$ after some reduction to a first-order system,
 \item $M$ is some smooth function with values in $\mathcal{B}(X^{m},X^{2m})$ which maps the values of the boundary operator applied to the stable solution to the vector containing all normal derivatives of this solution up to order $2m-1$,
 \item $\rho$ is a positive parameter that can be chosen in different ways and in dependence of $\xi'$ and $\lambda$,
 \item $b=\xi'/\rho$, $\sigma=(1+\lambda)/\rho^{2m}$ and $\hat{g}_\rho=(\hat{g}_1/\rho^{m_1},\ldots, \hat{g}_m/\rho^{m_m})^T$.
\end{itemize}
Another operator that we will use later is the spectral projection $\ms{P}_{-}$ of the matrix $A_0$ to the part of the spectrum that lies above the real line. This spectral projection hast the property that $\ms{P}_-(b,\sigma)M(b,\sigma)=M(b,\sigma)$.\\
For our purposes, we will rewrite the above representation in the following way: For $j=1,\ldots,m$ we write
  \[
   M_{\rho,j}(b,\sigma)\hat{g}_j := M(b,\sigma)\frac{\hat{g}_j\otimes e_j}{\rho^{m_j}}
  \]
so that we obtain
\begin{align}\label{Boutet:Equation:SolutionFormulaPoisson}
  \widehat{u}= e^{i\rho A_0(b,\sigma)x_1} M(b,\sigma)\widehat{g}_\rho=\sum_{j=1}^m e^{i\rho A_0(b,\sigma)x_1}M_{\rho,j}(b,\sigma)\hat{g}_j.
\end{align}
The functions $(\xi',\lambda)\mapsto e^{i\rho A_0(b,\sigma)x_1}M_{\rho,j}(b,\sigma)$ (note that $\rho,b$ and $\sigma$ depend on $(\xi,\lambda)$ where we oppress the dependence in the notation for the sake of readability) are exactly the Poisson symbol-kernels $\tilde{k}_j$ in Proposition \ref{Boutet:prop:Poisson_operator_DBVP}. In the following, we will show that they satisfy the symbol-kernel estimates in order to prove Proposition \ref{Boutet:prop:Poisson_operator_DBVP}.

\begin{lemma}\label{Boutet:lemma:DerivativeHomogeneous}
 Let $N\in\N$ and let $\Sigma_1,\ldots,\Sigma_N\subset \C$ be some sectors (or lines) in the complex plane. Let further $m\colon \prod_{i=1}^N \Sigma_i\setminus\{0\}\to \C$ be differentiable and homogeneous in the sense that there are numbers $\alpha_1,\ldots,\alpha_N,\alpha\in\R$ such that
 \[
  m(r^{\alpha_1} x_1,\ldots,r^{\alpha_N}x_N)=r^\alpha m(x_1,\ldots,x_N)\quad(r>0,x_i\in\Sigma_i, i=1,\ldots,N).
 \]
Then, we have
\[
 (\partial_jm)(r^{\alpha_1} x_1,\ldots,r^{\alpha_N}x_N)=r^{\alpha-\alpha_j} \partial_jm(x_1,\ldots,x_N)\quad(r\geq0,x_i\in\Sigma_i, i,j=1,\ldots,N).
\]
\end{lemma}

\begin{proof}
For the sake of simpler notation, we only consider the case $\Sigma_1=\ldots=\Sigma_N=\R$. In the other cases, one only has to consider derivatives in more directions, but the proof remains the same. Let $r>0$ and $x_i\in\Sigma_i$ for $i=1,\ldots,N$. Define $x=(x_1,\ldots,x_N)$ and $x_{(r)}:=(r^{\alpha_1}x_1,\ldots,r^{\alpha_N}x_n)$. Let further $e_j$ be the $j$-th unit normal vector. Then we have
 \begin{align*}
  (\partial_jm)(x_{(r)})=\lim_{h\to 0}\frac{m(x_{(r)}+he_j)}{h}=\lim_{h\to 0}r^{\alpha}\frac{m(x+\tfrac{h}{r_j^{\alpha_j}}e_j)}{h}=\lim_{\tilde{h}\to 0}r^{\alpha-\alpha_j}\frac{m(x+\tilde{h}e_j)}{\tilde{h}}=r^{\alpha-\alpha_j}(\partial_jm)(x).
 \end{align*}
\end{proof}

\begin{proposition}\label{Boutet:Proposition:VariableRescalingSymbolClass}
 Let $a_1,a_2>0$ such that $\frac{1}{a_1},\frac{1}{a_2}\in\N$. Then the function $(\xi,\lambda)\mapsto (\frac{\xi}{\langle \xi,\lambda\rangle_{\vec{a}}^{a_1}},\frac{1+\lambda}{\langle \xi,\lambda\rangle_{\vec{a}}^{a_2}})$ is a symbol in $S^{0,\infty}_{\vec{a}}(\R^n\times\Sigma,\C^{n+1})$.
\end{proposition}
\begin{proof}
 The function
 \[
  m\colon \R\times\R^n\times\Sigma\setminus\{0\}\colon (x,\xi,\lambda)\mapsto \frac{x+\lambda}{(x^{2/a_2}+|\xi|^{2/a_1}+|\lambda|^{2/a_2})^{a_2/2}}
 \]
is homogeneous in the sense that
\[
 m(r^{a_2}x,r^{a_1}\xi,r^{a_2}\lambda)=m(x,\xi,\lambda).
\]
Moreover, since $\frac{1}{a_1},\frac{1}{a_2}\in\N$ we also have that $m\in C^{\infty}(\R\times\R^n\times \Sigma\setminus\{0\},\C)$. In particular, for all $\alpha\in\N^{n}$ and $k\in\N^2$ we have that $\partial_{\xi}^{\alpha}\partial_{\lambda}^{k}m$ is bounded on the set
\[
 S_{\vec{a}}:=\{(x,\xi,\lambda)\in\R\times\R^n\times\Sigma\setminus\{0\}: x^{2/a_2}+|\xi|^{2/a_1}+|\lambda|^{2/a_1}= 1 \}
\]
and satisfies
\[
 (\partial_{\xi}^{\alpha}\partial_{\lambda}^{k}m)(r^{a_2}x,r^{a_1}\xi,r^{a_2}\lambda)=r^{-a_1|\alpha|-a_2|k|}(\partial_{\xi}^{\alpha}\partial_{\lambda}^{k}m)(x,\xi,\lambda).
\]
Thus, we have the estimate
\begin{align*}
 \sup_{(\xi,\lambda)\in\R^n\times\Sigma}\langle \xi,\lambda\rangle_{\vec{a}}^{a_1|\alpha|+a_2|k|} |\partial_{\xi}^{\alpha}\partial_{\lambda}^{k}m(1,\xi',\lambda) |&\leq  \sup_{(x,\xi,\lambda)\in\R\times\R^n\times\Sigma\setminus\{0\}}(x^{2/a_2}+|\xi'|^{2/\alpha_1}+|\lambda|^{2/\alpha_2})^{\tfrac{a_1|\alpha|}{2}+\tfrac{a_2|k|}{2}}|\partial_{\xi}^{\alpha}\partial_{\lambda}^{k}m(x,\xi,\lambda)|\\
 &\leq \| \partial_{\xi}^{\alpha}\partial_{\lambda}^{k}m \|_{L_{\infty}(S_{\vec{a}})}
\end{align*}
so that we obtain that $(\xi,\lambda)\mapsto \frac{1+\lambda}{\langle \xi,\lambda\rangle_{\vec{a}}^{a_2}}$ is a symbol in $S^{0,\infty}_{\vec{a}}(\R^n\times\Sigma,\C)$. A similar approach also shows the desired estimates for the other components.
\end{proof}

For the rest of this section, in \eqref{Boutet:Equation:SolutionFormulaPoisson} we fix
\[
 \rho(\xi',\lambda):=\langle\xi',\lambda\rangle_{\vec{a}}^{a_1}
\]
\[
 b(\xi',\lambda):=\frac{\xi'}{\langle\xi',\lambda\rangle_{\vec{a}}^{a_1}}\quad\mbox{and}\quad \sigma(\xi',\lambda):=\frac{1+\lambda}{\langle\xi',\lambda\rangle_{\vec{a}}^{2ma_1}}.
\]
In particular, if we choose ${\vec{a}}=(a_1,a_2)=(\tfrac{1}{2m},1)$ then we obtain
\[
 b(\xi',\lambda):=\frac{\xi'}{\langle\xi',\lambda\rangle_{\vec{a}}^{a_1}}\quad\mbox{and}\quad \sigma(\xi',\lambda):=\frac{1+\lambda}{\langle\xi',\lambda\rangle_{\vec{a}}^{a_2}}
\]
so that $(b,\sigma)$ coincides with the function in Proposition \ref{Boutet:Proposition:VariableRescalingSymbolClass}.

\begin{proposition}\label{Boutet:Proposition:CompositionHomogeneousBounded}
  Let again $a_1,a_2>0$ such that $\frac{1}{a_1},\frac{1}{a_2}\in\N$ and let $A$ be smooth with values in some Banach space $Z$. We further assume that $A$ and all its derivatives are bounded  on the range of $(b,\sigma)$. Then, we have that
  \[
    A\circ(b,\sigma)\in S^{0,\infty}_{\vec{a}}(\R^{n-1}\times\Sigma,Z).
  \]

\end{proposition}
\begin{proof} We show by induction on $|\alpha'|+|\gamma|$ that $D_{\xi'}^{\alpha'}D_{\lambda}^\gamma(A\circ(b,\sigma))$ is a linear combination of terms of the form $(D_{\xi'}^{\widetilde{\alpha}'}D_{\lambda}^{\widetilde{\gamma}}A)\circ(b,\sigma)\cdot f$ with $f\in S^{-a_1|\alpha'|-a_2|\gamma|,\infty}_{\vec{a}}(\R^{n-1}\times\Sigma)$, $\tilde{\alpha}'\in \N^{n-1}$ and $\tilde{\gamma}\in\N^2$. Obviously, this is true for $|\alpha'|+|\gamma|=0$. So let $j\in\{1,\ldots,n-1\}$. By induction hypothesis, we have that $D_{\xi'}^{\alpha'}D_{\lambda}^kA\circ(b,\sigma)$ is a linear combination of terms of the form $(D_{\xi'}^{\widetilde{\alpha}'}D_{\lambda}^{\tilde{\gamma}}A)\circ(b,\sigma)\cdot f$ with $f\in S^{-a_1|\alpha'|-a_2|\gamma|,\infty}_{\vec{a}}(\R^{n-1}\times\Sigma)$, $\tilde{\alpha}'\in \N^{n-1}$ and $\tilde{\gamma}\in\N^2$. Hence, for $D_{\xi_j}D_{\xi'}^{\alpha'}D_{\lambda}^{\gamma}A\circ(b,\sigma)$ it suffices to treat the summands separately, i.e. we consider $D_{\xi_j}((D_{\xi'}^{\tilde{\alpha}'}D_{\lambda}^{\tilde{\gamma}}A)\circ(b,\sigma)\cdot f)$. By the product rule and the chain rule, we have
 \begin{align*}
  &\quad\;D_{\xi_j}((D_{\xi'}^{\tilde{\alpha}'}D_{\lambda}^{\tilde{\gamma}}A)\circ(b,\sigma)\cdot f)\\
  &=(D_{\xi'}^{\tilde{\alpha}'}D_{\lambda}^{\tilde{\gamma}}A)\circ(b,\sigma))(D_{\xi_j}f)+\bigg(\sum_{l=2}^{n}D_{\xi_j}( \tfrac{\xi'_l}{\rho})\cdot f \cdot [(D_{\xi_l} D_{\xi'}^{\widetilde{\alpha}'}D_{\lambda}^{\tilde{\gamma}} A)\circ(b,\sigma)]\bigg)\\
  &\quad+D_{\xi_j}( \tfrac{1+\lambda}{\rho^{2m}})\cdot f \cdot [(D_{\xi_l} D_{\xi'}^{\widetilde{\alpha}'}D_{\lambda}^{\tilde{\gamma}} A)\circ(b,\sigma)]
 \end{align*}
By the induction hypothesis and Remark \ref{Boutet:rem:CommentsOnSymbols} (\ref{Boutet:rem:SymbolDerivatives}) and (\ref{Boutet:rem:SymbolProducts}) we have that $$(D_{\xi_j}f),(D_{\xi_j}\frac{\xi_1'}{\rho})f,\ldots,(D_{\xi_j} \frac{\xi'_{n-1}}{\rho})f,(D_{\xi_j} \tfrac{1+\lambda}{\rho^{2m}}) f\in S^{-a_1(|\alpha'|+1)-a_2|\gamma|,\infty}_{\vec{a}}(\R^{n-1}\times\Sigma).$$ The same computation for $\partial_{\lambda_1}$ and $\partial_{\lambda_2}$ instead of $\partial_j$ also shows the desired behavior and hence, the induction is finished. Finally, the assertion follows now from Proposition \ref{Boutet:Proposition:VariableRescalingSymbolClass} and Remark \ref{Boutet:rem:CommentsOnSymbols} (\ref{Boutet:rem:SymbolDerivatives}) and (\ref{Boutet:rem:SymbolProducts}).
\end{proof}

\begin{lemma}\label{Boutet:Lemma:PoissonSymbolDerivative}
 Let $n_1,n_2\in\R$ and $a=(a_1,a_2)=(\tfrac{1}{2m},1)$. Let further $f_0\in S^{n_1,\infty}_{\vec{a}}(\R^{n-1}\times\Sigma,\mc{B}( X^{2m}, X^{2m}))$ and $g_0\in S^{n_2,\infty}_{\vec{a}}(\R^{n-1}\times\Sigma,\mc{B}(X,X^{2m}))$. Then, for all $\alpha'\in\N^{n-1}$ and $\gamma\in\N^2$ we have that
 \[
  \partial_{\xi'}^{\alpha'}\partial_{\lambda}^{\gamma} f_0 \exp(i\rho A_0(b,\sigma)x_1)\ms{P}_-(b,\sigma)g_0
 \]
 is a linear combination of terms of the form
  \[
     f\exp(i\rho A_0(b,\sigma)x_1)\ms{P}_-(b,\sigma)gx_1^{p_1+p_2}
  \]
 where $f\in S^{n_1-a_1|\tilde{\alpha}'|-a_2|\tilde{\gamma}|+(a_1-a_2)p_2,\infty}_{\vec{a}}(\R^{n-1}\times\Sigma,\mc{B }( X^{2m}, X^{2m}))$, $g\in S^{n_2-a_1|\bar{\alpha}'|-a_2|\bar{\gamma}|,\infty}_{\vec{a}}(\R^{n-1}\times\Sigma,\mc{B}(X,X^{2m}))$, $|\tilde{\gamma}|+|\bar{\gamma}|+p_2= |\gamma|$ and $|\tilde{\alpha}'|+|\bar{\alpha}'|+p_1=|\alpha'|$.
\end{lemma}
\begin{proof}
 We show the assertion by induction on $|\alpha'|+|\gamma|$. Obviously, for $|\alpha'|+|\gamma|=0$ the assertion holds true. So let $\alpha'\in\N^{n-1}$ and $\gamma\in\N^2$. Let further
 \[
  \partial_{\xi'}^{\alpha'}\partial_{\lambda}^{\gamma} f_0 \exp(i\rho A_0(b,\sigma)x_1)\ms{P}_-(b,\sigma)g_0
 \]
be a linear combination of terms of the form
 \[
     f\exp(i\rho A_0(b,\sigma)x_1)\ms{P}_-(b,\sigma)gx_1^{p_1+p_2}
  \]
 where $f\in S^{n_1-a_1|\tilde{\alpha}'|-a_2|\tilde{\gamma}|+(a_1-a_2)p_2,\infty}_{\vec{a}}(\R^{n-1}\times\Sigma,\mc{B }( X^{2m}, X^{2m}))$, $g\in S^{n_2-a_1|\bar{\alpha}'|-a_2|\bar{\gamma}|,\infty}_{\vec{a}}(\R^{n-1}\times\Sigma,\mc{B}(X,X^{2m}))$, $|\tilde{\gamma}|+|\bar{\gamma}|+p_2= |\gamma|$ and $|\tilde{\alpha}'|+|\bar{\alpha}'|+p_1=|\alpha'|$. We treat the summands separately. Then, for $j=1,\ldots,n-1$ we have
 \begin{align*}
  \partial_j [fe^{i\rho A_0(b,\sigma)x_1}\ms{P}_-(b,\sigma)gx_1^{p_1+p_2}]&=\partial_j [f\ms{P}_-(b,\sigma)e^{i\rho A_0(b,\sigma)x_1}\ms{P}_-(b,\sigma)gx_1^{p_1+p_2}]\\
  &=[\partial_j f]\ms{P}_-(b,\sigma)e^{i\rho A_0(b,\sigma)x_1}\ms{P}_-(b,\sigma)gx_1^{p_1+p_2}\\
  &+ f[\partial_j \ms{P}_-(b,\sigma)]e^{i\rho A_0(b,\sigma)x_1}\ms{P}_-(b,\sigma)gx_1^{p_1+p_2}\\
  &+ fe^{i\rho A_0(b,\sigma)x_1}\ms{P}_-(b,\sigma)[\partial_ji\rho A_0(b,\sigma)]\ms{P}_-(b,\sigma)gx_1^{p_1+p_2+1}\\
  &+ f\ms{P}_-(b,\sigma)e^{i\rho A_0(b,\sigma)x_1}\ms{P}_-(b,\sigma)\partial_j[\ms{P}_-(b,\sigma)]gx_1^{p_1+p_2}\\
  &+ f\ms{P}_-(b,\sigma)e^{i\rho A_0(b,\sigma)x_1}\ms{P}_-(b,\sigma)[\partial_jg]x_1^{p_1+p_2}.
 \end{align*}
Here, we used that the spectral projection $\ms{P}_-(b,\sigma)$ commutes with $e^{i\rho A_0(b,\sigma)x_1}$. Using Remark \ref{Boutet:rem:CommentsOnSymbols} (\ref{Boutet:rem:SymbolDerivatives}) and (\ref{Boutet:rem:SymbolProducts}) and Proposition \ref{Boutet:Proposition:CompositionHomogeneousBounded} we obtain that in each summand, we have that either $|\tilde{\alpha}'|$, $|\bar{\alpha}'|$ or $p_1$ increases by $1$. The same computation as above also yields the desired estimate for $\partial_{\lambda_1}$ and $\partial_{\lambda_2}$, where either $|\tilde{\gamma}|$. $|\bar{\gamma}|$ or $p_2$ increases by $1$. Hence, we obtain the assertion.
\end{proof}

\begin{proposition}\label{Boutet:Proposition:PoissonSymbolOrderSingularity}
 Let again $a=(a_1,a_2)=(\tfrac{1}{2m},1)$. Then, we have the estimate
  \[
   \| x_1^rD_{x_1}^kD_{\xi'}^{\alpha'}D_{\lambda}^{\gamma}e^{i\rho A_0(b,\sigma)x_1}M_{\rho,j}(b,\sigma)\|_{\mc{B}(X,X^{2m})}\leq C \rho^{k-m_{j}-r-|\alpha'|-\tfrac{a_2}{a_1}|\gamma|}e^{-\tfrac{c}{2}\rho x_1}.
  \]
  for all $r,k\in\N$, $\alpha'\in\N^{n-1}$ and $\gamma\in\N^2$.
\end{proposition}
\begin{proof}
  By Lemma \ref{Boutet:Lemma:PoissonSymbolDerivative}, we have that $D_{\xi'}^{\alpha'}D_{\lambda}^{\gamma} e^{i\rho A_0(b,\sigma)x_1}M_{\rho,j}(b,\sigma)$ is a linear combination of terms of the form
  \[
   fe^{i\rho A_0(b,\sigma)x_1}\ms{P}_-(b,\sigma)gx_1^{p_1+p_2}
  \]
where $f\in S^{-a_1|\tilde{\alpha}'|-a_2|\tilde{\gamma}|+p_2(a_1-a_2),\infty}_{\vec{a}}(\R^{n-1}\times\Sigma,\mc{B }(X^{2m}, X^{2m}))$, $g\in S^{-a_1m_j-a_1|\bar{\alpha}'|-a_2|\bar{\gamma}|,\infty}_{\vec{a}}(\R^{n-1}\times\Sigma,\mc{B}(X,X^{2m^i}))$, $|\tilde{\gamma}|+|\bar{\gamma}|+p_2=|\gamma|$ and $|\tilde{\alpha}'|+|\bar{\alpha}'|+p_1=|\alpha'|$. But for such a term, we have that
{\allowdisplaybreaks
\begin{align*}
 &\quad\; \| x_1^{r}D_{x_1}^kf(\xi',\mu)e^{i\rho A_0(b,\sigma)x_1}\ms{P}_-(b,\sigma)g(\xi',\mu)x_1^{p_1+p_2}\|_{\mc{B}(X,X^{2m})}\\
  &\leq C x_1^{r} \sum_{l=0}^k \|f(\xi',\mu)e^{i\rho A_0(b,\sigma)x_1}\ms{P}_-(b,\sigma)[i\rho A_0(b,\sigma)]^{k-l}g(\xi',\mu)x_1^{[p_1+p_2-l]_+}\|_{\mc{B}(X,X^{2m})}\\
  &\leq C x_1^{r} \sum_{l=0}^k \rho^{k-l-m_{j}-|\tilde{\alpha}'|-|\bar{\alpha}'|-\frac{a_2}{a_1}(|\tilde{\gamma}|+|\bar{\gamma}|)+p_2-\frac{a_2}{a_1}p_2}e^{-c\rho x_1}x_1^{[p_1+p_2-l]_+}\\
  &\leq C x_1^{r} \sum_{l=0}^k \rho^{k-l-m_{j}-[p_1+p_2-l]_+-r-|\tilde{\alpha}'|-|\bar{\alpha}'|-\frac{a_2}{a_1}(|\tilde{\gamma}|+|\bar{\gamma}|)+p_2-\frac{a_2}{a_1}p_2}e^{-\tfrac{c}{2}\rho x_1}x_1^{[p_1+p_2-l]_+-[p_1+p_2-l]_+-r}\\
  &\leq C  \sum_{l=0}^k \rho^{k-m_{j}-r-(|\tilde{\alpha}'|+|\bar{\alpha}'|+p_1)-\frac{a_2}{a_1}(\tilde{k}+\bar{k}+p_2)}e^{-\tfrac{c}{2}\rho x_1}\\
  &\leq C \rho^{k-m_{j}-r-|\alpha'|-\frac{a_2}{a_1}|\gamma|}e^{-\tfrac{c}{2}\rho x_1}
\end{align*}}
This is the desired estimate.
\end{proof}
\begin{corollary}\label{Boutet:Proposition:PoissonSymbolDegree}
We have that $$[(x_1,\xi',\lambda)\mapsto e^{i\rho A_0(b,\sigma)x_1}M_{\rho,j}(b,\sigma)]\in S_{(1/2m,1)}^{-(1+m_{j})/2m,\infty}(\R^{n-1}\times\Sigma;S_{L_1}(\R_+,\mathcal{B}(X,X^{2m}))).$$
\end{corollary}
\begin{proof}
 This is obtained by computing the $L_1$-norms in Proposition \ref{Boutet:Proposition:PoissonSymbolOrderSingularity}.
\end{proof}

Putting together the above gives Proposition \ref{Boutet:prop:Poisson_operator_DBVP}:
\begin{proof}[Proof of Proposition \ref{Boutet:prop:Poisson_operator_DBVP}]
A combination of the solution formula~\eqref{Boutet:Equation:SolutionFormulaPoisson} and Corollary~\ref{Boutet:Proposition:PoissonSymbolDegree} gives the desired result, where the uniqueness statement is clear from the construction of the solution formula.
\end{proof}

\subsection{Mapping Properties}\label{Boutet:subsec:sec:Poisson;mapping_prop}

Recall the notation from Section~\ref{Boutet:subsubsec:trace_results}.

\begin{theorem}\label{Boutet:thm:mapping_anisotropic_Poisson}
Let $X$ be a Banach space, $\mathpzc{d}_{1}=1$, $\vec{p} \in (1,\infty)^{l}$, $r \in [1,\infty]$, $\gamma \in (-1,\infty)$, $\vec{w}' \in \prod_{j=2}^{l}A_{\infty}(\R^{\mathpzc{d}_{j}})$, $s \in \R$ and $\vec{a} \in (0,\infty)^{l}$.
Then $(\widetilde{k},g) \mapsto \operatorname{OPK}(\widetilde{k})g$ defines continuous bilinear operators
\[
S^{d-a_{1}}_{\mathpzc{d},\vec{a}}(\R^{n-1};\mathcal{S}_{L_{1}}(\R_{+};\mathcal{B}(X))) \times F^{s+d-a_{1}\frac{1+\gamma}{p_{1}},\vec{a}'}_{\vec{p}',p_{1},\mathpzc{d}'}(\R^{n-1},\vec{w}';X) \longra
F^{s,\vec{a}}_{\vec{p},r,\mathpzc{d}}(\R^{n}_{+},(w_{\gamma},\vec{w}');X)
\]
and
\[
S^{d-a_{1}}_{\mathpzc{d},\vec{a}}(\R^{n-1};\mathcal{S}_{L_{1}}(\R_{+};\mathcal{B}(X))) \times B^{s+d-a_{1}\frac{1+\gamma}{p_{1}},\vec{a}'}_{\vec{p}',r,\mathpzc{d}'}(\R^{n-1},\vec{w}';X) \longra
B^{s,\vec{a}}_{\vec{p},r,\mathpzc{d}}(\R^{n}_{+},(w_{\gamma},\vec{w}');X).
\]
\end{theorem}
Note that Theorem~\ref{Boutet:thm:mapping_anisotropic_Poisson} is an extension of the isotropic unweighted scalar-valued setting in \cite[Theorem~4.3]{Johnsen1996} in case of constant coefficients. But in contrast to \cite[Theorem~4.3]{Johnsen1996} we take $\vec{p}\in(1,\infty)^l$. Since we use Proposition~\ref{PEBVP:prop:Sobolev_embedding_Besov} in the proof we do not allow $\vec{p}\in[1,\infty)^l$ in this formulation. However, it should be possible to remove this restriction, see Remark~\ref{PEBVP:rmk:p=1}. The proof of Theorem~\ref{Boutet:thm:mapping_anisotropic_Poisson}, which adjusts the line of arguments in \cite[Theorem~4.3]{Johnsen1996} to our situation, will be given at the end of this section. 

\begin{corollary}\label{Boutet:cor:thm:mapping_anisotropic_Poisson;intersection_space}
Let $X$ be a UMD Banach space, $q,p,r \in (1,\infty)$, $v \in A_{q}(\R)$, $\gamma \in (-1,\infty)$, $s \in \R$ and $\rho \in (0,\infty)$. Let $\mathpzc{d}=(1,n-1,1)$ and $\vec{a}=(\frac{1}{\rho},\frac{1}{\rho},1)$. Then $(\widetilde{k},g) \mapsto \operatorname{OPK}(\widetilde{k})g$ defines a continuous bilinear operator
\begin{align*}
S^{d-\frac{1}{\rho}}_{\mathpzc{d},\vec{a}}(\R^{n};\mathcal{S}_{L_{1}}(\R_{+};\mathcal{B}(X))) \times & F^{\frac{1}{\rho}(s-\frac{1+\gamma}{p})+d,(\frac{1}{\rho},1)}_{(p,q),p}(\R^{n-1} \times \R,(1,v);X) \\ & \longra W^{1}_{q}(\R,v;F^{s}_{p,r}(\R^{n}_{+},w_{\gamma};X)) \cap L_{q}(\R,v;F^{s+\rho}_{p,r}(\R^{n}_{+},w_{\gamma};X)).
\end{align*}
\end{corollary}
\begin{proof}
Let $\tilde{s}$, $\tilde{\gamma}$, $\sigma$ and $\eta$ be as in Lemma~\ref{DSOP:lemma:embd_anisotropic_MR}. Then note that we have the embedding \eqref{DSOP:eq:lemma:embd_anisotropic_MR} while
\begin{equation*}\label{Boutet:lemma:DPBVP_bd-data;reg_bd-data}
\left(\sigma+\frac{\eta}{\rho}\right)+d-\frac{1}{\rho}\frac{1+(\tilde{\gamma}+\eta p)}{p} = \frac{1}{\rho}(s-\frac{1+\gamma}{p})+d.
\end{equation*}
Observing that
\[
F^{\sigma+\frac{\eta}{\rho},(\frac{1}{\rho},1)}_{(p,q),1}(\R^{n}_{+} \times \R,(w_{\tilde{\gamma}+\eta p},v);X) = F^{\sigma+\frac{\eta}{\rho},(\frac{1}{\rho},\frac{1}{\rho},1)}_{(p,p,q),1}(\R_{+} \times \R^{n-1} \times \R,(w_{\tilde{\gamma}+\eta p},1,v);X),
\]
the result thus follows from Theorem~\ref{Boutet:thm:mapping_anisotropic_Poisson}.
\end{proof}

Combined with the elementary embedding \eqref{eq:identity_Bessel-Potential_Triebel-Lizorkin;scalar-valued}
\eqref{DSOP:eq:prelim:emb_F_into_Lp}, the above corollary (with $r=1$) yields the following two results.
\begin{corollary}\label{Boutet:cor:thm:mapping_anisotropic_Poisson;intersection_space;H}
Let $X$ be a UMD Banach space, $q,p\in (1,\infty)$, $v \in A_{q}(\R)$, $\gamma \in (-1,p-1)$, $s \in \R$ and $\rho \in (0,\infty)$. Let $\mathpzc{d}=(1,n-1,1)$ and $\vec{a}=(\frac{1}{\rho},\frac{1}{\rho},1)$. Then $(\widetilde{k},g) \mapsto \operatorname{OPK}(\widetilde{k})g$ defines a continuous bilinear operator
\begin{align*}
S^{d-\frac{1}{\rho}}_{\mathpzc{d},\vec{a}}(\R^{n};\mathcal{S}_{L_{1}}(\R_{+};\mathcal{B}(X))) \times & F^{\frac{1}{\rho}(s-\frac{1+\gamma}{p})+d,(\frac{1}{\rho},1)}_{(p,q),p}(\R^{n-1} \times \R,(1,v);X) \\ & \longra W^{1}_{q}(\R,v;H^{s}_{p}(\R^{n}_{+},w_{\gamma};X)) \cap L_{q}(\R,v;H^{s+\rho}_{p}(\R^{n}_{+},w_{\gamma};X)).
\end{align*}
\end{corollary}

\begin{corollary}\label{Boutet:cor:thm:mapping_anisotropic_Poisson;intersection_space;W}
Let $X$ be a UMD Banach space, $q,p \in (1,\infty)$, $v \in A_{q}(\R)$, $\gamma \in (-1,\infty)$ and $k \in \N_1$. Let $\mathpzc{d}=(1,n-1,1)$ and $\vec{a}=(\frac{1}{k},\frac{1}{k},1)$. Then $(\widetilde{k},g) \mapsto \operatorname{OPK}(\widetilde{k})g$ defines a continuous bilinear operator
\begin{align*}
S^{d-\frac{1}{k}}_{\mathpzc{d},\vec{a}}(\R^{n};\mathcal{S}_{L_{1}}(\R_{+};\mathcal{B}(X))) \times & F^{-\frac{1+\gamma}{kp}+d,(\frac{1}{k},1)}_{(p,q),p}(\R^{n-1} \times \R,(1,v);X) \\ & \longra W^{1}_{q}(\R,v;L_{p}(\R^{n}_{+},w_{\gamma};X)) \cap L_{q}(\R,v;W^{k}_{p}(\R^{n}_{+},w_{\gamma};X)).
\end{align*}
\end{corollary}

\begin{remark}\label{Boutet:rmk:cor:thm:mapping_anisotropic_Poisson;intersection_space}
Corollary~\ref{Boutet:cor:thm:mapping_anisotropic_Poisson;intersection_space;W} could also directly be derived from Theorem~\ref{Boutet:thm:mapping_anisotropic_Poisson} using the elementary embedding (\cite[Lemma~7.2]{LV2018_Dir_Laplace})
\begin{align*}
&F^{1,(\frac{1}{k},\frac{1}{k},1)}_{(p,p,q),1}(\R_{+} \times \R^{n-1} \times \R,(w_{\gamma},1,v);X)  = F^{1,(\frac{1}{k},1)}_{(p,q),1}(\R^{n}_{+} \times \R,(w_{\gamma},v);X) \\
& \qquad\qquad \hookrightarrow W^{1}_{q}(\R,v;L_{p}(\R^{n}_{+},w_{\gamma};X)) \cap L_{q}(\R,v;W^{k}_{p}(\R^{n}_{+},w_{\gamma};X)).
\end{align*}
\end{remark}

A combination of the real interpolation results Lemma~\ref{Boutet:prelim:lemma_real_interpolation_max-reg_B} and \eqref{Boutet:intro:eq:real_int_F-spaces} with Corollary~\ref{Boutet:cor:thm:mapping_anisotropic_Poisson;intersection_space} yields the following:
\begin{corollary}\label{Boutet:cor:thm:mapping_anisotropic_Poisson;intersection_space;B}
Let $X$ be a UMD Banach space, $q,p \in (1,\infty)$, $v \in A_{q}(\R)$, $\gamma \in (-1,\infty)$, $s \in \R$ and $\rho \in (0,\infty)$. Let $\mathpzc{d}=(1,n-1,1)$ and $\vec{a}=(\frac{1}{\rho},\frac{1}{\rho},1)$. Then $(\widetilde{k},g) \mapsto \operatorname{OPK}(\widetilde{k})g$ defines a continuous bilinear operator
\begin{align*}
S^{d-\frac{1}{\rho}}_{\mathpzc{d},\vec{a}}(\R^{n};\mathcal{S}_{L_{1}}(\R_{+};\mathcal{B}(X))) \times & B^{\frac{1}{\rho}(s-\frac{1+\gamma}{p})+d,(\frac{1}{\rho},1)}_{(p,q),q}(\R^{n-1} \times \R,(1,v);X) \\ & \longra W^{1}_{q}(\R,v;B^{s}_{p,q}(\R^{n}_{+},w_{\gamma};X)) \cap L_{q}(\R,v;B^{s+\rho}_{p,q}(\R^{n}_{+},w_{\gamma};X)).
\end{align*}
\end{corollary}

\begin{theorem}\label{Boutet:thm:Poisson_par-dep_mapping_prop}
Let $X$ be a reflexive Banach space, $\Sigma\subset\C$ open, $d\in\R$ and $p,q \in (1,\infty)$. Let further $(\mathscr{A},\gamma) \in \{B,F\} \times (-1,\infty) \cup \{\mathcal{B},\mathcal{F}\} \times (-\infty,p-1)$, $s \in \R$ and $s_{0} \in (\frac{1+\gamma}{p}-1,\infty)$.
Then $(\widetilde{k}_\mu,g) \mapsto \operatorname{OPK}(\widetilde{k}_\mu)g$ defines a continuous bilinear operator
\[
S^{d-1,\infty}(\R^{n-1}\times\Sigma;\mathcal{S}_{L_{1}}(\R_{+};\mathcal{B}(X))) \times \partial\mathscr{A}^{s,|\mu|}_{p,q,\gamma}(\R^{n}_{+};X) \longra
\mathscr{A}^{s-d,|\mu|,s_0}_{p,q}(\R^{n}_{+};X)
\]
uniformly in $\mu$.
\end{theorem}
 We should refer the reader to \cite{Grubb&Kokholm1993} where related results have been obtained in the unweighted isotropic finite-dimensional $L_p$-setting over manifolds. The proof of Theorem~\ref{Boutet:thm:Poisson_par-dep_mapping_prop}, which is again inspired by  \cite[Theorem~4.3]{Johnsen1996}, will be given at the end of this section.\\

We also obtain the same assertion for the Bessel potential scale by the elementary embedding $F^{s-d,|\mu|,s_0}_{p,1,\gamma}(\R^n_+;X)\hookrightarrow H^{s-d,|\mu|,s_0}_{p,\gamma}(\R^n_+;X)$ for $\gamma\in(-1,p-1)$. Hence, the following corollary holds:

\begin{corollary}
Let $X$ be a reflexive Banach space, $\Sigma\subset\C$ open, $d\in\R$ and $p\in (1,\infty)$. Let further $\gamma\in(-1,p-1)$, $s \in \R$ and $s_{0} \in (\frac{1+\gamma}{p}-1,\infty)$.
Then $(\widetilde{k}_\mu,g) \mapsto \operatorname{OPK}(\widetilde{k}_\mu)g$ defines a continuous bilinear operator
\[
S^{d-1,\infty}(\R^{n-1}\times\Sigma;\mathcal{S}_{L_{1}}(\R_{+};\mathcal{B}(X))) \times \partial H^{s,|\mu|}_{p,\gamma}(\R^{n}_{+};X) \longra
H^{s-d,|\mu|,s_0}_{p}(\R^{n}_{+};X)
\]
uniformly in $\mu$.
\end{corollary}

\begin{lemma}\label{PIBVP:lemma:tensor_Dirac_anisotropic_B&F}
Let $X$ be a Banach space, $\vec{a} \in (0,\infty)^{l}$, $\vec{p} \in [1,\infty)^{l}$, $q \in [1,\infty]$, $\gamma \in (-\mathpzc{d}_{1},\infty)$ and $s \in (-\infty,a_{1}\left[\frac{\mathpzc{d}_{1}+\gamma}{p_{1}}-\mathpzc{d}_{1}\right])$.
Let $\vec{w} \in \prod_{j=1}^{l}A_{\infty}(\R^{\mathpzc{d}_{j}})$ be such that
$w_{1}(x_{1}) = |x_{1}|^{\gamma}$.
The linear operator
\[
T: \mathcal{S}'(\R^{n-\mathpzc{d}_{1}};X) \longra \mathcal{S}'(\R^{n};X),\, f \mapsto \delta_{0} \otimes  f.
\]
restricts to bounded linear operators from $B^{s+a_{1}\left(\mathpzc{d}_{1}-\frac{\mathpzc{d}_{1}+\gamma}{p_{1}}\right),\vec{a}'}_{\vec{p}',q,\mathpzc{d}'}
    (\R^{n-\mathpzc{d}_{1}},\vec{w}';X)$ to $B^{s,\vec{a}}_{\vec{p},q,\mathpzc{d}}(\R^{n},\vec{w};X)$
and from $F^{s+a_{1}\left(\mathpzc{d}_{1}-\frac{\mathpzc{d}_{1}+\gamma}{p_{1}}\right),\vec{a}'}_{\vec{p}',p_{1},\mathpzc{d}'}
    (\R^{n-\mathpzc{d}_{1}},\vec{w}';X)$ to $F^{s,\vec{a}}_{\vec{p},q,\mathpzc{d}}(\R^{n},\vec{w};X)$.
\end{lemma}
\begin{proof}
The Besov case is contained in \cite[Lemma~4.14]{Lindemulder2017_PIBVP}. So let us consider the Triebel-Lizorkin case. Using the Sobolev embedding from Proposition~\ref{PEBVP:prop:Sobolev_embedding_Besov}, we may without loss of generality assume that $p_{1}=q$, so that
\[
L_{\vec{p},\mathpzc{d}}(\R^{n},\vec{w})[\ell_{q}(\N)] = L_{\vec{p}',\mathpzc{d}'}(\R^{n-\mathpzc{d}_{1}},\vec{w}')[\ell_{q}(\N)[L_{p_{1}}(\R^{\mathpzc{d}_{1}},|\,\cdot\,|^{\gamma})]].
\]
Now the desired estimate can be obtained as in the proof of \cite[Lemma~4.14(i)]{Lindemulder2017_PIBVP}.
\end{proof}

\begin{lemma}\label{Boutet:lemma:Hormander_mapping_prop}
Let $X$ be a Banach space, $\vec{a} \in (0,\infty)^l$, $s \in \R$, $\vec{p} \in [1,\infty)^l$, $q \in [1,\infty]$ and $d \in \R$. Then $(f,p) \mapsto \operatorname{OP}(p)f$ defines continuous bilinear mappings
\begin{align*}
S^{d}_{\mathpzc{d},\vec{a}}(\R^{n};\mathcal{B}(X)) \times F^{s,\vec{a}}_{\vec{p},q,\mathpzc{d}}(\R^{n},\vec{w};X)
\longra F^{s-d,\vec{a}}_{\vec{p},q,\mathpzc{d}}(\R^{n},\vec{w};X)
\end{align*}
and
\begin{align*}
S^{d}_{\mathpzc{d},\vec{a}}(\R^{n};\mathcal{B}(X)) \times B^{s,\vec{a}}_{\vec{p},q,\mathpzc{d}}(\R^{n},\vec{w};X)
\longra B^{s-d,\vec{a}}_{\vec{p},q,\mathpzc{d}}(\R^{n},\vec{w};X).
\end{align*}
\end{lemma}
\begin{proof}
This follows directly from the fact that $F^{s,\vec{a},\mathpzc{d}}_{\vec{p},q}(\R^{n},\vec{w};X)$ and $B^{s,\vec{a},\mathpzc{d}}_{\vec{p},q}(\R^{n},\vec{w};X)$ are $(\mathpzc{d},\vec{a})$-admissible Banach spaces of tempered distributions with \eqref{Boutet:eq:prelim:Bessel-pot_B&F}.
\end{proof}

\begin{proof}[Proof of Theorem~\ref{Boutet:thm:mapping_anisotropic_Poisson}]
Let $\widetilde{k} \in S^{d-a_{1}}_{\mathpzc{d},\vec{a}}(\R^{n-1};\mathcal{S}_{L_{1}}(\R_{+};\mathcal{B}(X)))$.
Let $p \in S^{d-a_{1}}_{\mathpzc{d},\vec{a}}(\R^{n-1};\mathscr{S}_{L_{\infty}}(\R;\mathcal{B}(X)))$ be as in Lemma~\ref{Boutet:lemma:Poisson_tensor_Dirac_delta} for this given $\widetilde{k}$; so $\operatorname{OPK}(\tilde{k}) = r_+\operatorname{OP}[p](\delta_0\otimes \,\cdot\,)$.
Then, for every $\sigma \in \R$,
\begin{equation}\label{Boutet:eq:thm:mapping_anisotropic_Poisson;tensor_Dirac}
\operatorname{OPK}(\tilde{k}) = r_+\operatorname{OP}[p]\mathcal{J}^{\mathpzc{d}',\vec{a}'}_{-\sigma}(\delta_0\otimes \,\cdot\,)\mathcal{J}^{\mathpzc{d}',\vec{a}'}_{\sigma} = r_+\operatorname{OP}[p_{\sigma}](\delta_0\otimes \,\cdot\,)\mathcal{J}^{\mathpzc{d}',\vec{a}'}_{\sigma},
\end{equation}
where $p_{\sigma}(\xi) := p(\xi)J^{\mathpzc{d}',\vec{a}'}_{-\sigma}(\xi')$.
By Lemmas \ref{Boutet:lemma:incl_mathscrSLinfty_into_Hormander} and \ref{Boutet:Lemma:SymbolPointwiseMultiplication}, $p \mapsto p_{\sigma}$ defines a continuous linear mapping
\begin{equation}\label{Boutet:eq:thm:mapping_anisotropic_Poisson;Hormander}
S^{d-a_{1}}_{\mathpzc{d},\vec{a}}(\R^{n-1};\mathscr{S}_{L_{\infty}}(\R;\mathcal{B}(X)))
\stackrel{p \mapsto p_{\sigma}}{\longra} S^{d-a_{1}-\sigma}_{\mathpzc{d},\vec{a}}(\R^{n-1};\mathscr{S}_{L_{\infty}}(\R;\mathcal{B}(X))) \hookrightarrow S^{d-a_{1}-\sigma}_{\mathpzc{d},\vec{a}}(\R^{n};\mathcal{B}(X)).
\end{equation}
Choosing $\sigma \in \R$ such that $s-\sigma < a_{1}(\frac{1+\gamma}{p_{1}}-1)$,
a combination of \eqref{Boutet:eq:thm:mapping_anisotropic_Poisson;tensor_Dirac}, \eqref{Boutet:eq:thm:mapping_anisotropic_Poisson;Hormander}, Lemma~\ref{PIBVP:lemma:tensor_Dirac_anisotropic_B&F}, Lemma~\ref{Boutet:lemma:Hormander_mapping_prop} and the lifting property of weighted mixed-norm anisotropic $B$- and $F$-spaces (see \eqref{Boutet:eq:prelim:Bessel-pot_B&F}) gives the desired result. Indeed, we obtain the commutative diagram:
 \[
  \begin{tikzcd}
    F^{s+d-a_{1}\frac{1+\gamma}{p_{1}},\vec{a}'}_{\vec{p}',p_{1},\mathpzc{d}'}(\R^{n-1},\vec{w}';X) \arrow{r}{\mathcal{J}^{\mathpzc{d}',\vec{a}'}_{\sigma}} \arrow[swap]{dd}{\operatorname{OPK}(\tilde{k})}  & F^{s+d-\sigma-a_{1}\frac{1+\gamma}{p_{1}},\vec{a}'}_{\vec{p}',p_{1},\mathpzc{d}'}(\R^{n-1},\vec{w}';X) \arrow{d}{(\delta_0\otimes\,\cdot\,)} \\
     & F^{s+d-\sigma,\vec{a}}_{\vec{p},r,\mathpzc{d}}(\R^{n},(w_{\gamma},\vec{w}');X)\arrow{d}{\operatorname{OP}[p_{\sigma}]} \\
    F^{s,\vec{a}}_{\vec{p},r,\mathpzc{d}}(\R^{n}_{+},(w_{\gamma},\vec{w}');X)  &  F^{s,\vec{a}}_{\vec{p},r,\mathpzc{d}}(\R^{n},(w_{\gamma},\vec{w}');X) \arrow{l}{r_+}\\
  \end{tikzcd}
\]
\end{proof}

\begin{proof}[Proof of Theorem~\ref{Boutet:thm:Poisson_par-dep_mapping_prop}]
 We take $p$ as defined in Lemma \ref{Boutet:lemma:Poisson_tensor_Dirac_delta} so that we have the identity
\[
 r_+\operatorname{OP}(p)(\delta_0\otimes g)=\operatorname{OPK}(\tilde{k})g.
\]
Now, for $\sigma\in\R$ we define
\[
 p^{\sigma}(\xi,\mu):= \langle \xi,\mu \rangle^{s-s_0-d} p(\xi,\mu) \langle \xi,\mu \rangle^{-s+s_0+\sigma+1}\langle \xi',\mu \rangle^{-\sigma}
\]
so that we obtain
\[
 \operatorname{OPK}(\tilde{k}_\mu)=r_+\operatorname{OP}(p_\mu)(\delta_0\otimes \;\cdot\,)=r_+\Xi_\mu^{d+s_0-s}\operatorname{OP}(p^{\sigma}_\mu)\Xi_\mu^{s-\sigma+d-1-s_0}[\delta_0\otimes\;\cdot\,]\Xi_\mu^{\sigma}.
\]
By Lemma \ref{Boutet:Lemma:SymbolPointwiseMultiplication} and Lemma \ref{Boutet:lemma:incl_mathscrSLinfty_into_Hormander} we obtain that
\[
 S^{d-1,\infty}(\R^{n-1}\times\Sigma;\mathcal{S}_{L_1}(\R_+,\mathcal{B}(X)))\to S^{0,\infty}(\R^n\times\Sigma),\,\tilde{k}\mapsto p^{\sigma}
\]
is continuous. We even obtain that $(p^{\sigma}(\,\cdot\,,\mu))_{\mu\in\Sigma}$ defines a bounded family in $S^{0}(\R^n)$. Taking $\sigma>s-\frac{1+\gamma}{p}$ in combination with Corollary \ref{DBVP:cor:prop:par-dep_trace;tensoring_delta} yields the desired result as can be seen in the following commutative diagram
 \[
  \begin{tikzcd}
    \partial\mathscr{A}^{s,|\mu|}_{p,q,\gamma}(\R^{n}_{+};X) \arrow{r}{\Xi_\mu^\sigma} \arrow[swap]{ddd}{\operatorname{OPK}(\tilde{k}_\mu)}  & \partial\mathscr{A}^{s-\sigma,|\mu|}_{p,q,\gamma}(\R^{n}_{+};X) \arrow{d}{(\delta_0\otimes\,\cdot\,)} \\
    & \mathscr{A}^{s-1-\sigma,|\mu|,s_0}_{p,q}(\R^{n},w_{\gamma};X)\arrow{d}{\Xi_\mu^{s-\sigma-1-s_0}} \\
     & \mathscr{A}^{s_0,|\mu|,s_0}_{p,q}(\R^{n},w_{\gamma};X)\arrow{d}{\operatorname{OP}[p^{\sigma}_{\mu}]} \\
   \mathscr{A}^{s-d,|\mu|,s_0}_{p,q,\gamma}(\R^{n}_+;X)  &  \mathscr{A}^{s_0,|\mu|,s_0}_{p,q}(\R^{n},w_{\gamma};X) \arrow{l}{r_+\Xi_\mu^{s_0+d-s}}\\
  \end{tikzcd}
\]
\end{proof}


\section{Parabolic Problems}\label{Boutet:sec:parabolic}

In this section we consider the linear vector-valued parabolic initial-boundary value problem \eqref{Boutet:intro:pibvp}. As the main result of the paper, we solve the $L_{q,\mu}$-$H^{s}_{p,\gamma}$-maximal regularity problem, the $L_{q,\mu}$-$F^{s}_{p,r,\gamma}$-maximal regularity problem and the $L_{q,\mu}$-$B^{s}_{p,q,\gamma}$-maximal regularity problem for \eqref{Boutet:intro:pibvp} in Theorem~\ref{Boutet:theorem:DPIBVP}. This simultaneously generalizes \cite[Theorem~3.4]{Lindemulder2017_PIBVP} and \cite[Theorem~4.2]{Lindemulder2018_DSOP}.

Before we can state Theorem~\ref{Boutet:theorem:DPIBVP}, we first need to introduce some notation.

\subsection{Some notation and assumptions}\label{Boutet:subsec:sec:parabolic;notation}
Let $\mathscr{O}$ be either $\R^{n}_{+}$ or a $C^{N}$-domain in $\R^{n}$ with a compact boundary $\partial\mathscr{O}$, where $N \in \N$ is specified below, and $J = (0,T)$ with $T \in (0,\infty)$.
Let $X$ be a Banach space and let $\mathcal{A}(D),\mathcal{B}_{1}(D),\ldots,\mathcal{B}_{m}(D)$ be a $\mathcal{B}(X)$-valued differential boundary value system on $\mathscr{O} \times J$ as considered in Section~\ref{Boutet:subsec:sec:prelim:E&LS} where the coefficients satisfy certain smoothness conditions which we are going to introduce later. Put $m^{*} := \max\{m_{1},\ldots,m_{m}\}$ and $m_{*} := \min\{m_{1},\ldots,m_{m}\}$.

Let $q \in (1,\infty)$ and $\mu\in(-1,q-1)$. Let $\E$ and $\E^{2m}$ be given as either
\begin{enumerate}[(a)]
\item\label{Boutet:it:parabolic:H} $\E=H^{s}_{p,\gamma}(\mathscr{O};X)$ and $\E^{2m}=H^{s+2m}_{p,\gamma}(\mathscr{O};X)$ with $p \in (1,\infty)$, $\gamma \in (-1,p-1)$ and $s \in (\frac{1+\gamma}{p}+m^{*}-2m,\frac{1+\gamma}{p}+m_{*})$ \emph{(the Bessel potential case)}; or
\item\label{Boutet:it:parabolic:F} $\E=F^{s}_{p,r,\gamma}(\mathscr{O};X)$ and $\E^{2m}=F^{s+2m}_{p,r,\gamma}(\mathscr{O};X)$ with $p,r \in (1,\infty)$, $\gamma \in (-1,\infty)$ and $s \in (\frac{1+\gamma}{p}+m^{*}-2m,\frac{1+\gamma}{p}+m_{*})$ \emph{(the Triebel-Lizorkin case)},
\item\label{Boutet:it:parabolic:B} $\E=B^{s}_{p,q,\gamma}(\mathscr{O};X)$ and $\E^{2m}=B^{s+2m}_{p,q,\gamma}(\mathscr{O};X)$ with $p \in (1,\infty)$, $\gamma \in (-1,\infty)$ and $s \in (\frac{1+\gamma}{p}+m^{*}-2m,\frac{1+\gamma}{p}+m_{*})$ \emph{(the Besov case)},
\end{enumerate}
and set
\begin{equation*}
\kappa_{j,\E} = \kappa_{j,p,\gamma,s} := \frac{s+2m-m_{j}}{2m}-\frac{1+\gamma}{2mp}\in(0,1), \qquad j=1,\ldots,m
\end{equation*}
as well as
\begin{align*}
 \mathbb{E}^{\sigma}:=[\mathbb{E},\mathbb{E}^{2m}]_{\sigma/2m}=
 \begin{cases}
 H^{s+\sigma}_{p,\gamma}(\mathscr{O};X)\quad\text{in case \eqref{Boutet:it:parabolic:H}},\\
 F^{s+\sigma}_{p,r,\gamma}(\mathscr{O};X)\quad\text{in case \eqref{Boutet:it:parabolic:F}},\\
 B^{s+\sigma}_{p,q,\gamma}(\mathscr{O};X)\quad\text{in case \eqref{Boutet:it:parabolic:B}},
 \end{cases}
\end{align*}
for $\sigma\in[0,2m]$. Consider the following assumptions on $\mathscr{O}$ and $\mathcal{A}(D),\mathcal{B}_{1}(D),\ldots,\mathcal{B}_{m}(D)$:
\begin{enumerate}
\item[\textbf{$(\mathrm{SO})$}] 
\begin{description}
\item[Case \eqref{Boutet:it:parabolic:H}] $\mathscr{O}$ is $C^N$ with $N \geq \max\{s+2m,-s\}$.
\item[Case \eqref{Boutet:it:parabolic:F} \& \eqref{Boutet:it:parabolic:B}] $\mathscr{O}$ is $C^N$ with $N > \max\big\{s+2m,\left(\frac{1+\gamma}{p}-1\right)_++1-s\big\}$.
\end{description}
\item[\textbf{$(\mathrm{SAP})$}]
\begin{description}
\item[Case \eqref{Boutet:it:parabolic:H} with $s=0$] For $|\alpha| = 2m$ we have $a_{\alpha} \in BUC(\mathscr{O} \times J;\mathcal{B}(X))$. If $\mathscr{O}$ is unbounded, the limits $a_{\alpha}(\infty,t) := \lim_{|x| \to \infty}a_{\alpha}(x,t)$ exist uniformly with respect to $t \in J$, $|\alpha|=2m$.
\item[Case \eqref{Boutet:it:parabolic:H} with $s\neq0$] Let $\sigma>|s|$. For $|\alpha| = 2m$ we have $a_{\alpha} \in BUC(J;BUC^{\sigma}(\mathscr{O};\mathcal{B}(X))$ such that for all $\epsilon>0$ there is a $\delta>0$ such that for all $x_0\in\mathscr{O}$ and all $t\in J$ it holds that
\[
 \mathcal{R}(\{a_{\alpha}(x,t)-a_{\alpha}(x_0,t): x\in B(x_0,\delta)\cap\mathscr{O}\})<\epsilon.
\]
If $\mathscr{O}$ is unbounded, then for all $|\alpha|=2m$ the limits $a_{\alpha}(\infty,t) := \lim_{|x| \to \infty}a_{\alpha}(x,t)$ exist uniformly with respect to $t \in J$ and for all $\epsilon>0$ there is an $R>0$ such that
\[
 \mathcal{R}(\{a_{\alpha}(x,t)-a_{\alpha}(\infty,t): x\in\mathscr{O},\,|x|\geq R\})<\epsilon.
\]
\item[Case \eqref{Boutet:it:parabolic:F} \& \eqref{Boutet:it:parabolic:B}] For $|\alpha| = 2m$ we have $a_{\alpha} \in BUC(J;BUC^{\sigma}(\mathscr{O};\mathcal{B}(X))$ where $\sigma>\sigma_{s,p,\gamma}$. If $\mathscr{O}$ is unbounded, the limits $a_{\alpha}(\infty,t) := \lim_{|x| \to \infty}a_{\alpha}(x,t)$ exist uniformly with respect to $t \in J$, $|\alpha|=2m$.
\end{description}
\item[\textbf{$(\mathrm{SAL})$}]  For the lower order parts of $\mathcal{A}$ we only need $a_{\alpha}D^{\alpha}$, $|\alpha| < 2m$ to act as lower order perturbations in the sense that there exists $\sigma \in [2m-1,2m)$ such that
$a_{\alpha}D^{\alpha}$ is bounded from
\[
H^{\frac{\sigma}{2m}}_{q}(J,v_{\mu};\E) \cap L_{q}(J,v_{\mu};\E^{\sigma})
\]
to $L_{q}(J,v_{\mu};\E)$.
\item[\textbf{$(\mathrm{SBP})$}]
For each $j \in \{1,\ldots,m\}$ and $|\beta| = m_{j}$ there exist $s_{j,\beta} \in [q,\infty)$ and $r_{j,\beta} \in [p,\infty)$ with
\[
\kappa_{j,\E} > \frac{1}{s_{j,\beta}} + \frac{n-1}{2mr_{j,\beta}}
\quad \mbox{and} \quad \mu > \frac{q}{s_{j,\beta}}-1
\]
such that
\[
b_{j,\beta}\in\begin{cases}
 F^{\kappa_{j,\E}}_{s_{j,\beta},p}(J;L_{r_{j,\beta}}(\partial\mathscr{O};\mathcal{B}(X))) \cap L_{s_{j,\beta}}(J;B^{2m\kappa_{j,\E}}_{r_{j,\beta},p}(\partial\mathscr{O};\mathcal{B}(X)))\quad&\text{in cases \eqref{Boutet:it:parabolic:H} and \eqref{Boutet:it:parabolic:F}},\\
   F^{\kappa_{j,\E}}_{s_{j,\beta},q}(J;L_{r_{j,\beta}}(\partial\mathscr{O};\mathcal{B}(X))) \cap L_{s_{j,\beta}}(J;B^{2m\kappa_{j,\E}}_{r_{j,\beta},q}(\partial\mathscr{O};\mathcal{B}(X)))\quad&\text{in case \eqref{Boutet:it:parabolic:B}}.
\end{cases}
\]
If $\mathscr{O}=\R^{n}_{+}$, the limits $b_{j,\beta}(\infty,t) := \lim_{|x'| \to \infty}b_{j,\beta}(x',t)$ exist uniformly with respect to $t \in J$, $j \in \{1,\ldots,m\}$, $|\beta|=m_{j}$.
\item[\textbf{$(\mathrm{SBL})$}]
There is a $\sigma\in[2m-1,2m)$ such that $b_{j,\beta}\operatorname{tr}_{\partial\mathscr{O}}D^{\beta}$ is bounded from
\[
  H^{\frac{\sigma}{2m}}_{q,\mu}(J;\E) \cap L_{q,\mu}(J;\E^{\sigma})
\]
to
    \begin{equation}\label{Boutet:eq:SBL}
    \begin{cases}
 F^{\kappa_{j,\E}}_{q,p,\mu}(J;L_{p}(\partial\mathscr{O};X)) \cap L_{q,\mu}(J;F^{2m\kappa_{j,\E}}_{p,p}(\partial\mathscr{O};X)) \quad\text{in cases \eqref{Boutet:it:parabolic:H} and \eqref{Boutet:it:parabolic:F}},\\
 B^{\kappa_{j,\E}}_{q,q,\mu}(J;L_{p}(\partial\mathscr{O};X)) \cap L_{q,\mu}(J;B^{2m\kappa_{j,\E}}_{p,q}(\partial\mathscr{O};X))\quad\text{in case \eqref{Boutet:it:parabolic:B}},
 \end{cases}
   \end{equation}
\end{enumerate}

\begin{example}\label{Boutet:ex:smoothness_coef_parabolic} Now we give some examples regarding the conditions on the coefficients.
\begin{enumerate}[(i)]
 \item In the Hilbert space case, both $R$-boundedness conditions can be dropped as $R$-boundedness and uniform boundedness are equivalent.
 \item The condition on the $R$-boundedness in case \eqref{Boutet:it:parabolic:H} of assumption $(\mathrm{SAP})$ with $s\neq0$ is satisfied if $\sigma\in(0,1)$ is large enough and if $\dom$ is bounded and smooth. We refer the reader to \cite[Theorem~8.5.21]{Hytonen&Neerven&Veraar&Weis2016_Analyis_in_Banach_Spaces_II}. Therein, $R$-boundedness of the range of functions with fractional smoothness depending on the type and cotype of the involved Banach spaces and the geometry of the underlying domain is derived.
 \item Condition $(\mathrm{SAL})$ in case \eqref{Boutet:it:parabolic:H} with $s\neq0$, case \eqref{Boutet:it:parabolic:F} and case \eqref{Boutet:it:parabolic:B} is for example satisfied for $a_{\alpha}\in L_{\infty}(J;B^{\kappa_{\alpha}}_{\infty,1}(\mathscr{O},\mathcal{B}(X)))$ with $\kappa_{\alpha}>\sigma_{s+2m-|\alpha|,s,p,\gamma}$ $(|\alpha|<2m)$, see Proposition~\ref{Prop:PointwiseMultiplication}. In case \eqref{Boutet:it:parabolic:H} with $s=0$ it is satisfied for $a_{\alpha}\in L_{\infty}(\mathscr{O}\times J;\mathcal{B}(X))$.
 \item Condition (SBL) is for example satisfied for
 \[
b_{j,\beta}\in\begin{cases}
 F^{\kappa_{j,\E}}_{s_{j,\beta},p}(J;L_{r_{j,\beta}}(\partial\mathscr{O};\mathcal{B}(X))) \cap L_{s_{j,\beta}}(J;B^{2m\kappa_{j,\E}}_{r_{j,\beta},p}(\partial\mathscr{O};\mathcal{B}(X)))\quad&\text{in cases \eqref{Boutet:it:parabolic:H} and \eqref{Boutet:it:parabolic:F}},\\
   F^{\kappa_{j,\E}}_{s_{j,\beta},q}(J;L_{r_{j,\beta}}(\partial\mathscr{O};\mathcal{B}(X))) \cap L_{s_{j,\beta}}(J;B^{2m\kappa_{j,\E}}_{r_{j,\beta},q}(\partial\mathscr{O};\mathcal{B}(X)))\quad&\text{in case \eqref{Boutet:it:parabolic:B}}.
\end{cases}
\]
where
\[
\kappa_{j,\E} > \frac{1}{s_{j,\beta}} + \frac{n-1}{2mr_{j,\beta}} + \frac{|\beta|-m_{j}}{2m}
\quad \mbox{and} \quad \mu > \frac{q}{s_{j,\beta}}-1
\]
for some $s_{j,\beta} \in [q,\infty)$ and $r_{j,\beta} \in [p,\infty)$.
 \end{enumerate}
\end{example}

In the $L_{q,\mu}$-$\E$-maximal regularity approach in Theorem~\ref{Boutet:theorem:DPIBVP} we look for solutions
\[
u \in W^{1}_{q}(J,v_{\mu};\E) \cap L_{q}(J,v_{\mu};\E^{2m})
\]
of the problem
\begin{equation}\label{Boutet:eq:theorem:DPIBVP}
\left\{\begin{array}{rlll}
\partial_{t}u + \mathcal{A}(D)u &= f, & \text{on $\dom \times J$},  \\
\mathcal{B}_{j}(D)u &= g_{j}, & \text{on $\partial\dom \times J$}, & j=1,\ldots,m, \\
u(0) &= u_{0}, & \text{on $\dom$}.
\end{array}\right.
\end{equation}
and characterize the data $f$, $g=(g_{1},\ldots,g_{m})$ and $u_{0}$ for which this actually can be solved.

Let us now introduce some notation for the function spaces appearing in this problem. For an open interval $I \subset \R$ and $v\in A_q(\R)$, we put
\begin{equation*}
\begin{split}
\D_{q,v}(I;\E) &:= L_{q}(I,v;\E), \\
\M_{q,v}(I;\E) &:= W^{1}_{q}(I,v;\E) \cap L_{q}(I,v;\E^{2m}),\\
\B_{q,v,j}(I;\E) &:=
\begin{cases}
 F^{\kappa_{j,\E}}_{q,p}(I,v;L_{p}(\partial\mathscr{O};X)) \cap L_{q}(I,v;F^{2m\kappa_{j,\E}}_{p,p}(\partial\mathscr{O};X)) \quad\text{in cases \eqref{Boutet:it:parabolic:H} and \eqref{Boutet:it:parabolic:F}},\\
 B^{\kappa_{j,\E}}_{q,q}(I,v;L_{p}(\partial\mathscr{O};X)) \cap L_{q}(I,v;B^{2m\kappa_{j,\E}}_{p,q}(\partial\mathscr{O};X))\quad\text{in case \eqref{Boutet:it:parabolic:B}},
 \end{cases}\\
 & \qquad \text{for}\:\: j=1,\ldots,m,\\
\B_{q,v}(I;\E) &:= \bigoplus_{j=1}^{m}\B_{q,v,j}(I;\E).\\
\end{split}
\end{equation*}
For the power weight $v=v_{\mu}$, with $\mu \in (-1,q-1)$, we simply replace $v$ by $\mu$ in the subscripts: $\D_{q,\mu}(I;\E) := \D_{q,v_{\mu}}(I;\E)$, $\M_{q,\mu}(I;\E) := \M_{q,v_{\mu}}(I;\E)$, $\B_{q,\mu,j}(I;\E) = \B_{q,v_{\mu},j}(I;\E)$ and $\B_{q,\mu}(I;\E) = \B_{q,v_{\mu}}(I;\E)$.
In this case we furthermore define
\[
\I_{q,\mu}(I;\E) := B^{s+2m(1-\frac{1+\mu}{q})}_{p,q,\gamma}(\mathscr{O};X).
\]

In Theorem~\ref{Boutet:theorem:DPIBVP} we will in particularly see that
\[
\M_{q,\mu}(J;\E) \longra \B_{q,\mu}(J;\E) \oplus \I_{q,\mu}(J;\E),\,u \mapsto (\mathcal{B}(D)u,u_{0}),
\]
which basically just is a trace theory part of the problem.
In view of the commutativity of taking traces, $\mathrm{tr}_{\partial\dom} \circ \mathrm{tr}_{t=0} = \mathrm{tr}_{t=0} =\mathrm{tr}_{\partial\dom}$, when well-defined, we also have to impose a compatibility condition on $g$ and $u_{0}$ in \eqref{Boutet:eq:theorem:DPIBVP}.
In order to formulate this precisely, let us define
\[
\mathcal{B}^{t=0}_{j}(D) := \sum_{|\beta| \leq m_{j}}b_{j,\beta}(0,\,\cdot\,)\mathrm{tr}_{\partial\mathscr{O}}D^{\beta}, \qquad j=1,\ldots,m,
\]
and
\[
\mathbb{IB}_{q,\mu}(I;\E) := \left\{ (g,u_{0}) \in \B_{q,\mu}(I;\E) \oplus \I_{q,\mu}(I;\E) : \mathrm{tr}_{t=0}g_{j} - \mathcal{B}^{t=0}_{j}(D)u_{0} = 0 \:\:\mbox{when}\:\: \kappa_{j,\E} > \frac{1+\mu}{q} \right\}
\]
where $I\in\{J,\R_+\}$.
\begin{remark}\label{Boutet:rmk:compatibility_cond}
Let $I\in\{J,\R_+\}$. Regarding the compatibility condition
\[
\mathrm{tr}_{t=0}g_{j} - \mathcal{B}^{t=0}_{j}(D)u_{0} = 0 \:\:\mbox{when}\:\: \kappa_{j,\E} > \frac{1+\mu}{q}
\]
in the definition of $\mathbb{IB}_{q,\mu}(I;\E)$, let us remark the following.
For simplicity of notation we restrict ourselves to cases \eqref{Boutet:it:parabolic:H} and \eqref{Boutet:it:parabolic:F}, case \eqref{Boutet:it:parabolic:B} being analogous.
Suppose $\kappa_{j,\E} > \frac{1+\mu}{q}$. Then $(g_{j},u_{0}) \mapsto \mathrm{tr}_{t=0}g_{j} - \mathcal{B}^{t=0}_{j}(D)u_{0}$ is a well-defined bounded linear operator $\B_{q,\mu}(I;\E) \oplus \I_{q,\mu}(I;\E) \to L_{p}(\partial\mathscr{O};X)$ as
\[
\B_{q,\mu,j}(I;\E) \hookrightarrow F^{\kappa_{j,\E}}_{q,p}(I,v_{\mu};L_{p}(\partial\mathscr{O};X))
\]
and
\[
D^{\beta}: \I_{q,\mu}(I;\E) \longra B^{s+2m(1-\frac{1+\mu}{q})-m_{j}}_{p,q,\gamma}(\mathscr{O};X), \qquad |\beta| \leq m_{j},
\]
with
\[
s+2m(1-\frac{1+\mu}{q})-m_{j} = 2m\left( \kappa_{j,\E}-\frac{1+\mu}{q}+\frac{1+\gamma}{2mp} \right) > \frac{1+\gamma}{p}.
\]
\end{remark}

\subsection{Statement of the Main Result}\label{Boutet:subsec:sec:parabolic;main_result}

\begin{theorem}\label{Boutet:theorem:DPIBVP}
Let the notations be as in Subsection~\ref{Boutet:subsec:sec:parabolic;notation} with $v=v_{\mu}$, $\mu \in (-1,q-1)$. Suppose that $X$ is a UMD space, that $\mathscr{O}$ satisfies the smoothness condition \textbf{$(\mathrm{SO})$}, that $\mathcal{A}(D),\mathcal{B}_{1}(D),\ldots,\mathcal{B}_{m}(D)$ satisfies the smoothness conditions $(\mathrm{SAP})$, $(\mathrm{SAL})$, $(\mathrm{SBP})$ and $(\mathrm{SBL})$ as well as the conditions $(\mathrm{E})_{\phi}$, $(\mathrm{LS})_{\phi}$ for some $\phi \in (0,\frac{\pi}{2})$, and that
$\kappa_{j,\E} \neq \frac{1+\mu}{q}$ for all $j \in \{1,\ldots,m\}$.
Then the problem \eqref{Boutet:intro:pibvp} enjoys the property of maximal $L_{q,\mu}$-$\E$-regularity with $\mathbb{IB}_{q,\mu}(J;\E)$ as the optimal space of initial-boundary data, i.e.\
\[
\M_{q,\mu}(J;\E) \longra \D_{q,\mu}(J;\E) \oplus \mathbb{IB}_{q,\mu}(J;\E),\, u \mapsto (\partial_{t}u+\mathcal{A}(D)u,\mathcal{B}(D)u,u_{0})
\]
defines an isomorphism of Banach spaces. In particular, the problem
\eqref{Boutet:eq:theorem:DPIBVP} admits a unique solution $u \in \M_{q,\mu}(J;\E)$ if and only if $(f,g,u_{0}) \in \D_{q,\mu}(J;\E) \oplus \mathbb{IB}_{q,\mu}(J;\E)$.
\end{theorem}

\begin{remark}\label{Boutet:rmk:theorem:DPIBVP;Lq-Lp}
In the $L_{q,\mu}$-$L_{p,\gamma}$-case the proof simplifies a bit on the function space theoretic side of the problem, yielding a simpler proof than the previous ones (\cite{DHP2} ($\mu=0$, $\gamma=0$), \cite{MeySchnau2} ($q=p$, $\mu \in [0,p-1)$, $\gamma=0$) and \cite{Lindemulder2017_PIBVP}).
\end{remark}

\begin{remark}\label{Boutet:rmk:theorem:DPIBVP;Besov}
In case~\eqref{Boutet:it:parabolic:B} of Theorem~\ref{Boutet:theorem:DPIBVP} we have only allowed the special case that the microscopic parameter of the Besov space coincides with the temporal integrability parameter. On a technical level this restriction comes from a real interpolation argument. It is unclear to us what the correct space of boundary data should be in the general Besov case.
\end{remark}

Analogously to \cite[Section~4.3]{Lindemulder2018_DSOP}, we obtain the following smoothing result as a corollary to Theorem~\ref{Boutet:theorem:DPIBVP}.
It basically says that, in the case of smooth coefficients, there is $C^{\infty}$-regularity in the spatial variable with some quantitative blow-up near the boundary for the solution $u$ when $f=0$ and $u_{0}=0$ (see the discussion after \cite[Corollary~1.3]{Lindemulder2018_DSOP}).

\begin{corollary}\label{Boutet:cor:theorem:DPIBVP;smoothing}
Let the notations and assumptions be as in Theorem~\ref{Boutet:theorem:DPIBVP}.
Assume that, in addition, we are in cases \eqref{Boutet:it:parabolic:H} or \eqref{Boutet:it:parabolic:F} and that $a_{\alpha} \in BC^{\infty}(\dom;\mathcal{B}(X))$ for each $|\alpha| \leq 2m$.
Then
\begin{align*}
&\{u \in \M_{q,\mu}(J;\E) : \partial_{t}u+\mathcal{A}(D)u = 0, u_{0}=0 \} \\
&\qquad \hookrightarrow
\quad \bigcap_{\nu > -1}\left[ W^{1}_{q,\mu}(J;F^{s+\frac{\nu-\gamma}{p}}_{p,1,\nu}(\dom;X))\cap L_{q,\mu}(J;F^{s+\frac{\nu-\gamma}{p}+2m}_{p,1,\nu}(\dom;X)) \right] \\
&\qquad\qquad \hookrightarrow \quad \bigcap_{k \in \N}\left[ W^{1}_{q,\mu}(J;W^{k}_{p}(\dom,w_{\gamma+(k-s)p}^{\partial\dom};X))\cap L_{\mu}(J;W^{k+2m}_{p}(\dom,w_{\gamma+(k-s)p}^{\partial\dom};X)) \right].
\end{align*}
\end{corollary}

\subsection{The Proof of Theorem~\ref{Boutet:theorem:DPIBVP}}\label{Boutet:subsec:sec:parabolic;proof}

For the proof of Theorem~\ref{Boutet:theorem:DPIBVP} we will first look at model problems on $\mathscr{O}=\R^{n}_{+}$, from which the general case can be derived by means of a localization procedure.

\begin{proposition}\label{Boutet:theorem:DPBVP_model}
Let $X$ be a UMD Banach space and assume that $(\mathcal{A},\mathcal{B}_{1},\ldots,\mathcal{B}_{m})$ is homogeneous with constant-coefficients on $\mathscr{O}=\R^{n}_{+}$ and satisfies $(\mathrm{E})_{\phi}$ and $(\mathrm{LS})_{\phi}$ for some $\phi \in (0,\frac{\pi}{2})$.
Let $q \in (1,\infty)$ and $v \in A_{q}(\R)$. Let $\E$ and $\E^{2m}$ be given as in either \eqref{Boutet:it:parabolic:H}, \eqref{Boutet:it:parabolic:F} or \eqref{Boutet:it:parabolic:B} (with $\mathscr{O}=\R^{n}_{+}$).
Assume that $\kappa_{j,\E} \neq \frac{1+\mu}{q}$ for all $j \in \{1,\ldots,n\}$.
Then $u \mapsto (\partial_{t}u + (1+\eta+\mathcal{A}(D))u,\mathcal{B}(D)u)$ defines an isomorphism of Banach spaces
\[
\M_{q,v}(\R;\E) \longra
\D_{q,v}(\R;\E) \oplus \B_{q,v}(\R;\E),
\]
where $\M_{q,v}(\R;\E)$, $\D_{q,v}(\R;\E)$, $\B_{q,v}(\R;\E)$ are as in Subsection~\ref{Boutet:subsec:sec:parabolic;notation}. Moreover, the norm of the solution operator can be estimated by a bound which is polynomial in $\eta$.
\end{proposition}

\begin{proposition}\label{Boutet:prop:theorem:DPBVP_model;initial_value}
Let $X$ be a UMD Banach space and assume that $(\mathcal{A},\mathcal{B}_{1},\ldots,\mathcal{B}_{m})$ is homogeneous with constant-coefficients on $\mathscr{O}=\R^{n}_{+}$ and satisfies  $(\mathrm{E})_{\phi}$ and $(\mathrm{LS})_{\phi}$ for some $\phi \in (0,\frac{\pi}{2})$. Let $I\in\{J,\R_+\}$.
Let $q \in (1,\infty)$ and $\mu \in (-1,q-1)$. Let $\E$ and $\E^{2m}$ be given as in either \eqref{Boutet:it:parabolic:H}, \eqref{Boutet:it:parabolic:F} or \eqref{Boutet:it:parabolic:B} (with $\mathscr{O}=\R^{n}_{+}$).
Then $u \mapsto (\partial_{t}u + (1+\mathcal{A}(D))u,\mathcal{B}(D)u,u(0))$ defines an isomorphism of Banach spaces
\[
\M_{q,\mu}(I;\E) \longra
\D_{q,\mu}(I;\E) \oplus \IB_{q,\mu}(I;\E),
\]
where $\M_{q,\mu}(I;\E)$, $\D_{q,\mu}(I;\E)$, $\IB_{q,\mu}(I;\E)$ are as in the beginning of this section.
\end{proposition}

\begin{lemma}\label{Boutet:lemma:DPBVP_bd-data}
Let $X$ be a UMD Banach space and assume that $(\mathcal{A},\mathcal{B}_{1},\ldots,\mathcal{B}_{m})$ is homogeneous with constant-coefficients on $\mathscr{O}=\R^{n}_{+}$ and satisfies $(\mathrm{E})_{\phi}$ and $(\mathrm{LS})_{\phi}$ for some $\phi \in (0,\frac{\pi}{2})$.
Let $q \in (1,\infty)$ and $v \in A_{q}(\R)$. Let $\E$ and $\E^{2m}$ be given as in either \eqref{Boutet:it:parabolic:H}, \eqref{Boutet:it:parabolic:F} or \eqref{Boutet:it:parabolic:B} (with $\mathscr{O}=\R^{n}_{+}$).
Then
\begin{equation}\label{Boutet:eq:DPBVP_bd-data;bdd_boundary_operator}
\mathcal{B}(D): \M_{q,v}(\R;\E) \longra \B_{q,v}(\R;\E),\,u \mapsto (\mathcal{B}_{1}(D)u,\ldots,\mathcal{B}_{m}u),
\end{equation}
is a well-defined bounded linear operator and the differential parabolic boundary value problem
\begin{equation}\label{Boutet:eq:DPBVP_bd-data}
\left\{\begin{array}{rll}
\partial_{t}u + (1+\eta+\mathcal{A}(D))u &= 0,  \\
\mathcal{B}_{j}(D)u &= g_{j}, & j=1,\ldots,m,
\end{array}\right.
\end{equation}
admits a bounded linear solution operator
\[
\mathscr{S}^{(\eta)}: \B_{q,v}(\R;\E) \longra \M_{q,v}(\R;\E),\,(g_{1},\ldots,g_{m}) \mapsto u,
\]
for all $\eta\geq0$ where $\M_{q,v}(\R;\E)$, $\B_{q,v}(\R;\E)$ are as in the beginning of this section. The norm of $\mathscr{S}^{(\eta)}$ can be estimated by a polynomial bound in $\eta$.
Moreover, there is uniqueness of solutions in $\M_{q,v}(\R;\E)$: if $u \in \M_{q,v}(\R;\E)$ and $g=(g_{1},\ldots,g_{m}) \in \B_{q,v}(\R;\E)$ satisfy \eqref{Boutet:eq:DPBVP_bd-data}, then $u=\mathscr{S}^{(\eta)}g$.
\end{lemma}
\begin{proof}
That \eqref{Boutet:eq:DPBVP_bd-data;bdd_boundary_operator} is a well-defined bounded linear operator follows from Corollaries \ref{Boutet:lemma:DPBVP_boundedness_higher_order_traces}, \eqref{Boutet:lemma:DPBVP_boundedness_higher_order_traces;H} and \eqref{Boutet:lemma:DPBVP_boundedness_higher_order_traces;B}.
So we just need to establish the existence of a bounded linear solution operator $\mathscr{S}^{(\eta)}:(\mathcal{S}(\R^{n};X)^{m},\norm{\,\cdot\,}_{\B}) \longra \M$.
But the existence of such a solution operator as well as the polynomial bounds in $\eta$ follow from a combination of Corollary~\ref{Boutet:prop:Poisson_operator_DBVP;parabolic} with Corollary \ref{Boutet:cor:thm:mapping_anisotropic_Poisson;intersection_space}, \ref{Boutet:cor:thm:mapping_anisotropic_Poisson;intersection_space;H} or \ref{Boutet:cor:thm:mapping_anisotropic_Poisson;intersection_space;B}.

Finally, let us prove the uniqueness of solutions. For this it suffices to show that $\mathcal{S}(\R^{n}_{+} \times \R;X)$ is dense in $\M$. Indeed, using this density, the uniqueness statement follows from a combination of \eqref{Boutet:eq:DPBVP_bd-data;bdd_boundary_operator}, the uniqueness statement in Corollary~\ref{Boutet:prop:Poisson_operator_DBVP;parabolic} and the continuity of our solution operator $\mathscr{S}^{(\eta)}:\B \longra \M$.

For this density, note that $W^{1}_{q}(\R,v;\E^{2m})$ is dense in $\M$ by a standard convolution argument (in the time variable). So
\[
\mathcal{S}(\R) \otimes \mathcal{S}(\R^{n}_{+};X) \stackrel{d}{\subset} \mathcal{S}(\R) \otimes \E^{2m} \stackrel{d}{\subset} \mathcal{S}(\R;\E^{2m}) \stackrel{d}{\hookrightarrow} W^{1}_{q}(\R,v;\E^{2m}) \stackrel{d}{\hookrightarrow} \M,
\]
yielding the required density.
\end{proof}

\begin{lemma}\label{Boutet:lemma:R-boundedness_for_the_symbol}
 Let $X$ be a Banach space. Suppose $\mathcal{A}(D)=\sum_{|\alpha| =2m}a_{\alpha}D^{\alpha}$ with $a_{\alpha} \in \mathcal{B}(X)$ is parameter-elliptic with angle of ellipticity $\phi_{\mathcal{A}}$ and let $\phi>\phi_{\mathcal{A}}$.
Then, for all $s=(s_1,s_2,s_3)\in\R^3$ we have that
 \begin{equation}\label{Boutet:eq:lemma:R-sect_full-space_F-scale;R-bdd}
\kappa_{\alpha} := \mathcal{R}\left\{ \langle \xi \rangle^{|\alpha|} D^{\alpha}_{\xi}(s_1+s_2\lambda+s_3|\xi|^{2m})(1+\lambda+\mathcal{A}(\xi))^{-1} : \lambda \in \Sigma_{\pi-\phi}, \xi \in \R^{n} \right\} < \infty \quad\text{in}\quad \mathcal{B}(X)
\end{equation}
for all $\alpha\in\N^n$.
\end{lemma}
\begin{proof}
In order to establish \eqref{Boutet:eq:lemma:R-sect_full-space_F-scale;R-bdd}, we define
\[
 f\colon \R\times\R^{n}\times\Sigma_{\pi-\phi'}\to \mathcal{B}(X),\; (x,\xi,\lambda)\mapsto (s_1x^{2m}+s_2\lambda+s_3|\xi|^{2m})(x^{2m}+\lambda+ A(\xi))^{-1},
\]
where $\phi_{\mathcal{A}}<\phi'<\phi$ as well as
\begin{align*}
 f_{\alpha}(x,\xi,\lambda):=(x^2+|\xi|^2)^{|\alpha|/2}\partial_{\xi}^\alpha f(x,\xi,\lambda),\quad g_{\alpha}(x,\xi,\lambda):=(x^2+|\xi|^2+\lambda^{1/m})^{|\alpha|/2}\partial_{\xi}^\alpha f(x,\xi,\lambda)
\end{align*}
for $\alpha\in\N^n$. By geometric considerations, we obtain that
\[
 x^2+|\xi|^2\leq \frac{\big|x^2+|\xi|^2+\lambda^{1/m}\big|}{\cos\big(\tfrac{\pi}{2}-\max\{\tfrac{\pi}{2},\phi'\}\big)}
\]
for all $(x,\xi,\lambda)\in\R\times\R^n\times\Sigma_{\pi-\phi'}$. Hence, Kahane's contraction principle yields
\begin{align*}
 \kappa_{\alpha}=\mathcal{R}&\left\{ f_{\alpha}(1,\xi,\lambda) : \lambda \in \Sigma_{\pi-\phi}, \xi \in \R^{n} \right\}\\
 \lesssim \mathcal{R}&\left\{ f_{\alpha}(1,\xi,\lambda) : \lambda \in \Sigma_{\pi-\phi}, \xi \in \R^{n}\text{ such that }|\lambda|\leq |\xi|^{2m}\text{ or } |\lambda|\leq 1 \right\}\\
 + \mathcal{R} &\left\{ g_{\alpha}(1,\xi,\lambda) : \lambda \in \Sigma_{\pi-\phi}, \xi \in \R^{n}\text{ such that }|\lambda|\geq |\xi|^{2m}\text{ and } |\lambda|\geq 1 \right\}.
\end{align*}
 Obviously, we have that $f(cx,c\xi,c^{2m}\lambda)=f(x,\xi,\lambda)$ for all $c>0$. Lemma \ref{Boutet:lemma:DerivativeHomogeneous} shows that the same holds for $f_{\alpha}$ and $g_{\alpha}$. Hence, by choosing $c=(1+|\xi|^2+|\lambda|^{1/m})^{1/2}$ and defining
 \begin{align*}
 D_1&:=\operatorname{cl}\left\{\bigg(\frac{1}{c},\frac{\xi}{c},\frac{\lambda}{c^{2m}}\bigg) :\lambda \in \Sigma_{\pi-\phi}, \xi \in \R^{n}\text{ such that }|\lambda|\leq |\xi|^{2m}\text{ or } |\lambda|\leq 1  \right\},\\
  D_2&:=\operatorname{cl}\left\{\bigg(\frac{1}{c},\frac{\xi}{c},\frac{\lambda}{c^{2m}}\bigg) :\lambda \in \Sigma_{\pi-\phi}, \xi \in \R^{n}\text{ such that }|\lambda|\geq |\xi|^{2m}\text{ and } |\lambda|\geq 1  \right\},
\end{align*}
 we obtain
\begin{align*}
 \kappa_{\alpha}\lesssim \mathcal{R}&(f(D_1))+\mathcal{R}(f(D_2)).
\end{align*}
But since $f_{\alpha}$ is holomorphic on
\[
 \R\times\R^n\times\Sigma_{\pi-\phi'}\setminus\{(0,0,c_3):c_3\in\Sigma_{\pi-\phi'}\}\supset D_1
\]
and since $g_{\alpha}$ is holomorphic on
\[
 \R\times\R^n\times\Sigma_{\pi-\phi'}\setminus\{(c_1,c_2,0):c_1\in\R, c_2\in\R^{n}\}\supset D_2
\]
we obtain that \cite[Proposition 3.10]{Denk&Hieber&Pruess2001_monograph} implies
\[
 \kappa_{\alpha}\lesssim \mathcal{R}\{ f_{\alpha}(D_1) \}+\mathcal{R}\{ g_{\alpha}(D_2) \}<\infty.
\]
by the compactness of $D_1$ and $D_2$.
\end{proof}

\begin{lemma}\label{Boutet:lemma:R-sect_full-space_F-scale}
Let $X$ be a Banach space, $p,r \in [1,\infty)$, $w \in A_{\infty}(\R^{n})$, $\mathscr{A}\in\{B,F\}$ and $s \in \R$.
Suppose $\mathcal{A}(D)=\sum_{|\alpha| =2m}a_{\alpha}D^{\alpha}$ with $a_{\alpha} \in \mathcal{B}(X)$ is parameter-elliptic with angle of ellipticity $\phi_{\mathcal{A}}$. Let $A$ be the realization of $\mathcal{A}(D)$ in $\mathscr{A}^{s}_{p,r}(\R^{n},w;X)$ with domain $D(A)=\mathscr{A}^{s+2m}_{p,r}(\R^{n},w;X)$.
Then $0\in\rho(1+A)$ and $1+A$ is $R$-sectorial with angle $\omega_{R}(1+A) \leq \phi_{\mathcal{A}}$.
\end{lemma}
\begin{proof}
By Lemma \ref{Boutet:lemma:R-boundedness_for_the_symbol} with $s=(0,1,0)$ and Lemma \ref{Boutet:lemma:R-bdd_FM_F-scale}, we have
\begin{equation}\label{Boutet:eq:lemma:R-sect_full-space_F-scale;R-bdd_resolvents}
\mathcal{R}\{\lambda(1+\lambda+A)^{-1} : \lambda \in \Sigma_{\pi-\phi}\} < \infty
\end{equation}
where $\phi>\phi_{\mathcal{A}}$. This shows the $\mathcal{R}$-sectoriality of $1+A$. Hence, it only remains to show that $D(A)=\mathscr{A}^{s+2m}_{p,r}(\R^{n},w;X)$. But choosing $(s_1,s_2,s_3)=(1,0,1)$ in \eqref{Boutet:eq:lemma:R-sect_full-space_F-scale;R-bdd} shows that
\begin{align*}
\kappa_{\alpha} := \mathcal{R}\left\{ \langle\xi\rangle^{|\alpha|} D^{\alpha}_{\xi}(1+|\xi|^{2m})(1+\lambda+\mathcal{A}(\xi))^{-1} : \lambda \in \Sigma_{\pi-\phi}, \xi \in \R^{n} \right\} < \infty \quad\text{in}\quad \mathcal{B}(X), \quad \alpha \in \N^{n}.
\end{align*}
Using \eqref{Boutet:eq:prelim:FM_B&F} together with \eqref{Boutet:eq:prelim:Bessel-pot_B&F}, we obtain that $(1+\lambda+A)^{-1}$ maps $\mathscr{A}^{s}_{p,r}(\R^{n},w;X)$ into $\mathscr{A}^{s+2m}_{p,r}(\R^{n},w;X)$. This shows that $D(A)=\mathscr{A}^{s+2m}_{p,r}(\R^{n},w;X)$.
\end{proof}

\begin{proof}[Proof of Proposition~\ref{Boutet:theorem:DPBVP_model}]
We first show that the differential parabolic boundary value problem
\begin{equation}\label{Boutet:eq:DPBVP_model}
\begin{array}{rll}
\partial_{t}u + (1+\eta+\mathcal{A}(D))u &= f,  \\
\mathcal{B}_{j}(D)u &= g_{j}, & j=1,\ldots,m,
\end{array}
\end{equation}
admits a bounded linear solution operator
\[
\mathscr{T}:L_{q}(\R,v;\E) \oplus \B \longra \M,\,(f,g_{1},\ldots,g_{m}) \mapsto u
\]
To this end, for $k \in \{0,2m\}$ let
\[
\bar{\E}^{k}:= \left\{\begin{array}{ll}
H^{s+k}_{p}(\R^{n},w_{\gamma};X),& \text{in case \eqref{Boutet:it:parabolic:H}}, \\
F^{s+k}_{p,r}(\R^{n},w_{\gamma};X),& \text{in case \eqref{Boutet:it:parabolic:F}},\\
B^{s+k}_{p,q}(\R^{n},w_{\gamma};X),& \text{in case \eqref{Boutet:it:parabolic:B}},
\end{array}\right.
\]
and put $\bar{\M}:=W^{1}_{q}(\R,v;\bar{\E}) \cap L_{q}(\R,v;\bar{\E}^{2m})$.
The realization of $1+\mathcal{A}(D)$ in $\bar{\E}$ with domain $\bar{\E}^{2m}$ has $0$ in its resolvent and is $R$-sectorial with angle $<\frac{\pi}{2}$, which in the cases \eqref{Boutet:it:parabolic:F} and \eqref{Boutet:it:parabolic:B} is contained in Lemma~\ref{Boutet:lemma:R-sect_full-space_F-scale} and which in case \eqref{Boutet:it:parabolic:H} can be derived as in \cite[Corollary~5.6]{Denk&Hieber&Pruess2001_monograph} using the operator-valued Mikhlin theorem for $H^{s}_{p}(\R^{n},w_{\gamma};X)$ (see Proposition~\ref{Boutet:prop:prelim:operator-Mihklin}).
As a consequence (see Section~\ref{Boutet:subsec:prelim:UMD}), the parabolic problem
\[
\partial_{t}\bar{u} + (1+\eta+\mathcal{A}(D))\bar{u} = \bar{f} \qquad \text{on} \quad \R^{n} \times \R
\]
admits a bounded linear solution operator
\[
\mathscr{R}: L_{q}(\R,v;\bar{\E}) \longra \bar{\M},\,\bar{f} \mapsto \bar{u}.
\]
Choosing an extension operator
\[
\mathcal{E} \in \mathcal{B}\left(L_{q}(\R,v;\E),L_{q}(\R,v;\bar{\E})\right),
\]
recalling \eqref{Boutet:eq:DPBVP_bd-data;bdd_boundary_operator}, denoting by $r_{+} \in \mathcal{B}(\bar{\M},\M)$ the operator of restriction from $\R^{n} \times \R$ to $\R^{n}_{+} \times \R$ and denoting by $\mathscr{S}$ the solution operator from Lemma~\ref{Boutet:lemma:DPBVP_bd-data}, we find that
\[
\mathscr{T}(f,g_{1},\ldots,g_{m}) := r_{+}\mathscr{R}\mathcal{E}f - \mathscr{S}\mathcal{B}(D)r_{+}\mathscr{R}\mathcal{E}f  + \mathscr{S}(g_{1},\ldots,g_{m})
\]
defines a solution operator as desired.

Finally, the uniqueness follows from the uniqueness obtained in Lemma~\ref{Boutet:lemma:DPBVP_bd-data}.
\end{proof}

\begin{lemma}\label{DBVP:lemma:resolvent_A_subject_to_B}
Let $X$ be a UMD Banach space and assume that $(\mathcal{A},\mathcal{B}_{1},\ldots,\mathcal{B}_{m})$ is homogeneous with constant-coefficients on $\mathscr{O}=\R^{n}_{+}$ and satisfies $(\mathrm{E})_{\phi}$ and $(\mathrm{LS})_{\phi}$ for some $\phi \in (0,\frac{\pi}{2})$.
Let $q \in (1,\infty)$ and $v \in A_{q}(\R)$.
Let $\E$ and $\E^{2m}$ be given as in either \eqref{Boutet:it:parabolic:H}, \eqref{Boutet:it:parabolic:F} or \eqref{Boutet:it:parabolic:B} (with $\mathscr{O}=\R^{n}_{+}$).
Let $A_{B}$ be the realization of $\mathcal{A}(D)$ in $\E$ with domain $D(A_{B}) = \{u \in \E^{2m} : \mathcal{B}(D)u=0 \}$. Then there is an equivalence of norms in $D(A_{B}) = \{u \in \E^{2m} : \mathcal{B}(D)u=0 \}$, $-(1+A_{B})$ is the generator of an exponentially stable analytic semigroup on $\E$ and $1+A_{B}$ enjoys the property of $L_{q}(\R_{+},v_{\mu})$-maximal regularity.
\end{lemma}
\begin{proof}
As a consequence of Proposition~\ref{Boutet:theorem:DPBVP_model}, $1+\eta+A_{B}$ satisfies the conditions of Lemma~\ref{DBVP:lemma:prelim:max-reg;imaginary_axis} with $\normm{\,\cdot\,} = \norm{\,\cdot\,}_{\E^{2m}}$. Therefore, there is an equivalence of norms in $D(1+\eta+A_{B}) = D(A_{B}) = \{u \in \E^{2m} : \mathcal{B}(D)u=0 \}$ and $1+\eta+A_{B}$ is a closed linear operator on $\E$ enjoying the property of $L_{q}(\R,v)$-maximal regularity. Moreover, it follows from the polynomial bounds in Proposition \ref{Boutet:theorem:DPBVP_model} and Lemma~\ref{DBVP:lemma:prelim:max-reg;imaginary_axis} that $\C_{+} \subset \varrho(1+A_B)$ and that $\eta\mapsto (\eta+1+A_B)^{-1}$ is polynomially bounded.
Thus, $1+A_{B}$ satisfies the conditions of Lemma~\ref{DBVP:lemma:prelim:max-reg} and the desired result follows.
\end{proof}

\begin{lemma}\label{DBVP:lemma:time_trace}
Let the notations be as in Subsection~\ref{Boutet:subsec:sec:parabolic;notation} with $v=v_{\mu}$, $\mu \in (-1,q-1)$, and suppose that $X$ is a UMD space.
Then $\tr_{t=0}:u \mapsto u(0)$ is a retraction
\[
\tr_{t=0}: \M_{q,\mu}(J;\E) \longra \I_{q,\mu}(J;\E).
\]
\end{lemma}
\begin{proof}
This can be derived from \cite[Theorem~1.1]{Meyries&Veraar2014_traces}/\cite[Theorem~3.4.8]{Pruess&Simonett2016_book}, see \cite[Section~6.1]{Lindemulder2017_PIBVP} and \cite[Lemma~4.8]{Lindemulder2018_DSOP}.
\end{proof}
%
%

\begin{proof}[Proof of Proposition~\ref{Boutet:prop:theorem:DPBVP_model;initial_value}]
That $u \mapsto (u'+(1+\mathcal{A}(D))u,\mathcal{B}(D)u,u(0))$ is a bounded operator
\[
\M_{q,\mu}(\R_{+};\E) \longra \D_{q,\mu}(\R_{+};\E) \oplus \B_{q,\mu}(\R_{+};\E) \oplus \I_{q,\mu}(\E)
\]
follows from a combination of Proposition~\ref{Boutet:theorem:DPBVP_model} (choosing an extension operator $\M_{q,\mu}(\R_{+};\E) \to \M_{q,\mu}(\R;\E)$) and Lemma~\ref{DBVP:lemma:time_trace}.
That it maps to $\D_{q,\mu}(\R_{+};\E) \oplus \IB_{q,\mu}(\R_{+};\E)$ can be seen as follows:
we only need to show that
\begin{align}\label{eq:cor:max-reg_bdd_domain_bd-data;interval;comp_cond}
\tr_{t=0}\mathcal{B}_{j}(D)u = \mathcal{B}_{j}(D)\tr_{t=0}u, \qquad u \in \M_{q,\mu}(\R_{+};\E),
\end{align}
when $\kappa_{j,\E} > \frac{1+\mu}{q}$ (also see Remark~\ref{Boutet:rmk:compatibility_cond}), which simply follows from
\[
W^{1}_{q,\mu}(\R_{+};\E^{2m}) \stackrel{d}{\hookrightarrow} \M_{q,\mu}(\R_{+};\E).
\]
Here this density follows from a standard convolution argument (in the time variable) in combination with an extension/restriction argument.

Let $A_{B}$ be as in Lemma~\ref{DBVP:lemma:resolvent_A_subject_to_B}.
Then there is an equivalence of norms in $D(A_{B}) = \{u \in \E^{2m} : \mathcal{B}(D)u=0 \}$, $-(1+A_{B})$ is the generator of an exponentially stable analytic semigroup on $\E$ and $1+A_{B}$ enjoys the property of $L_{q}(\R_{+},v_{\mu})$-maximal regularity.
Now the desired result can be derived from Proposition~\ref{Boutet:theorem:DPBVP_model} as in \cite[Theorem~7.16]{LV2018_Dir_Laplace}.
\end{proof}

\begin{proof}[Proof of Theorem~\ref{Boutet:theorem:DPIBVP}]
This can be derived from the model problem case considered in Proposition~\ref{Boutet:prop:theorem:DPBVP_model;initial_value} by a standard localization procedure, see Appendix \ref{Appendix_Localization}.
\end{proof}


\section{Elliptic Problems}\label{Boutet:sec:elliptic}

\subsection{Smoothness Assumptions on the Coefficients} \label{Boutet:subsec:Elliptic_Smoothness_Coefficients}

Let $\mathscr{O}$ be either $\R^{n}_{+}$ or a $C^{N}$-domain in $\R^{n}$ with a compact boundary $\partial\mathscr{O}$, where $N \in \N$ is specified below. 
Let further $X$ be a reflexive Banach space and let $\mathcal{A}(D),\mathcal{B}_{1}(D),\ldots,\mathcal{B}_{m}(D)$ be a $\mathcal{B}(X)$-valued differential boundary value system on $\mathscr{O}$ as considered in Section~\ref{Boutet:subsec:sec:prelim:E&LS}, where the coefficients satisfy certain smoothness conditions which we are going to introduce later. Put $m^{*} := \max\{m_{1},\ldots,m_{m}\}$ and $m_{*} := \min\{m_{1},\ldots,m_{m}\}$. Let $p,q\in(1,\infty)$. For $s\in\R$ let $\F^s$ and $\partial \F^s$ be given as either
\begin{enumerate}[(A)]
 \item \label{Boutet:it:elliptic:H} $\F^s:=H^{s}_{p,\gamma}(\mathscr{O},X)$ and $\partial\F^s:=\partial H^{s}_{p,\gamma}(\partial\mathscr{O},X)$ where $\gamma\in(-1,p-1)$ and $X$ is a UMD space.
 \item \label{Boutet:it:elliptic:BF} $\F^s:=\mathscr{A}^{s}_{p,q,\gamma}(\mathscr{O},X)$ and $\partial\F^s:=\partial \mathscr{A}^{s}_{p,q,\gamma}(\partial\mathscr{O},X)$ where $(\gamma,\mathscr{A}) \in(-1,\infty) \times \{B,F\} \cup (-\infty,p-1) \times \{\mathcal{B},\mathcal{F}\}$.
\end{enumerate}
In the following, we use the notation
\[
 \vartheta_{j,s}:=s+2m-m_{j}-\frac{1+\gamma}{p} \qquad j=1,\ldots,m.
\]
Consider the following assumptions on $\mathscr{O}$ and $\mathcal{A}(D),\mathcal{B}_{1}(D),\ldots,\mathcal{B}_{m}(D)$:
\begin{enumerate}
\item[\textbf{$(\mathrm{SO})_s$}] 
\begin{description}
\item[Case \eqref{Boutet:it:elliptic:H}] $\mathscr{O}$ is $C^N$ with $N \geq \max\{s+2m,-s\}$.
\item[Case \eqref{Boutet:it:elliptic:BF} with $(\gamma,\mathscr{A}) \in(-1,\infty) \times \{B,F\}$] $\mathscr{O}$ is $C^N$ with $N > \max\big\{s+2m,\left(\frac{1+\gamma}{p}-1\right)_++1-s\big\}$.
\item[Case \eqref{Boutet:it:elliptic:BF} with $(\gamma,\mathscr{A}) \in(-\infty,p-1) \times \{\mathcal{B},\mathcal{F}\}$] $\mathscr{O}$ is $C^N$ with $N > \max\big\{-s,-\left(\frac{1+\gamma}{p}\right)_-+1+s+2m\big\}$.
\end{description}
\item[\textbf{$(\mathrm{SAP})_s$}]
\begin{description}
\item[Case \eqref{Boutet:it:elliptic:H} with $s=0$] For $|\alpha| = 2m$ we have $a_{\alpha} \in BUC(\mathscr{O};\mathcal{B}(X))$. If $\mathscr{O}$ is unbounded, the limits $a_{\alpha}(\infty) := \lim_{|x| \to \infty}a_{\alpha}(x)$ $(|\alpha|=2m)$ exist.
\item[Case \eqref{Boutet:it:elliptic:H} with $s\neq0$] Let $\sigma>|s|$. For $|\alpha| = 2m$ we have $a_{\alpha} \in BUC^{\sigma}(\mathscr{O};\mathcal{B}(X))$ such that for all $\epsilon>0$ there is a $\delta>0$ such that for all $x_0\in\mathscr{O}$ it holds that
\[
 \mathcal{R}(\{a_{\alpha}(x)-a_{\alpha}(x_0): x\in B(x_0,\delta)\cap\mathscr{O}\})<\epsilon.
\]
If $\mathscr{O}$ is unbounded, then for all $|\alpha|=2m$ the limits $a_{\alpha}(\infty) := \lim_{|x| \to \infty}a_{\alpha}(x)$ exist and for all $\epsilon>0$ there is an $R>0$ such that
\[
 \mathcal{R}(\{a_{\alpha}(x)-a_{\alpha}(\infty): x\in\mathscr{O},\,|x|\geq R\})<\epsilon.
\]
\item[Case \eqref{Boutet:it:elliptic:BF} ] For $|\alpha| = 2m$ we have $a_{\alpha} \in BUC^{\sigma}(\mathscr{O};\mathcal{B}(X))$ where $\sigma>\sigma_{s,p,\gamma}$. If $\mathscr{O}$ is unbounded, the limits $a_{\alpha}(\infty) := \lim_{|x| \to \infty}a_{\alpha}(x)$ $(|\alpha|=2m)$ exist.
\end{description}
\item[\textbf{$(\mathrm{SAL})_s$}]  There exists $\sigma \in [2m-1,2m)$ such that
$a_{\alpha}D^{\alpha}$ is bounded from $\F^{\sigma+s}$ to $\F^{s}$.
\item[\textbf{$(\mathrm{SBP})_s$}]
For each $j \in \{1,\ldots,m\}$ and $|\beta| = m_{j}$ there exists $r_{j,\beta} \in [p,\infty)$ with $\vartheta_{j,s}> \frac{n-1}{r_{j,\beta}}$ such that $b_{j,\beta} \in B^{\vartheta_{j,s}}_{r_{j,\beta},p}(\partial\mathscr{O};\mathcal{B}(X))$.
If $\mathscr{O}=\R^{n}_{+}$, the limits $b_{j,\beta}(\infty) := \lim_{|x'| \to \infty}b_{j,\beta}(x')$ exists for all $j \in \{1,\ldots,m\}$ and $|\beta|=m_{j}$.
\item[\textbf{$(\mathrm{SBL})_s$}]
There is a $\sigma\in[2m-1,2m)$ such that $b_{j,\beta}\operatorname{tr}_{\partial\mathscr{O}}D^{\beta}$ is bounded from $\F^{s+\sigma}$
to $\partial\F^{s+2m-m_j}$.
\end{enumerate}

\begin{example}\label{Boutet:ex:smoothness_coef_elliptic}
\begin{enumerate}[(i)]
 \item Regarding the $R$-boundedness condition in case \eqref{Boutet:it:elliptic:H} of $(\mathrm{SAP})_s$ with $s\neq0$, the same as in Example \ref{Boutet:ex:smoothness_coef_parabolic} applies.
 \item Condition $(\mathrm{SAL})_s$ in case \eqref{Boutet:it:elliptic:H} with $s\neq0$ and case \eqref{Boutet:it:elliptic:BF} is for example satisfied for $a_{\alpha}\in B^{\kappa_{\alpha}}_{\infty,1}(\mathscr{O},\mathcal{B}(X))$ with $\kappa_{\alpha}>\sigma_{s+2m-|\alpha|,s,p,\gamma}$ $(|\alpha|<2m)$, see Proposition~\ref{Prop:PointwiseMultiplication}. In case \eqref{Boutet:it:elliptic:H} with $s=0$ it is satisfied for $a_{\alpha}\in L_{\infty}(\mathscr{O};\mathcal{B}(X))$.
 \item Condition $(\mathrm{SBL})_s$ in case \eqref{Boutet:it:elliptic:H} with $s\neq0$ and case \eqref{Boutet:it:elliptic:BF} is for example satisfied for $b_{j,\beta}\in B^{\kappa_{\beta,j}}_{\infty,1}(\partial\mathscr{O},\mathcal{B}(X))$ with $\kappa_{\beta,j}>\sigma_{s_0,s_1,p,0}$ $(|\beta|<m_j)$ where $s_0=s+2m-|\beta|-\frac{1+\gamma}{p}$ and $s_1=s+2m-m_j-\frac{1+\gamma}{p}$, see Proposition~\ref{Prop:PointwiseMultiplication}. Note that with this choice of parameters we have $\sigma_{s_0,s_1,p,0}\geq s_1> 0$.
\end{enumerate}

\end{example}

\subsection{Parameter-dependent Estimates}

\begin{theorem}\label{Boutet:thm:EBVP_par-dep}
 Let the notations be as in Subsection \ref{Boutet:subsec:Elliptic_Smoothness_Coefficients}. Let further $s_0\in(\frac{1+\gamma}{p}-1,\frac{1+\gamma}{p}+m_{*})$ and $s_{1}\in[s_0,\infty)$ and assume that the smoothness conditions $(\mathrm{SO})_t$, $(\mathrm{SAP})_t$, $(\mathrm{SAL})_t$, $(\mathrm{SBP})_t$ and $(\mathrm{SBL})_t$ are satisfied for all $t\in[s_0,s_1]$. Suppose that also $\mathrm{(E)}_\phi$ and $\mathrm{(LS)}_\phi$ are satisfied for some $\phi\in(0,\pi)$. Then, there is a $\lambda_\phi>0$ such that for all $\lambda\in\Sigma_{\pi-\phi}$, $t \in [s_{0},s_{1}]$ and all
\[
 (f,g_1,\ldots,g_m)\in\F^t\oplus\bigoplus_{j=1}^m\partial\F^{t+2m-m_j}
\]
there exists a unique solution $u\in\F^{t+2m,\mu}$ of the problem
\begin{equation}\label{DBVP:thm:inhom_bd_data_higher_reg_model;BVP}
\left\{\begin{array}{rll}
(\lambda+\lambda_{\phi}+\mathcal{A}(\,\cdot\,,D))u &=f &  \mbox{on}\:\: \dom, \\
\mathcal{B}_{j}(D)u&=g_{j} & \mbox{on}\:\: \partial\dom, \qquad j=1,\ldots,m.
\end{array}\right.
\end{equation}
Moreover, for this solution there are the parameter-dependent estimates (independent of $t$)
\begin{align*}
\norm{u}_{\F^{t+2m}} + |\lambda|^{\frac{t+2m-s_{0}}{2m}}\norm{u}_{\F^{s_{0}}}
&\eqsim \:
\norm{f}_{\F^{t}} + |\lambda|^{\frac{t-s_{0}}{2m}}\norm{f}_{\F^{s_{0}}} \\
&\quad + \sum_{j=1}^{m}\left(\norm{g_{j}}_{\partial\F^{t+2m-m_{j}}} + |\lambda|^{\frac{t+2m-m_{j}-\frac{1+\gamma}{p}}{2m}}\norm{g_{j}}_{L_{p}(\partial\mathscr{O};X)} \right).
\end{align*}
\end{theorem}

\begin{remark}\label{Boutet:rmk:par-dep_est}
Parameter-dependent estimates as in Theorem~\ref{Boutet:thm:EBVP_par-dep} have been obtained in an unweighted $L_p$-setting with $s_0=0$ in \cite{Denk&Seger2016,Grubb&Kokholm1993}, also see the references given therein. To the best of our knowledge the case $s_0\neq0$ has not been treated before, except for \cite{Lindemulder2019_DSOE} where estimates as in Theorem~\ref{Boutet:thm:EBVP_par-dep} have been obtained for second order elliptic boundary value problems subject to Dirichlet boundary value problems. The main additional tool compared to \cite{Denk&Seger2016,Grubb&Kokholm1993} is that we allow parameter-dependent function spaces with base spaces that are different from just $L_p$. By real interpolation we could even derive parameter-dependent estimates for Besov spaces with Triebel-Lizorkin spaces or Bessel potential spaces as base spaces. This also includes the cases treated in \cite{Denk&Seger2016,Grubb&Kokholm1993} with $L_p$ as a base space. 
\end{remark}

In the proof of the above theorem we will use the following lemmas.

\begin{lemma}\label{Boutet:lemma:parameter-elliptic-whole-space}
  Let $X$ be a reflexive Banach space, $p,q\in(1,\infty)$, $(w,\mathscr{A}) \in A_{\infty}(\R^{n}) \times \{B,F\} \cup [A_{\infty}]'_{p}(\R^{n}) \times \{\mathcal{B},\mathcal{F}\}$ and $s_0,t\in\R$ with $t\geq s_0$.
 Let $\mathcal{A}(D)$ be a differential operator of order $2m$ with constant $\mathcal{B}(X)$-valued coefficients satisfying $(E)_\phi$ for some $\phi\in(0,\pi]$. Given $f\in \mathscr{A}^{s_0}_{p,q}(\R^n,w;X)$ let $u:=(\lambda+\mathcal{A}(D))^{-1}f$. Then, for all $\lambda_0>0$ we have the estimate
 \[
  \|u \|_{\mathscr{A}^{t+2m}_{p,q}(\R^n,w;X)}+|\lambda|^{\frac{2m+t-s_0}{2m}}\|u \|_{\mathscr{A}^{s_0}_{p,q}(\R^n,w;X)}\eqsim_{\lambda_0} \|f \|_{\mathscr{A}^{t}_{p,q}(\R^n,w;X)}+ |\lambda|^{\frac{t-s_0}{2m}}\|f \|_{\mathscr{A}^{s_0}_{p,q}(\R^n,w;X)}\quad(\lambda\in(\lambda_0+\Sigma_\phi)).
 \]
\end{lemma}
\begin{proof}
 We substitute $\lambda=\mu^{2m}$ so that $(\xi,\mu)\mapsto (\mu^{2m}+\mathcal{A}(\xi))^{-1}$ is a parameter-dependent H\"ormander symbol of order $-2m$ and regularity $\infty$. Hence, if we define
 \[
  p_{2m}(\xi,\mu):=\langle\xi,\mu\rangle^{2m}(\mu^{2m}+\mathcal{A}(\xi))^{-1}= \langle\xi,\mu\rangle^{t+2m-s_0}(\mu^{2m}+\mathcal{A}(\xi))^{-1}\langle\xi,\mu\rangle^{s_0-t},
 \]
then $(p(\,\cdot\,,\mu))_{\mu\in\Sigma_{\phi}/2m}$ and $(p(\,\cdot\,,\mu)^{-1})_{\mu\in\Sigma_{\phi}/2m}$ are bounded families in the parameter-independent H\"ormander symbols $S^0(\R^n,\mathcal{B}(X))$ of order $0$. In particular, by \eqref{Boutet:eq:prelim:FM_B&F} together with a duality argument for the dual scales, we have that
\begin{align*}
 \|u \|_{\mathscr{A}^{t+2m,\mu,s_0}_{p,q}(\R^n,w,X)}
 &=\|(\mu^{2m}+\mathcal{A}(D))^{-1}f\|_{\mathscr{A}^{t+2m,\mu,s_0}_{p,q}(\R^n,w;X)}\\
 &=\|\Xi_\mu^{s_0-t-2m}\operatorname{OP}(p_{2m}(\,\cdot\,,\mu)\Xi_\mu^{t-s_0}f\|_{\mathscr{A}^{t+2m,\mu,s_0}_{p,q}(\R^n,w;X)}\\
 &= \|\operatorname{OP}(p_{2m}(\,\cdot\,,\mu)\Xi_\mu^{t-s_0}f\|_{\mathscr{A}^{s_0}_{p,q}(\R^n,w;X)}\\
 &\eqsim \|\Xi_\mu^{t-s_0}f\|_{\mathscr{A}^{s_0}_{p,q}(\R^n,w;X)}\\
 &=\|f\|_{\mathscr{A}^{t,\mu,s_0}_{p,q}(\R^n,w;X)}.
\end{align*}
Using the equivalence \eqref{DBVP:eq:par-dep;equiv_norm_parameter_expl;domain} we obtain
\[
 \|u\|_{\mathscr{A}^{t+2m}_{p,q}(\R^n,w;X)}+\langle\mu\rangle^{t+2m-s_0}\|u\|_{\mathscr{A}^{s_0}_{p,q}(\R^n,w;X)}\eqsim \|f\|_{\mathscr{A}^{t}_{p,q}(\R^n,w;X)}+ \langle\mu\rangle^{t-s_0}\|f\|_{\mathscr{A}^{s_0}_{p,q}(\R^n,w;X)}.
\]
Replacing $\mu^{2m}$ by $\lambda$ again yields the assertion.
\end{proof}

\begin{lemma}\label{Boutet:lemma:parameter-elliptic-whole-space-bessel}
  Let $X$ be a UMD Banach space, $p,q\in(1,\infty)$, $w \in A_{p}(\R^{n})$, $s_0,t\in\R$ with $t\geq s_0$.
 Suppose that $\mathcal{A}(D)$ is a homogeneous differential operator of order $2m$ with constant coefficients in $\mathcal{B}(X)$ satisfying $(E)_\phi$ for some $\phi\in(0,\pi]$. Given $f\in H^{s_0}_{p}(\R^n,w;X)$ let $u:=(\lambda+\mathcal{A}(D))^{-1}f$. Then, for all $\lambda_0>0$ we have the estimate
 \[
  \|u \|_{H^{t+2m}_{p}(\R^n,w;X)}+|\lambda|^{\frac{2m+t-s_0}{2m}}\|u \|_{H^{s_0}_{p}(\R^n,w;X)}\eqsim_{\lambda_0} \|f \|_{H^{t}_{p}(\R^n,w;X)}+ |\lambda|^{\frac{t-s_0}{2m}}\|f \|_{H^{s_0}_{p}(\R^n,w;X)}\quad(\lambda\in(\lambda_0+\Sigma_\phi)).
 \]
\end{lemma}
\begin{proof}
 The proof is almost the same as the one of Lemma \ref{Boutet:lemma:parameter-elliptic-whole-space}. But instead of \eqref{Boutet:eq:prelim:FM_B&F} we use Proposition \ref{Boutet:prop:prelim:operator-Mihklin} together with Lemma \ref{Boutet:lemma:R-boundedness_for_the_symbol}.
\end{proof}

\begin{proof}[Proof of Theorem \ref{Boutet:thm:EBVP_par-dep}]
 First, we consider case \eqref{Boutet:it:elliptic:BF}. By localization, we only have to treat the case of a homogeneous system with constant $\mathcal{B}(X)$-valued coefficients on $\mathscr{O}=\R^n_+$, see the comments in Appendix \ref{Appendix_Localization}. Taking Rychkov's extension operator $\mathscr{E}$ (see Theorem~\ref{DBVP:thm:Rychkov's_extension_operator}), we can represent the solution as
 \[
  u=r_+(\lambda-\mathcal{A}(D))^{-1}_{\R^n}\mathscr{E}f+\sum_{j=1}^m\operatorname{OPK}(\tilde{k}_{j,\lambda})(g_j-\mathcal{B}_j(D)(\lambda+\mathcal{A}(D))^{-1}_{\R^n}\mathscr{E}f).
 \]
Here, $(\lambda-\mathcal{A}(D))^{-1}_{\R^n}$ denotes the resolvent in the whole space as in Lemma \ref{Boutet:lemma:parameter-elliptic-whole-space} and $\tilde{k}_{j,\lambda}$ are the Poisson symbol kernels as in Proposition \ref{Boutet:prop:Poisson_operator_DBVP}. For the estimate, we treat the summands separately. We write
\begin{align*}
 u_1&:=r_+(\lambda-\mathcal{A}(D))^{-1}_{\R^n}\mathscr{E}f,\\
 u_{2,j}&:=\operatorname{OPK}(\tilde{k}_{j,\lambda})\mathcal{B}_j(D)(\lambda+\mathcal{A}(D))^{-1}_{\R^n}\mathscr{E}f,\\
 u_{3,j}&:=\operatorname{OPK}(\tilde{k}_{j,\lambda})g_j.
\end{align*}
First, by Theorem \ref{DBVP:thm:Rychkov's_extension_operator} and Lemma \ref{Boutet:lemma:parameter-elliptic-whole-space} we have that
\begin{align}\begin{aligned}\label{Boutet:eq:Parameter_u1}
 \norm{u_1}_{\mathscr{A}^{t+2m}_{p,q,\gamma}(\R^{n}_+;X)} + |\lambda|^{\frac{t+2m-s_{0}}{2m}}\norm{u_1}_{\mathscr{A}^{s_{0}}_{p,q,\gamma}(\R^{n}_+;X)} &\eqsim \norm{\mathscr{E}f}_{\mathscr{A}^{t}_{p,q}(\R^{n},w_{\gamma};X)} + |\lambda|^{\frac{t-s_{0}}{2m}}\norm{\mathscr{E}f}_{\mathscr{A}^{s_{0}}_{p,q}(\R^{n},w_{\gamma};X)}\\
 &\eqsim \norm{f}_{\mathscr{A}^{t}_{p,q,\gamma}(\R^{n}_+;X)} + |\lambda|^{\frac{t-s_{0}}{2m}}\norm{f}_{\mathscr{A}^{s_{0}}_{p,q,\gamma}(\R^{n}_+;X)}.
 \end{aligned}
\end{align}
For $u_2$ we substitute $\lambda=\mu^{2m}$ again. Then, Theorem \ref{Boutet:thm:Poisson_par-dep_mapping_prop}, Proposition \ref{DBVP:prop:par-dep_trace}, Lemma \ref{Boutet:lemma:parameter-elliptic-whole-space} and Theorem \ref{DBVP:thm:Rychkov's_extension_operator} yield
\begin{align*}
 \|u_{2,j}\|_{\mathscr{A}^{t+2m,|\mu|,s_0}_{p,q,\gamma}(\R^{n}_+;X)}&\lesssim \|\mathcal{B}_j(D)(\lambda+\mathcal{A}(D))^{-1}_{\R^n}\mathscr{E}f\|_{\partial\mathscr{A}^{t+2m-m_j,|\mu|}_{p,q,\gamma}(\R^{n}_+;X)}\\
 &\lesssim \|(\lambda+\mathcal{A}(D))^{-1}_{\R^n}\mathscr{E}f\|_{\mathscr{A}^{t+2m,|\mu|,s_0+m_j}_{p,q}(\R^{n},w_{\gamma};X)}\\
 &\lesssim \|f\|_{\mathscr{A}^{t,|\mu|,s_0+m_j}_{p,q,\gamma}(\R^{n}_+;X)}\\
 &\eqsim \| f \|_{\mathscr{A}^{t}_{p,q,\gamma}(\R^{n}_+;X)}+\langle\mu\rangle^{t-s_0-m_j}\| f \|_{\mathscr{A}^{s_0+m_j}_{p,q,\gamma}(\R^{n}_+;X)}\\
 &\lesssim \| f \|_{\mathscr{A}^{t}_{p,q,\gamma}(\R^{n}_+;X)}+\langle\mu\rangle^{t-s_0}\| f \|_{\mathscr{A}^{s_0}_{p,q,\gamma}(\R^{n}_+;X)}.
\end{align*}
Substituting $\lambda=\mu^{2m}$ again yields
\[
 \norm{u_{2,j}}_{\mathscr{A}^{t+2m}_{p,q,\gamma}(\R^{n}_+;X)} + |\lambda|^{\frac{t+2m-s_{0}}{2m}}\norm{u_{2,j}}_{\mathscr{A}^{s_{0}}_{p,q,\gamma}(\R^{n}_+;X)} \lesssim_{\lambda_0} \norm{f}_{\mathscr{A}^{t}_{p,q,\gamma}(\R^{n}_+;X)} + |\lambda|^{\frac{t-s_{0}}{2m}}\norm{f}_{\mathscr{A}^{s_{0}}_{p,q,\gamma}(\R^{n}_+;X)}.
\]
Finally, it follows from Theorem \ref{Boutet:thm:Poisson_par-dep_mapping_prop} that
\begin{align*}
 \|u_{3,j}\|_{\mathscr{A}^{t+2m,|\mu|,s_0}_{p,q,\gamma}(\R^{n}_+;X)}\lesssim \|g_j\|_{\partial\mathscr{A}^{t+2m-m_j,|\mu|}_{p,q,\gamma}(\R^{n}_+;X)}
\end{align*}
so that \eqref{DBVP:eq:par-dep_Besov;sum} together with the substitution $\lambda=\mu^{2m}$ yields
\begin{align*}
\norm{u_{3,j}}_{\mathscr{A}^{t+2m}_{p,q,\gamma}(\mathscr{O};X)} + |\lambda|^{\frac{t+2m-s_{0}}{2m}}\norm{u_{3,j}}_{\mathscr{A}^{s_{0}}_{p,q,\gamma}(\mathscr{O};X)}
&\lesssim_{\lambda_0}\norm{g_{j}}_{\partial\mathscr{A}^{t+2m-m_{j}}_{p,q,\gamma}(\partial\mathscr{O};X)} + |\lambda|^{\frac{t+2m-m_{j}-\frac{1+\gamma}{p}}{2m}}\norm{g_{j}}_{L_{p}(\partial\mathscr{O};X)} .
\end{align*}
Summing up $u=u_1\sum_{j=1}^m u_{2,j}+u_{3,j}$ yields
\begin{align*}
\norm{u}_{\mathscr{A}^{t+2m}_{p,q,\gamma}(\R^{n}_{+};X)} + |\lambda|^{\frac{t+2m-s_{0}}{2m}}\norm{u}_{\mathscr{A}^{s_{0}}_{p,q,\gamma}(\R^{n}_{+};X)}
&\lesssim \:
\norm{f}_{\mathscr{A}^{t}_{p,q,\gamma}(\R^{n}_{+};X)} + |\lambda|^{\frac{t-s_{0}}{2m}}\norm{f}_{\mathscr{A}^{s_{0}}_{p,q,\gamma}(\R^{n}_{+};X)} \\
&\quad + \sum_{j=1}^{m}\left(\norm{g_{j}}_{\partial\mathscr{A}^{t+2m-m_{j}}_{p,q,\gamma}(\partial\R^{n}_{+};X)} + |\lambda|^{\frac{t+2m-m_{j}-\frac{1+\gamma}{p}}{2m}}\norm{g_{j}}_{L_{p}(\partial\R^{n}_{+};X)} \right).
\end{align*}
The inverse estimate follows from Proposition \ref{DBVP:prop:par-dep_trace} together with the estimate
\[
 \|(\lambda+\mathcal{A}(D))u\|_{\mathscr{A}^{s,\mu,s_0}_{p,q,\gamma}(\R^{n}_{+},X)}\lesssim_{\lambda_0} \|u\|_{\mathscr{A}^{s+2m,\mu,s_0}_{p,q,\gamma}(\R^{n}_{+},X)}.
\]

 Case \eqref{Boutet:it:elliptic:H} can be carried out in almost the exact same way. One just has to use the extension operator from Proposition~\ref{prop:extensionH} instead of Rychkov's extenstion operator,
 use Lemma \ref{Boutet:lemma:parameter-elliptic-whole-space-bessel} instead of Lemma \ref{Boutet:lemma:parameter-elliptic-whole-space} and use the elementary embedding $F^{s}_{p,1,\gamma}(\R^n,X)\hookrightarrow H^{s}_{p,\gamma}(\R^n,E)$ for the Poisson operator estimates.
\end{proof}

\subsection{Operator Theoretic Results}
The $L_{q}$-maximal regularity established in Theorem~\ref{Boutet:theorem:DPIBVP} for the special case of homogeneous initial-boundary data gives $L_{q}$-maximal regularity and thus $R$-sectoriality for the realizations of the corresponding elliptic differential operators:

\begin{corollary}\label{Boutet:cor:R-sect}
Let $\mathscr{O}$ be either $\R^{n}_{+}$ or a $C^{N}$-domain in $\R^{n}$ with a compact boundary $\partial \mathscr{O}$, where $N \in \N$. Let $X$ be a UMD Banach space and let $(\mathcal{A}(D),\mathcal{B}_{1}(D),\ldots,\mathcal{B}_{m}(D))$ be a $\mathcal{B}(X)$-valued differential boundary value system on $\mathscr{O}$ as considered in Section~\ref{Boutet:subsec:sec:prelim:E&LS} and put $m^{*} := \max\{m_{1},\ldots,m_{m}\}$.
Let $\E$ and $\E^{2m}$ be given as in Case \eqref{Boutet:it:parabolic:H}, Case \eqref{Boutet:it:parabolic:F} or Case \eqref{Boutet:it:parabolic:B} as in Section \ref{Boutet:subsec:sec:parabolic;notation}.
Let  $(\mathcal{A}(D),\mathcal{B}_{1}(D),\ldots,\mathcal{B}_{m}(D))$ be a $\mathcal{B}(X)$-valued differential boundary value system of order $2m$ on $\dom$ that satisfies
$(\mathrm{E})_{\phi}$ and $(\mathrm{LS})_{\phi}$ for some $\phi \in (0,\frac{\pi}{2})$. Moreover, we assume that the coefficients satisfy the conditions \textbf{$(\mathrm{SO})_s$}, $(\mathrm{SAP})_s$, $(\mathrm{SAL})_s$, $(\mathrm{SBP})_s$ and $(\mathrm{SBL})_s$ from Section \ref{Boutet:subsec:Elliptic_Smoothness_Coefficients}.\footnote{Here, we identify Case \eqref{Boutet:it:parabolic:H}, Case \eqref{Boutet:it:parabolic:F} and Case \eqref{Boutet:it:parabolic:B}  from Section \ref{Boutet:subsec:sec:parabolic;notation} with Case \eqref{Boutet:it:elliptic:H}, the Triebel-Lizorkin version of Case \eqref{Boutet:it:elliptic:BF} and the Besov version of Case \eqref{Boutet:it:elliptic:BF} from Section \ref{Boutet:subsec:Elliptic_Smoothness_Coefficients}, respectively.} Let $A_{B}$ be the realization of $\mathcal{A}(D)$ in $\E$ with domain $D(A_{B}) = \{u \in \E^{2m} : \mathcal{B}(D)u=0\}$.
For every $\theta \in (\phi,\frac{\pi}{2})$ there exists $\mu_{\theta} > 0$ such that $\mu_{\theta}+A$ is $R$-sectorial with angle $\omega_{R}(\mu_{\theta}+A_{B}) \leq \theta$.
\end{corollary}
\begin{proof}
    This is a direct consequence of Theorem \ref{Boutet:theorem:DPIBVP} and Proposition \ref{Boutet:prelim:MaxReg_R-Sec}. Indeed, if we write $\mathcal{A}(x,D)=\sum_{|\alpha|\leq2m} a_{\alpha}(x)D^{\alpha}$ and $\mathcal{B}_{j}(x,D)=\sum_{|\beta|\leq m_j} \operatorname{tr}_{\partial\mathscr{O}}b_{j,\beta}(x)D^{\beta}$, then it follows from our assumption that $\tilde{a}_{\alpha}:=a_{\alpha}\otimes1_{J}$ and $\tilde{b}_{j,\beta}:=b_{j,\beta}\otimes1_J$ satisfy the conditions $(\mathrm{SAP})$, $(\mathrm{SAL})$, $(\mathrm{SBP})$ and $(\mathrm{SBL})$ from Section \ref{Boutet:subsec:sec:parabolic;notation}.
\end{proof}

The following result is an immediate corollary to Theorem \ref{Boutet:thm:EBVP_par-dep}.

\begin{corollary}\label{Boutet:cor:thm:EBVP_par-dep;sectoriality}
Consider the situation of Theorem \ref{Boutet:thm:EBVP_par-dep} with $s=s_0=s_1$. Let $A_{B}$ be the realization of $\mathcal{A}(D)$ in $\F^s$ with domain $D(A_{B}) = \{u \in \F^{s+2m} : \mathcal{B}(D)u=0\}$.
For every $\theta \in (\phi,\pi)$ there exists $\mu_{\theta} > 0$ such that $\mu_{\theta}+A_{B}$ is sectorial with angle $\phi_{\mu_{\theta}+A} \leq \theta$.
\end{corollary}

\begin{remark}\label{Boutet:rmk:thm:R-sect;functional-calc}
From the $R$-sectoriality and sectoriality in Corollary~\ref{Boutet:cor:R-sect} and Corollary~\ref{Boutet:cor:thm:EBVP_par-dep;sectoriality}, respectively, one could derive boundedness of the $H^{\infty}$-functional calculus using interpolation techniques from \cite{Dore1999,Kalton&Kunstmann&Weis2006}:
\cite[Corollary~7.8]{Kalton&Kunstmann&Weis2006} and \cite[Theorem~3.1]{Dore1999} give a bounded $H^{\infty}$-calculus of the part of $A_{B}$ in the Rademacher interpolation space $\langle \E,D(A) \rangle_{\theta}$ and in the real interpolation space $(\E,D(A))_{\theta,q}$, respectively.
In this way one could improve the $R$-sectoriality to a bounded $H^{\infty}$-functional calculus in Corollary~\ref{Boutet:cor:R-sect} and the sectoriality to a bounded $H^{\infty}$-functional calculus in the $B$- and $\mathcal{B}$-cases in Corollary~\ref{Boutet:cor:thm:EBVP_par-dep;sectoriality}.
We expect that this way, one should be able to obtain a bounded $H^{\infty}$-calculus in the $B$- and $\mathcal{B}$-cases, as there are results on interpolation with boundary conditions also in the vector-valued case, see \cite[Chapter VIII.2, Theorem 2.4.4]{Amann_2019}. If they extend to the weighted setting, then the boundedness of the $H^{\infty}$-calculus in the $B$- and $\mathcal{B}$-cases can be derived. 
\end{remark}

\begin{remark}\label{Boutet:rmk:cor:thm:EBVP_par-dep;dual_scales}
The scales of weighted $\mathcal{B}$- and $\mathcal{F}$-spaces, the dual scales to the scales of weighted $B$- and $F$-spaces, naturally appear in duality theory.
In \cite{Lindemulder2019_DSOE} they were used to describe the adjoint operators for realizations of second order elliptic operators subject to the Dirichlet boundary condition in weighted $B$- and $F$-spaces (see \cite[Remark~9.13]{Lindemulder2019_DSOE}), which was an important ingredient in the application to the heat equation with multiplicative noise of Dirichlet type at the boundary in weighted $L_{p}$-spaces in \cite{Lindemulder&Veraar_boundary_noise} through the so-called Dirichlet map (see \cite[Theorem~1.2]{Lindemulder2019_DSOE}).
The incorporation of the scales of weighted $\mathcal{B}$- and $\mathcal{F}$-spaces in Theorem~\ref{Boutet:thm:EBVP_par-dep} and  Corollary~\ref{Boutet:cor:thm:EBVP_par-dep;sectoriality}
would allow us similarly to describe the adjoint of the operator $A_{B}$ from Corollary~\ref{Boutet:cor:R-sect}, which could then be used to extend \cite{Lindemulder&Veraar_boundary_noise} to more general parabolic boundary value problems with multiplicative noise at the boundary.
\end{remark}
%

\appendix
  \renewcommand{\theequation}{\thesection-\arabic{equation}}
\section{A Weighted Version of a Theorem due to Cl\'ement and Pr\"uss}\label{DBVP:sec:appendix:CP2001}

The following theorem is a weighted version of a result from \cite{Clement&Pruess2001} (see \cite[Theorem~5.3.15]{Hytonen&Neerven&Veraar&Weis2016_Analyis_in_Banach_Spaces_I}).
For its statement we need some notation that we first introduce.

Let $X$ be a Banach space. We write $\widehat{C^{\infty}_{c}}(\R^{n};X) := \mathscr{F}^{-1}C^{\infty}_{c}(\R^{n};X)$ and $\widehat{L^{1}}(\R^{n};X) := \mathscr{F}^{-1}L^{1}(\R^{n};X)$. Then
\[
L_{1,\mathrm{loc}}(\R^{n};\mathcal{B}(X)) \times \widehat{C^{\infty}_{c}}(\R^{n};X) \longra \widehat{L^{1}}(\R^{n};X),\,(m,f) \mapsto \mathcal{F}^{-1}[m\hat{f}] =: T_{m}f.
\]
For $p \in (1,\infty)$ and $w \in A_{p}(\R^{n})$ we define $\mathfrak{M}L_{p}(\R^{n},w;X)$ as the space of all $m \in L_{1,\mathrm{loc}}(\R^{n};\mathcal{B}(X))$ for which $T_{m}$ extends to a bounded linear operator on $L_{p}(\R^{n},w;X)$, equipped with the norm
\[
\norm{m}_{\mathfrak{M}L_{p}(\R^{n},w;X)} := \norm{T_{m}}_{\mathcal{B}(L_{p}(\R^{n},w;X))}.
\]

\begin{theorem}\label{DBVP:thm:appendix:Clement&Pruess2001}
Let $X$ be a Banach space, $p \in (1,\infty)$ and $w \in A_{p}(\R^{n})$.
For all $m \in \mathfrak{M}L_{p}(\R^{n},w;X)$ it holds that
\[
\{ m(\xi) : \text{$\xi$ is a Lebesgue point of $m$} \}
\]
is $R$-bounded with
\begin{align*}
\norm{m}_{L_{\infty}(\R^{n};\mathcal{B}(X))} &\leq \mathcal{R}_{p}\left(\{ m(\xi) : \text{$\xi$ is a Lebesgue point of $m$} \}\right) \lesssim_{p,w} \norm{m}_{\mathfrak{M}L_{p}(\R^{n},w;X)}.
\end{align*}
\end{theorem}
\begin{proof}
This can be shown as in \cite[Theorem~5.3.15]{Hytonen&Neerven&Veraar&Weis2016_Analyis_in_Banach_Spaces_I}.
Let us comment on some modifications that have to be made for the second estimate.
Modifying the H\"older argument given there according to \eqref{DSOP:eq:subsec:prelim:weights;pairing;mixed-norm}, the implicit constant $C_{p,w}$ of interest can be estimated by
\[
C_{p,w} \leq \liminf_{\epsilon \to 0}\epsilon^{d}\norm{\phi(\epsilon\,\cdot\,)}_{L_{p}(\R^{n},w)}
\norm{\psi(\epsilon\,\cdot\,)}_{L_{p'}(\R^{n},w'_{p})},
\]
where $\phi,\psi \in \mathcal{S}(\R^{n})$ are such that $\hat{\phi}$, $\check{\psi}$ are compactly supported with the property that $\int \hat{\phi} \check{\psi} d\xi = 1$.
By a change of variable,
\[
\epsilon^{d}\norm{\phi(\epsilon\,\cdot\,)}_{L_{p}(\R^{n},w)}
\norm{\psi(\epsilon\,\cdot\,)}_{L_{p'}(\R^{n},w'_{p})} =
\norm{\phi}_{L_{p}(\R^{n},w(\epsilon\,\cdot\,))}\norm{\psi}_{L_{p'}(\R^{n},w'_{p}(\epsilon\,\cdot\,))}.
\]
Since $\mathcal{S}(\R^{n}) \hookrightarrow L_{p}(\R^{n},w)$ with norm estimate only depending on $n$, $p$ and $[w]_{A_{p}}$ (as a consequence of \cite[Lemma~4.5]{Meyries&Veraar2012_sharp_embedding}) and since the $A_{p}$-characteristic is invariant under scaling, the desired result follows.
\end{proof}

\section{Pointwise Multiplication}\label{Boutet:appendix:pointwise}
\begin{lemma}\label{DBVP:lemma:pointwise_multiplier}
Let $\mathscr{O}$ be either $\R^{d}_{+}$ or a $C^{\infty}$-domain in $\R^{d}$ with a compact boundary $\partial \mathscr{O}$, let $X$ be a Banach space, $U \in \{\R^{d},\mathscr{O}\}$ and let either
\begin{itemize}
\item[(i)] $p \in [1,\infty)$, $q \in [1,\infty]$, $\gamma \in (-1,\infty)$ and $\mathscr{A} \in \{B,F\}$; or
\item[(ii)] $X$ be reflexive, $p,q \in (1,\infty)$, $\gamma \in (-\infty,p-1)$ and $\mathscr{A} \in \{\mathcal{B},\mathcal{F}\}$.
\end{itemize}
Let $s_{0},s_{1} \in \R$ and $\sigma \in \R$ satisfy $\sigma > \sigma_{s_{0},s_{1},p,\gamma}$.
Then for all $m \in B^{\sigma}_{\infty,1}(U;\mathcal{B}(X))$ and $f \in \mathscr{A}^{s_{0} \vee s_{1}}_{p,q}(U,w^{\partial\mathscr{O}}_{\gamma};X)$ there is the estimate
\begin{equation}\label{DBVP:eq:lemma:pointwise_multiplier;estimate}
\norm{mf}_{\mathscr{A}^{s_{1}}_{p,q}(U,w^{\partial\mathscr{O}}_{\gamma};X)}
\lesssim \norm{m}_{L_{\infty}(U;\mathcal{B}(X))+B^{-(s_{0}-s_{1})_{+}}_{\infty,1}(U;\mathcal{B}(X))}
\norm{f}_{\mathscr{A}^{s_{0} \vee s_{1}}_{p,q}(U,w^{\partial\mathscr{O}}_{\gamma};X)}
+ \norm{m}_{B^{\sigma}_{\infty,1}(U;\mathcal{B}(X))}
\norm{f}_{\mathscr{A}^{s_{0}}_{p,q}(U,w^{\partial\mathscr{O}}_{\gamma};X)}.
\end{equation}
\end{lemma}
\begin{proof}
The proof of \cite[Lemma~3.1]{Lindemulder2019_DSOE} carries over verbatim to the $X$-valued setting.
\end{proof}

\begin{remark}\label{DBVP:rmk:lemma:pointwise_multiplier}
In connection to the above lemma, note that
\begin{equation}\label{eq:DBVP:rmk:lemma:pointwise_multiplier;1}
L_{\infty}(U;\mathcal{B}(X)) + B^{-(s_{0}-s_{1})_{+}}_{\infty,1}(U;\mathcal{B}(X)) =
\left\{\begin{array}{ll} L_{\infty}(U;\mathcal{B}(X)),& s_{1} \geq s_{0},\\
B^{s_{1}-s_{0}}_{\infty,1}(U;\mathcal{B}(X)), & s_{1} < s_{0},
\end{array}\right.
\end{equation}
as a consequence of $B^{0}_{\infty,1} \hookrightarrow L_{\infty} \hookrightarrow B^{0}_{\infty,\infty}$ and $B^{s}_{\infty,\infty} \hookrightarrow B^{s-\epsilon}_{\infty,1}$, $s \in \R$, $\epsilon >0$.
Furthermore,
\begin{equation}\label{eq:DBVP:rmk:lemma:pointwise_multiplier;2}
B^{\sigma}_{\infty,1}(U;\mathcal{B}(X)) \hookrightarrow L_{\infty}(U;\mathcal{B}(X)) + B^{-(s_{0}-s_{1})_{+}}_{\infty,1}(U;\mathcal{B}(X))
\end{equation}
as $\sigma > s_{1}-s_{0} \geq -(s_{0}-s_{1})_{+}$.
\end{remark}

\begin{remark}\label{DBVP:rmk:lemma:pointwise_multiplier;general_weight}
Lemma~\ref{DBVP:lemma:pointwise_multiplier} has a version for more general weights: $A_{\infty}$-weights in case (i) and $[A_{\infty}]'_{p}$-weights in case (ii). The condition $\sigma>\sigma_{s_{0},s_{1},p,\gamma}$ then has to be replaced by
\[
\sigma > \max\left\{ \left(\frac{1}{\rho_{w,p}}-1\right)_{+}-s_{0},-\left(\frac{1}{\rho_{w'_{p},p'}}-1\right)_{+}+s_{1} ,s_{1}-s_{0}\right\},
\]
where $\rho_{w,p} := \sup\{ r \in (0,1) : w \in A_{p/r}\}$ with the convention that $\sup\emptyset = \infty$ and $\frac{1}{\infty}=0$.
\end{remark}

\begin{definition}
 Let $(S,\mathscr{A},\mu)$ be a measure space and $X$ a Banach space. Then we define the space $\mathcal{R}L_{\infty}(S;\mathcal{B}(X))$ as the space of all strongly measurable functions $f\colon S\to\mathcal{B}(X)$ such that
\[
 \|f\|_{\mathcal{R}L_{\infty}(S;\mathcal{B}(X))}:=\inf_{g}\mathcal{R}\{g(\omega):\omega\in S\}<\infty
\]
where the infimum is taken over all strongly measurable $g\colon S\to\mathcal{B}(X)$ such that $f=g$ almost everywhere.
\end{definition}

\begin{lemma}\label{DBVP:lemma:pointwise_multiplier_bessel}
Let $\mathscr{O}$ be either $\R^{d}_{+}$ or a $C^{\infty}$-domain in $\R^{d}$ with a compact boundary $\partial \mathscr{O}$, let $X$ be a UMD Banach space, $U \in \{\R^{d},\mathscr{O}\}$, $p\in(1,\infty)$ and $w\in A_p(\R^n)$. Let further $s_0,s_1 \in \R$ and $\sigma \in \R$ satisfy $\sigma > \max\{-s_{0},s_{1},s_{1}-s_{0}\}$.
Then for all $m \in B^{\sigma}_{\infty,1}(U;\mathcal{B}(X))$ and $f \in H^{s_1\vee s_0}_{p}(U,w;X)$ there is the estimate
\begin{equation}\label{DBVP:eq:lemma:pointwise_multiplier;estimate;bessel}
\norm{mf}_{H^{s_1}_{p}(U,w;X)}
\lesssim \norm{m}_{\mathcal{R}L_{\infty}(U;\mathcal{B}(X))}
\norm{f}_{H^{s_1}_{p}(U,w;X)}
+ \norm{m}_{B^{\sigma}_{\infty,1}(U;\mathcal{B}(X))}
\norm{f}_{H^{s_0}_{p}(U,w;X)}.
\end{equation}
\end{lemma}
\begin{proof}
It suffices to consider the case $U=\R^{d}$. We use paraproducts as in
\cite[Section~4.4]{RS1996} and \cite[Section~4.2]{Meyries&Veraar2015_pointwise_multiplication}.

By \cite[Lemma~4.4]{Meyries&Veraar2015_pointwise_multiplication}, the paraproduct $\Pi_{1}:(m,f) \mapsto \Pi_{1}(m,f)$ gives rise to bounded bilinear mapping
\[
\Pi_{1}: \mathcal{R}L_{\infty}(\R^{d};\mathcal{B}(X)) \times H^{s}_{p}(\R^{d},w;X) \longra H^{s}_{p}(\R^{d},w;X).
\]

By a slight modification of \cite[Lemma~4.6]{Meyries&Veraar2015_pointwise_multiplication} (see \cite[Lemma~3.1]{Lindemulder2019_DSOE}), for $i \in \{2,3\}$, the paraproduct $\Pi_{i}:(m,f) \mapsto \Pi_{i}(m,f)$ gives rise to bounded bilinear mapping
\[
\Pi_{i}:  B^{\sigma}_{\infty,1}(\R^{d};\mathcal{B}(X)) \times F^{s_{0}}_{p,\infty}(\R^{d},w;X) \longra F^{s_{1}}_{p,1}(\R^{d},w;X)
\]
and thus a bounded bilinear mapping
\[
\Pi_{i}: B^{\sigma}_{\infty,1}(\R^{d};\mathcal{B}(X)) \times H^{s_{0}}_{p}(\R^{d},w;X) \longra H^{s_{1}}_{p}(\R^{d},w;X). \qedhere
\]
\end{proof}

\begin{proposition}\label{Prop:PointwiseMultiplication}
 Under the conditions of Lemma \ref{DBVP:lemma:pointwise_multiplier} with $s_0\geq s_1$, we have the continuous bilinear mapping
 \begin{equation}\label{DBVP:eq:lemma:pointwise_multiplier;bilinear;BF}
B^{\sigma}_{\infty,1}(U;\mathcal{B}(X)) \times \mathscr{A}^{s_0}_{p,q}(U,w^{\partial\mathscr{O}}_{\gamma};X)
\longra \mathscr{A}^{s_{1}}_{p,q}(U,w^{\partial\mathscr{O}}_{\gamma},X),\,(m,f) \mapsto mf.
\end{equation}
Under the conditions of Lemma \ref{DBVP:lemma:pointwise_multiplier_bessel} with $s_0\geq s_1$, we have the continuous bilinear mapping
 \begin{equation}\label{DBVP:eq:lemma:pointwise_multiplier;bilinear;H}
B^{\sigma}_{\infty,1}(U;\mathcal{B}(X)) \times H^{s_{0}}_{p}(U,w;X)
\longra H^{s_{1}}_{p}(U,w,X),\,(m,f) \mapsto mf.
\end{equation}
\end{proposition}
\begin{proof}
\eqref{DBVP:eq:lemma:pointwise_multiplier;bilinear;BF} is a direct consequence of Lemma~\ref{DBVP:lemma:pointwise_multiplier}. The case $s=0$ in \eqref{DBVP:eq:lemma:pointwise_multiplier;bilinear;H} follows from \cite[Proposition~3.8]{Meyries&Veraar2015_pointwise_multiplication}.
The case $s_0 > s_1$ in \eqref{DBVP:eq:lemma:pointwise_multiplier;bilinear;H} follows from the $A_{p}$-version of \eqref{DBVP:eq:lemma:pointwise_multiplier;bilinear;BF} (see Remark~\ref{DBVP:rmk:lemma:pointwise_multiplier;general_weight}) as $\sigma > \sigma_{s_{0}-\epsilon,s_{1}+\epsilon,p,w}$ for sufficiently small $\epsilon > 0$.
\end{proof}

\section{Comments on the Localization and Perturbation Procedure} \label{Appendix_Localization}

The localization and perturbation arguments are quite technical but standard, let us just say the following.
The localization in Theorem~\ref{Boutet:theorem:DPIBVP} can be carried out as in \cite[Sections 2.3 $\&$ 2.4]{Mey_PHD-thesis} and \cite[Appendix~B]{Lindemulder2017_PIBVP}, where we need to use some of the
pointwise estimates from Appendix~\ref{Boutet:appendix:pointwise} as well some of the localization and rectification results for weighted Besov and Triebel-Lizorkin spaces from \cite[Section~4]{Lindemulder2019_DSOE} (which extend to the vector-valued situation) in order to perform all the arguments. Furthermore, the localization in Theorem \ref{Boutet:thm:EBVP_par-dep} can be carried out as in \cite[Theorem~9.2]{Lindemulder2019_DSOE}. The results in \cite[Section~4]{Lindemulder2019_DSOE} are a generalization of results on the invariance of Besov- and Triebel-Lizorkin spaces under diffeomorphic transformations such as \cite[Theorem 4.16]{Scharf2013} to the weighted anisotropic mixed-norm setting. They lead to the conditions $\mathrm{(SO)}$ in Section~\ref{Boutet:sec:parabolic} and $\mathrm{(SO)}_s$ in Section~\ref{Boutet:sec:elliptic}.

We would also like to mention \cite{Dong&Galarati2018} and \cite{Dong&Galarati2018b}, where the authors treat maximal $L_q$-$L_p$-regularity for parabolic boundary value problems on the half-space in which the elliptic operators have top order coefficients in the VMO class in both time and space variables. In their proofs, they do not use localization for the results on VMO coefficients, but they extend some techniques by Krylov as well as Dong and Kim.\\

While the geometric steps of the localization procedure in our setting are the same as in the standard $L_p$-setting, there are some differences in what kind of perturbation results we need. The main difference lies in the treatment of the top order perturbation of the differential operator on the domain. More precisely, the following lemma is a useful tool in the localization procedure for our setting.

\begin{lemma}\label{lemma:Perturbation}
 Let $E$ be a Banach space and $A\colon E\supset D(A)\to E$ a closed linear operator. Suppose that there is a constant $C>0$ such that for all $\lambda>0$ and all $u\in D(A)$ it holds that
 \begin{align}\label{Appendix:EstimateParameterDependent}
  \| u\|_{D(A)} + \lambda \|u\|_E\leq C \|(\lambda+A) u\|_E.
 \end{align}
Let $\normm{\,\cdot\,}\colon E\to[0,\infty)$ a mapping and $\theta\in(0,1)$ such that
\begin{align}\label{Appendix:Interpolation}
 \normm{u} \leq \|u\|_E^{1-\theta}\|u\|^{\theta}_{D(A)}
\end{align}
holds for all $u\in D(A)$. Let further $P\colon D(A)\to E$ and suppose that there are constants $\delta, C'\in(0,\infty)$ such that
\begin{align}\label{Appendix:InterpolationInequality}
 \|P(u)\|_E\leq \delta\|u\|_{D(A)}+C'\normm{u}
\end{align}
for all $u\in D(A)$. Then there is $\lambda_0\in(0,\infty)$ only depending on $\delta,C'$ and $\theta$ such that for all $\lambda\geq\lambda_0$ and all $u\in D(A)$ we have the estimate
\begin{align}\label{Appendix:Perturbation}
 \|P(u)\|_E\leq 2\delta C\|(\lambda+A)u\|_E.
\end{align}
\end{lemma}
\begin{proof}
 For $u\in D(A)$ we have that
 \begin{align*}
  \|P(u)\|_E&\leq \delta\|u\|_{D(A)}+C'\normm{u} \leq \delta\|u\|_{D(A)}+C' \|u\|_E^{1-\theta}\|u\|^{\theta}_{D(A)} \\
  &\leq 2\delta\|u\|_{D(A)} +\delta C_\delta \|u\|_{E}
 \end{align*}
with $C_{\delta}:=(\frac{\delta}{C'\theta})^{\theta/(1-\theta)}(1-\theta)$. Here, we used Young's inequality with the Peter-Paul trick. Using \eqref{Appendix:EstimateParameterDependent} with $\lambda\geq C_\delta/2$, we can further estimate
\begin{align*}
 \|P(u)\|_E&\leq 2\delta C \| (\lambda+A)u\|_{E}
\end{align*}
so that $\lambda_0=C_{\delta}/2$ is the asserted parameter.
\end{proof}

 If one wants to apply Lemma~\ref{lemma:Perturbation} for a localization procedure, then one can treat the top order perturbation as follows: Suppose that the differential operator has the form $1+\sum_{|\alpha|=2m} (a_{\alpha}+p_{\alpha}(x))D^\alpha$ with $p_{\alpha}$ being small in a certain norm. The mapping $P$ in Lemma~\ref{lemma:Perturbation} can be chosen to be $P(u)(x)=\sum_{|\alpha|=2m}p_{\alpha}(x)D^\alpha u(x)$ and $A$ can be chosen to be the realization of $1+\sum_{|\alpha|=2m}a_{\alpha}D^{\alpha}$ in $\E$ with vanishing boundary conditions. Now one can use Lemma~\ref{DBVP:lemma:pointwise_multiplier} in combination with Remark~\ref{DBVP:rmk:lemma:pointwise_multiplier} in the Besov-Triebel-Lizorkin case and Lemma~\ref{DBVP:lemma:pointwise_multiplier_bessel} in the Bessel potential case to obtain an estimate of the form \eqref{Appendix:InterpolationInequality}. In order to do this for example in the parabolic case, one chooses $\epsilon\in(0,2m)$ such that $\sigma>\sigma_{s,p,\gamma}+\epsilon\geq \sigma_{s-\epsilon,\epsilon,p,\gamma}$. Then one chooses $s_0=s-\epsilon$ and $s_1=s$. These choices lead to the estimate
 \[
  \|p_{\alpha}D^{\alpha}u\|_{\E}\lesssim \|p_{\alpha}\|_{L_{\infty}} \|u\|_{\E^{2m}}+\|p_{\alpha}\|_{BUC^{\sigma}}\|u\|_{\E^{2m-\epsilon}}\quad(|\alpha|=2m)
 \]
in the Besov and Triebel-Lizorkin cases and
 \[
  \|p_{\alpha}D^{\alpha}u\|_{\E}\lesssim \|p_{\alpha}\|_{\mathcal{R}L_{\infty}} \|u\|_{\E^{2m}}+\|p_{\alpha}\|_{BUC^{\sigma}}\|u\|_{\E^{2m-\epsilon}}\quad(|\alpha|=2m)
 \]
in the Bessel potential case. It holds that $\|\cdot\|_{D(A)}\eqsim\|\cdot\|_{\E^{2m}}$ on $D(A)\subset\E^{2m}$. Hence, if $\theta=1-\frac{\epsilon}{2m}$ and $E=\E$, then these estimates would correspond to \eqref{Appendix:InterpolationInequality} in Lemma~\ref{lemma:Perturbation} where $\normm{\,\cdot\,}=M\|\cdot\|_{\E^{2m-\epsilon}}$ for a suitable constant $M>0$ such that \eqref{Appendix:Interpolation} holds. Note that \eqref{Appendix:EstimateParameterDependent} follows from the sectoriality of $A$. Therefore, if $p_{\alpha}$ is small in $L_{\infty}$ or $\mathcal{R}L_{\infty}$-norm, respectively, then Lemma~\ref{lemma:Perturbation} shows that $P$ is just a small perturbation of a suitable shift of the operator $A$.

\subsection*{Acknowledgement.}
The authors would like to thank Mark Veraar for pointing out the Phragmen-Lindel\"of Theorem (see \cite[Corollary 6.4.4]{Conway_1978}) for the proof of Lemma~\ref{DBVP:lemma:prelim:max-reg;imaginary_axis}.
They would also like to thank Robert Denk for useful discussions on the Boutet de Monvel calculus.

\end{document}